%% file: gmst_ArXiv.tex
\documentclass[12pt]{amsart}
\usepackage{amscd,amssymb,mathrsfs}
\usepackage[backref,bookmarks=true,colorlinks=true,
pdfstartview=FitV, linkcolor=blue, citecolor=red,
urlcolor=green]{hyperref}
\usepackage{mathtools} 
\usepackage{graphicx,subfig} 
\begin{document}
\title[Spaces of actions on the hyperbolic plane]
{Automorphisms of two-generator free groups and 
spaces of isometric actions on the hypebolic plane}

\author[Goldman]{William Goldman}
\address{
{\it Goldman:\/} Department of Mathematics \\
University of Maryland \\
College Park, MD 20742, USA}
\author[McShane]{Greg McShane}
\address{
{\it McShane:\/} Department of Mathematics \\
Institut Fourier \\
Grenoble, France}

\author[Stantchev]{George Stantchev}
\address{
{\it Stantchev:\/}
University of Maryland}

\author[Tan]{Ser Peow Tan}
\address{
{\it Tan:\/} Department of Mathematics \\
National University of Singapore \\
2, Science Drive 2, Singapore 117543, Singapore}
\date{\today}

%



\newtheorem{thm}{Theorem}[subsection]
\newtheorem*{thm*}{Theorem}
\newtheorem{lem}[thm]{Lemma}
\newtheorem{conj}[thm]{Conjecture}
\newtheorem{cor}[thm]{Corollary}
\newtheorem{add}[thm]{Addendum}
\newtheorem{prop}[thm]{Proposition}
\newtheorem*{thm4*}{Theorem~4}
\newtheorem*{thm8*}{Theorem~8}
\newtheorem{definition}[thm]{Definition}
\newtheorem{remark}[thm]{Remark}

\setcounter{tocdepth}{3} 

\newcommand{\Ht}{\mathsf{H}^2} 
\newcommand{\Hth}{\mathsf{H}^3} 
\newcommand{\Hom}{\mathsf{Hom}} 
\newcommand{\Fix}{\mathsf{Fix}} 
\newcommand{\R}{\mathbb{R}}   
\newcommand{\C}{\mathbb{C}}
\newcommand{\Z}{\mathbb{Z}}
\newcommand{\Q}{\mathbb{Q}}
\newcommand{\N}{\mathbb{N}}
\newcommand{\Ft}{{\mathsf F}_2} 
\newcommand{\Prim}{{\mathsf{Prim}}(\Ft)} 

\newcommand{\T}{\mathsf{T}}
\newcommand{\E}{\mathsf{Edge}(\T)}
\newcommand{\Vt}{\mathsf{Vert}(\vT)}
\newcommand{\V}{\mathsf{Vert}(\T)}
\newcommand{\vE}{\overrightarrow{{\mathsf{Edge}}(\T)}}
\newcommand{\vT}{\overrightarrow{\T}}
\renewcommand{\O}{\Omega}
\newcommand{\Coloring}{\Phi} 
\newcommand{\kC}{\kappa_\Coloring}
\newcommand{\kCk}{\kC^{-1}(k)}
\newcommand{\kCtwo}{\kC^{-1}(2)}
\newcommand{\kCminusTwo}{\kC^{-1}(-2)}
\renewcommand{\L}{\mathsf{L}}
\newcommand{\BB}{\mathbf{B}} 
\newcommand{\B}{\mathfrak{B}}  
\newcommand{\Fricke}{{\mathfrak F}} 
\newcommand{\hGamma}{\widehat{\Gamma}}
\newcommand{\FO}[1]{\mathfrak{O}(#1)} 
\newcommand{\gFO}[1]{\mathfrak{O}'(#1)} 

\newcommand{\GL}{\mathsf{GL}}
\newcommand{\SL}{\mathsf{SL}}
\newcommand{\PGL}{\mathsf{PGL}}
\newcommand{\PSL}{\mathsf{PSL}}
\newcommand{\GLtZ}{\GL(2,\Z)}
\newcommand{\PGLtZ}{\PGL(2,\Z)}
\newcommand{\PGLtZc}{\PGLtZ_\Coloring}  
\newcommand{\PGLtZt}{\PGLtZ_{(2)}} 
\newcommand{\GLtZc}{\GLtZ_\Coloring}  
\newcommand{\GLtZt}{\GLtZ_{(2)}} 
\newcommand{\SLtR}{\SL(2,\R)}
\newcommand{\SLtRpm}{\SL_{\pm}(2,\R)}
\newcommand{\PSLtR}{\PSL(2,\R)}
\newcommand{\PGLtR}{\PGL(2,\R)}
\newcommand{\GLtR}{\GL(2,\R)}
\newcommand{\PGLtC}{\PGL(2,\C)}
\newcommand{\PSLtC}{\PSL(2,\C)}
\newcommand{\SLtC}{\SL(2,\C)}
\newcommand{\SUt}{\mathsf{SU}(2)}
\newcommand{\pmI}{\{ \pm \Id \}} 
\newcommand{\pmOne}{\{ \pm 1 \}} 
\newcommand{\Center}{\mathsf{Center}} Ä
\newcommand{\Aut}{\mathsf{Aut}}
\newcommand{\Inn}{\mathsf{Inn}}
\newcommand{\Out}{\mathsf{Out}}    
\newcommand{\Ou}{\Out(\Ft)} 
\newcommand{\Au}{\Aut(\Ft)} 
\newcommand{\In}{\Inn(\Ft)} 
\newcommand{\Id}{\mathbb{I}}  
\newcommand{\tr}{\mathsf{tr}}
\newcommand{\Ham}{\mathsf{Ham}}
\renewcommand{\mod}[1]{\big(\mathsf{mod}~#1\big)}  
\newcommand{\Image}{\mathsf{Image}}
\newcommand{\Soo}{\Sigma_{1,1}}  
\newcommand{\Sthz}{\Sigma_{3,0}}  
\newcommand{\Coo}{C_{1,1}}  
\newcommand{\Czt}{C_{0,2}}  
\newcommand{\vedge}[1]{\psi_{#1}} 
\newcommand{\inv}{^{-1}}  
\newcommand{\PG}{\mathsf{Pants}(\Sigma_{1,1})} 
\newcommand{\Po}{\mathsf{P}^1}
\newcommand{\RPo}{\Po(\R)} 
\newcommand{\QPo}{\Po(\Q)} 
\newcommand{\Xx}{\mathfrak{X}}  
\newcommand{\longdash}{{\rule[.05in]{.2in}{.01in}}}
\newcommand{\llongdash}{{\rule[.05in]{.3in}{.01in}}}
\newcommand{\dd}[1]{\partial_{#1}}
\newcommand{\tdd}[1]{{\widetilde\partial}_{#1}}
\newcommand{\Mod}{\mathsf{Mod}}
\newcommand{\Rep}{\mathsf{Rep}}
\newcommand{\Ker}{\mathsf{Ker}}
\newcommand{\Det}{\mathsf{Det}}
\newcommand{\sgn}{\mathsf{sgn}}
\newcommand{\graph}{\mathsf{graph}}
\newcommand{\RepF}{\Rep\big(\Ft,\SLtC\big)}
\newcommand{\HomF}{\Hom\big(\Ft,\SLtC\big)}
\newcommand{\II}{\mathfrak{I}} 
\newcommand{\tII}{\widetilde{\II}} 
\newcommand{\kk}{\mathbf{k}} 
\newcommand{\MM}{\mathbf{M}} 
\newcommand{\QQ}{\mathbf{Q}} 
\newcommand{\PP}{\mathsf{P}} 
\newcommand{\SymThree}{\mathfrak{S}_3}
\newcommand{\scalefig}{0.3} 

\newcommand{\e}{\mathfrak{e}}          
\newcommand{\Inne}{\Inn^\e}   

\newcommand{\normalsubgroup}{\unlhd} 
\newcommand{\area}{\mathsf{area}} 
\newcommand{\Nd}{\mathfrak{N}_{\delta}}

\newcommand{\sltc}{\mathfrak{sl}(2,\C)} 
\newcommand{\Ad}{\mathsf{Ad}} 
\newcommand{\Axis}{\mathsf{Axis}} 
\newcommand{\Pixy}{\Pi_{xy}} 
\newcommand{\EndInv}{\mathcal{E}(\mu)} 

\begin{abstract}
The automorphisms of a two-generator free group $\Ft$ acting on the
space of orientation-preserving isometric actions of $\Ft$ on
hyperbolic 3-space defines a dynamical system.  Those actions which
preserve a hyperbolic plane but not an orientation on that plane is an
invariant subsystem, which reduces to an action of a group $\Gamma$ on
$\R^3$ by polynomial automorphisms preserving the cubic polynomial
\[ \kC(x,y,z) := -x^2 -y^2 + z^2 + x y z -2 \]  
and an area form on the level surfaces $\kCk$.

The Fricke space of marked
hyperbolic structures on the 2-holed projective plane with funnels
or cusps identifies with the subset $\Fricke(C_{0,2})\subset\R^3$
defined by 
\[ z\le -2, \quad  xy + z \ge 2. \]
The generalized Fricke space
of marked hyperbolic structures on the 1-holed Klein bottle with a
funnel, a cusp, or a conical singularity identifies with the subset
$\Fricke'(C_{1,1})\subset\R^3$ defined by 
\[ 
z>2, \quad  x y z \ge x^2 + y^2.
\]
We show that $\Gamma$ acts properly on the subsets
$\Gamma\cdot\Fricke(C_{0,2})$ and
$\Gamma\cdot\Fricke'(C_{1,1})$.
Furthermore for each $k<2$,
the action of $\Gamma$ is ergodic 
on the complement of
$\Gamma\cdot\Fricke(C_{0,2})$ in $\kCk$ for $k < -14$.
In particular, the action is ergodic on all of $\kCk$ for $-14\le k < 2$.

For $k>2$, the orbit $\Gamma\cdot\Fricke(C_{1,1})$ is open and dense in $\kCk$.
We conjecture its complement has measure zero.

\end{abstract}

\maketitle
\newpage
\tableofcontents
\listoffigures
\section{Introduction}

This paper concerns moduli spaces of actions of discrete groups on
hyperbolic space. Spaces of $\PSLtC$-representations of fundamental
groups of surfaces of negative Euler characteristic are natural
objects upon which mapping class groups and related automorphism
groups act. They arise as deformation spaces of (possibly singular) 
locally homogeneous geometric structures on surfaces. 
We relate the interpretation in terms of geometric structures to
the dynamics of automorphism groups, see also \cite{STY}. 

The rank two free group $\Ft$ is the simplest such surface group.  
It arises as the fundamental group of four non-homeomorphic surfaces: 
the three-holed sphere $\Sigma_{(0,3)}$, 
the one-holed torus $\Soo$, 
the two-holed projective plane (or {\em cross-surface\/}) $C_{(0,2)}$, 
and the one-holed Klein bottle $C_{(1,1)}$.  
Of these the first two are orientable and the second two are nonorientable.
Furthermore $\Soo$ enjoys the remarkable property that every
automorphism of $\pi_1(\Soo)$ is induced by a homeomorphism
$\Soo\longrightarrow\Soo$. Equivalently, every homotopy-equivalence
$\Soo\longrightarrow\Soo$ is homotopic to a homeomorphism.
For concreteness, we choose free generators $X$ and $Y$ for $\Ft$.  
When $\Ft$ arises as the fundamental group 
$\pi_1(\Sigma)$ of a nonorientable surface $\Sigma$, 
we require that $X$ and $Y$ correspond to orientation-reversing simple closed curves.
This defines a homomorphism 
\begin{equation}\Ft\xrightarrow{~\Phi~}\{\pm 1\}\end{equation}\label{eq:Coloring}
where
$\Phi(X) = \Phi(Y) = -1$. 
Then a loop $\gamma$ preserves orientation if and only if $\Phi(\gamma) = +1$ 
and reverses orientation if and only if $\Phi(\gamma) = -1$. 

$\Ft$ possesses a rich and mysterious representation theory into the group 
$\PSLtC$ of orientation-preserving isometries of $\Hth$. 
Such representations define  orientation-preserving isometric
actions on $\Hth$, and intimately relate to geometric
structures on surfaces and $3$-manifolds having fundamental group
isomorphic to $\Ft$. 

The dynamical system arises from an action of the
{\em outer automorphism group\/} $\Ou := \Au/\In$.  
The action on the abelianization $\Z^2$ of $\Ft$ defines an isomorphism
$\Ou \cong \GLtZ$.  
The group $\Ou$ acts on the quotient 
$\RepF$ by conjugation;
the quotient space consists of equivalence classes of orientation-preserving isometric actions of $\Ft$ on 
hyperbolic $3$-space $\Hth$. 
By an old result of Vogt~\cite{MR1508833}   (sometimes also attributed to 
Fricke~\cite{MR0183872}),  
the quotient variety $\RepF$ is isomorphic to $\C^3$,
where the coordinates $\xi,\eta,\zeta\in\C$ correspond to the traces of matrices representing
$\rho(X), \rho(Y), \rho(XY)$ respectively.
We define a group $\hGamma$, arising from $\Ou$, which acts {\em effectively\/} and
{\em polynomially\/}  on $\C^3$,
preserving the cubic polynomial
\[
\kappa(\xi,\eta,\zeta) := \xi^2 + \eta^2 + \zeta^2 - \xi\eta\zeta -2\]
corresponding  to the trace of the commutator $[\rho(X),\rho(Y)]$.
This group is commensurable to $\GLtZ$.
See \S\ref{sec:FaithfulAction} for details.

%
%

This paper concerns isometric actions of $\Ft$ on the hyperbolic plane $\Ht\subset\Hth$.
Isometric actions on $\Ht$ which preserve orientation 
correspond to representations into $\PSLtR\subset\PSLtC$,
and are discussed in detail in Goldman~\cite{MR2026539}. 
In particular they form a moduli space, parametrized by a quotient space of a subset of 
$\R^3\subset\C^3$, upon which $\Ou$ acts, preserving the polynomial $\kappa$,
and \cite{MR2026539} 
gives the dynamical decomposition of the $\hGamma$-action  on the level sets of $\kappa$.
(See also \cite{MR2264541}.)

The present paper concerns orientation-preserving isometric actions on $\Hth$ which preserve $\Ht$, 
but do {\em not\/}  preserve the orientation on $\Ht$. 
Important cases of such actions arise from hyperbolic structures on the two nonorientable surfaces
$C_{(0,2)}$ and $C_{(1,1)}$.  
As above, we assume that the generators $X,Y$ of their fundamental groups are represented by 
orientation-reversing isometries of $\Ht$,
and the effect on orientation is recorded in the homomorphism $\Phi$ defined in \eqref{eq:Coloring},
which may be interpreted as a Stiefel-Whitney class. 

Explicitly they correspond to representations $\rho$ into $\PSLtC$ where
$\rho(X)$ and $\rho(Y)$ are represented by {\em purely imaginary\/} matrices.
This corresponds the case when the trace coordinates $\xi,\eta$ are {\em purely imaginary,\/}
that is, the moduli space is another {\em real form\/}  $i\R\times i\R \times\R \subset \C^3$

The effect on orientation of $\Ht$ is extra information which breaks the symmetry of $\hGamma$. 
This leads to a sligntly smaller subgroup $\Gamma\subset\hGamma$, also commensurable to $\GLtZ$,
which preserves this extra information. 
See
\S\ref{sec:RankTwoFreeGroupAndAutos} 
for a detailed description of $\Gamma$ and its action on 
\[ \R^3 \cong i\R\times i\R \times\R \subset \C^3.  \] 
For now, define $\Gamma$ as the group of automorphisms generated by the three {\em Vieta involutions\/} 
\[
\bmatrix \xi \\ \eta \\ \zeta  \endbmatrix \xmapsto{~\II_1~}
\bmatrix \eta\zeta-\xi \\ \eta \\ \zeta  \endbmatrix,\qquad 
\bmatrix \xi \\ \eta \\ \zeta  \endbmatrix \xmapsto{~\II_2~}
\bmatrix \xi \\ \xi\zeta-\eta \\ \zeta  \endbmatrix, \qquad
\bmatrix \xi \\ \eta \\ \zeta  \endbmatrix \xmapsto{~\II_3~}
\bmatrix \xi \\ \eta \\ \xi\eta -\zeta  \endbmatrix,
\]
the three {\em sign-changes\/}
\[
\bmatrix \xi \\ \eta \\ \zeta \endbmatrix  \xmapsto{~\sigma_1~} \bmatrix \xi \\  -\eta \\ -\zeta \endbmatrix, \qquad
\bmatrix \xi \\ \eta \\ \zeta \endbmatrix  \xmapsto{~\sigma_2~}\bmatrix -\xi \\  \eta \\ -\zeta \endbmatrix, \qquad
\bmatrix \xi \\ \eta \\ \zeta \endbmatrix \xmapsto{~\sigma_3~}\bmatrix -\xi \\  -\eta \\ \zeta \endbmatrix \]
and coordinate permutations preserving $\Phi$.

\subsection{Fricke spaces and Fricke orbits}

Representations 
\[ \Ft\xrightarrow{\rho}\PGLtR \] 
arise naturally
from hyperbolic surfaces with fundamental group isomorphic to $\Ft$.
If $S$ is such a hyperbolic surface, then the composition of
an isomorphism  $\Ft\cong\pi_1(S)$ with the holonomy representation
$\pi_1(S)\hookrightarrow\PGLtR$  a representation $\rho$.
Such representations are characterized as those which are injective with
discrete image, and we call such representations {\em discrete embeddings.\/}
Moreover $\Ou$ acts properly
on the space of equivalence classes of discrete embeddings in $\PGLtR$.

If $\rho$ is a discrete embedding, then the hyperbolic surface 
\[ S := \Ht/\Image(\rho) \] 
has fundamental group $\Image(\rho)\cong \Ft$. 
If $S$ is orientable, then its holonomy representation maps into the group
$\PSLtR$ of {\em orientation-preserving\/} isometries of $\Ht$. 
The corresponding homomorphism 
$\Coloring$ is trivial and is treated in 
\cite{MR2026539}. 
This paper concerns the case when $S$ is nonorientable and $\Coloring$ is nontrivial. 

When $S$ is nonorientable, it is homeomorphic to either a
two-holed cross-surface $C_{(2,0)}$
or a one-holed Klein bottle $C_{(1,1)}.$ 
(See, for example, Norbury~\cite{MR2399656} for a lucid description of
these surfaces.) Let $S$ be one of these two surfaces.
The {\em Fricke space\/} of $S$
is defined as the space of marked complete
hyperbolic structures on $\mathsf{int}(S)$.
Here the  ends are either cusps or {\em funnels,\/}
ends of infinite area bounded by closed geodesics.
A homotopy equivalence $\Soo\longrightarrow M$ to a hyperbolic surface 
$M\approx S$ induces a discrete embedding
$\Ft \longrightarrow \PGLtR$
which we assume determines the invariant $\Coloring$.
The set $\FO{S}$ 
of corresponding characters admits an action of $\Gamma$
and is a disjoint union of the copies of the Fricke space $\Fricke(S)$ of $S$.
The mapping class group $\Mod(S)$ acts on $\Fricke(S)$, 
and the components of $\FO{S}$ bijectively correspond to the cosets
$\Mod(S)\backslash\Gamma$.
We call $\FO{S}$ the {\em Fricke orbit\/} of $S$.

The {\em generalized Fricke space\/} $\Fricke'(S)$ of $S$ 
includes hyperbolic structures $S$
where some of the boundary components are replaced by complements of
conical singularities. 
Define the {\em generalized Fricke orbit\/} $\gFO{S}$ as the orbit of
the generalized Fricke space $\Fricke'(S)$.

Some of the components of the $\FO{S}$ have a particularly
simple description, and correspond to the Fricke space $\Fricke(S)$.
In his dissertation~\cite{MR2705402}, Stantchev
computed the Fricke spaces of these surfaces and proved partial results
on the $\Gamma$-action on the level sets of $\kappa$.
This paper presents and extends Stantchev's results.
(Compare also \cite{GoldmanStantchev}.)

If $S\approx C_{(0,2)}$,  take $X$ and $Y$
to be  orientation-reversing simple loops with one intersection, such that $\partial S$
consists of curves in the homotopy classes $Z := XY$ and $Z' := XY^{-1}$.
Then the Fricke space $\Fricke(C_{0,2})$ identifies with the
subset
\[
\Fricke(C_{0,2}) := \{ (x,y,z)\in\R^3 \mid z \le -2, \; x y + z \ge 2\}
\]
with boundary traces $z$ and 
\[ z' := - x y - z. \]
Here $X,Y$ are represented by purely imaginary matrices in $\SLtC$,
with respective traces $ix, iy\in i\R$.

If $S\approx \Coo$, then, taking $X$ and $Y$
to be disjoint orientation-reversing simple loops, 
$\partial S$ is represented by a curve
in the homotopy class $X^2Y^2$, and the boundary 
trace is 
\[ \delta := x^2 + y^2 - x y z + 2. \]
The Fricke space of $C_{1,1}$ identifies 
with
\[ \Fricke(C_{1,1}) := \{ (x,y,z)\in\R^3 \mid z < -2, \; x y  z > x^2 + y^2 + 4\}. \]
The generalized Fricke space
$\Fricke'(\Coo)$ comprises hyperbolic structures
on the one-holed Klein bottle $\Coo$ 
with a boundary component or  a cusp,  or on the Klein bottle $C_{1,0}$
with a {\em conical singularity\/} of cone angle $0 < \theta < 2\pi$.
It identifies with
\[ \Fricke'(C_{1,1}) := \{ (x,y,z)\in\R^3 \mid z < -2, \; x y  z > x^2 + y^2\}.\]
where the boundary trace satisfies
$\delta =  -2\cos(\theta/2)$
for $-2 < \delta < 2$. 
The generalized Fricke orbit $\gFO{\Coo}$ is 
the union of all $\Gamma$-translates of the generalized Fricke space
$\Fricke'(C_{1,1})$.

\subsection{The orientation-preserving case}
Representations $\rho$ with 
$$
\xi,\eta,\zeta \le -2
$$ correspond to hyperbolic structures on $\Sigma_{0,3}$ whose ends are either
collars about closed geodesics or cusps. In particular the
{\em Fricke space\/} $\Fricke(\Sigma_{0,3})$ of 
$\Sigma_{0,3}$ identifies with
\begin{equation*}
\Fricke(\Sigma_{0,3}) := \big(\infty, -2]^3 \subset \R^3 \subset \C^3.
\end{equation*}
Its stabilizer in $\hGamma$ 
consists of the symmetric group
\begin{equation*}
\SymThree\subset \hGamma.
\end{equation*}
Furthermore, if
$\phi\in\hGamma\setminus\SymThree$, then
\begin{equation*}
\Fricke(\Sigma_{0,3}) \,\cap\, \phi\Fricke(\Sigma_{0,3}) \;=\; \emptyset.
\end{equation*}
In particular $\hGamma$ acts properly on
$\hGamma\Fricke(\Sigma_{0,3})$.

Equivalence classes of representations $\rho$ with 
\[ \xi > 2,\  \eta > 2,\  \zeta > 2 \]
and $\kappa(\xi,\eta,\zeta) \le -2$ form the Fricke space
$\Fricke(\Sigma_{1,1})$ corresponding to marked complete hyperbolic structures on the one-holed torus.
The ends are either cusps or {\em funnels,\/} 
that is, collars about closed geodesics or cusps.

More generally, when $-2 < \kappa(\xi,\eta,\zeta)< 2$ and $(\xi,\eta,\zeta)\in [2, \infty)^3$, the representation
$\rho$ corresponds to a hyperbolic structure on a torus with one singularity.
This singularity has a neighborhood isometric to a cone of cone angle 
$\cos\inv(\kappa(\xi,\eta,\zeta)/2).$
We call the collection of marked hyperbolic structures whose ends are either funnels, cusps, or complete to conical singularities 
the {\em generalized Fricke space\/} of $\Sigma$ .
It identifies in trace coordinates with the subset
\begin{equation*}
\Fricke'(\Sigma_{1,1}) := \{ (\xi,\eta,\zeta)\in\ (2,\infty)^3 \mid
\kappa(\xi,\eta,\zeta) < 2 \} \subset \R^3 \subset \C^3
\end{equation*}
%
%
If $\gamma\in\hGamma$, then either 
\[\gamma\cdot\Fricke(\Sigma_{1,1}) = \Fricke(\Sigma_{1,1})\] 
or $\gamma$ is a nonzero element of $\Hom(\Ft,\{\pm\Id\})$.
In particular $\hGamma\cdot\Fricke(\Sigma_{1,1})$ has four components,
corresponding to the four elements of $\Hom(\Ft,\{\pm\Id\})$.
Furthermore $\hGamma$ acts properly on
$\hGamma\cdot\Fricke(\Sigma_{1,1})$.

By definition, $\kappa< 2$ on $\Fricke(\Sigma_{1,1})$,
and  $\kappa\ge 18$ on $\Fricke(\Sigma_{0,3})$.
In particular, $\kappa > 2$.

The main results of 
\cite{MR2026539}
may be summarized as follows. 
The dynamics behavior depends crucially on the commutator trace
$k$. 
The classification divides into three cases, depending on whether
$k< 2$, $k= 2$, or $k > 2$, respectively.

\begin{thm*}
Suppose $k<2$. Let
\[  U := [-2,2]^3 \cap \kappa^{-1}(-\infty,2] \subset \R^3  \]
be the subset corresponding to unitary representations.
\begin{itemize}
\item  
$\hGamma$ acts properly on $\kappa^{-1}(k) \setminus U$.
\item
The action of $\hGamma$ on $U\cap\kappa^{-1}(k)$
is ergodic. 
\item If $k<-2$, then
\begin{equation*}
U\cap\kappa^{-1}(k) = \emptyset
\end{equation*}
\end{itemize}\end{thm*}

\noindent
The case $k=2$ corresponds to reducible representations  
and %
the     %
action of $\hGamma$ on $\kappa^{-1}(2)\setminus U$ is ergodic.
\begin{thm*} Suppose $k> 2$.
\begin{itemize}
\item
$\hGamma$ acts properly on
$\hGamma\cdot \Fricke(\Sigma_{0,3})$.

\item  
The action of  $\hGamma$ on 
$\kappa^{-1}(k) \setminus \hGamma\cdot \Fricke(\Sigma_{0,3})$
is ergodic.
\item
When $2\le k\le 18$, the action of 
$\hGamma$ on $ \kappa^{-1}(k)$ is ergodic.
\end{itemize}
\end{thm*}

\subsection{The Main Theorem}

The two homeomorphism types of nonorientable surfaces with fundamental
group $\Ft$, belong to the two-holed {\em cross-surface\/} (projective plane) $C_{0,2}$ and the
one-holed Klein bottle $C_{1,1}$.
(Note that $C_{0,2}$ and $C_{1,1}$ can each be obtained from
a three-holed sphere $\Sigma_{0,3}$
by attaching 
one or two cross-caps, respectively.)
In his thesis~\cite{MR2705402}, Stantchev
computed the Fricke spaces of these surfaces and proved partial results
on the $\Gamma$-action on the level sets of $\kappa$.
This paper presents and extends Stantchev's results.

In particular the Fricke space $\Fricke(C_{0,2})$ identifies with the
subset
\[\Fricke(C_{0,2}) := \{ (x,y,z)\in\R^3 \mid z \le -2, \; x y + z \ge 2\}\]
with boundary traces $z$ and 
\[ z' := - x y - z, \]
both of which satisfy $z,z'\le - 2$. 
The {\em generalized\/} Fricke space of $C_{1,1}$ identifies
with
\[ 
\Fricke'(C_{1,1}) := \{ (x,y,z)\in\R^3 \mid z < -2, \; x y  z > x^2 + y^2\}.
\]
with boundary trace $\delta := x^2 + y^2 - x y z + 2$.
Define
\begin{align*}
\R^3 &\xrightarrow{\kC} \R \\
(x,y,z) &\longmapsto -x^2 -y^2 + z^2  + x y z - 2,
\end{align*} corresponding to the trace of the commutator.
Again we divide the statements into three parts, depending on 
whether the commutator trace $k$ satisfies
$k< 2$, $k= 2$, or $k > 2$, respectively.

\begin{thm*} 
Suppose that $k < 2$. Then:
\begin{itemize}
\item
$\Gamma$ acts properly on $\Gamma\cdot\Fricke(C_{0,2})$.
\item
The action of  $\Gamma$ on the complement
$
(\kC)\inv (k)\setminus\Gamma\cdot\Fricke(C_{0,2})
$
is ergodic. 
\end{itemize}
When $k=2$, the action is ergodic. 

\noindent
Suppose $k>2$. Then:
\begin{itemize}
\item
$\Gamma$ acts properly on $\Gamma\cdot\Fricke'(C_{1,1})$.
\item $
\mathsf{interior}\Big( \kappa\inv(k)\setminus\Gamma\cdot\Fricke'(C_{1,1})\Big)
= \emptyset $
\end{itemize}
\end{thm*}
\noindent
In the last case, we conjecture that 
$\kappa\inv(k)\setminus\Gamma\cdot\Fricke'(C_{1,1})$
has measure zero.



Each of the surfaces $\Sigma_{0,3}$, $\Sigma_{1,1}$,
$C_{0,2}$, $C_{1,1}$ has fundamental group isomorphic to $\Ft$.
We choose a basis  $(X,Y)$ for $\pi_1(\Sigma)$ represented
by simple closed curves having geometric intersection number $0$ or $1$.
If $\Sigma$ is nonorientable then we choose $X,Y$ to reverse orientation.

\bigskip

These provide further examples, the first being Goldman~\cite{MR2026539},
Tan-Wong-Zhang~\cite{MR2405161},
of where the 
$\Gamma$-action is proper although the representations themselves
have dense image. 
Representations of $\Ft$ which are {\em primitive-stable\/} in the sense of Minsky~\cite{MR3038545} are included among the examples
in Tan-Wong-Zhang~\cite{MR2191691}.

A notable feature of this classification is that the dihedral characters corresponding to the triples $(0,0,z)$ where $z<-2$ or $z>2$
lie in the closure of the domain of discontinuity. 
These correspond to strong degenerations of a hyperbolic structure
on a Klein bottle with one cone point (as the angle approaches $2\pi$).
These degenerations lie on the boundary of the generalized Fricke orbits, 
but are not generic points on the boundary.

The dynamics of mapping class group actions on character varieties
is quite complicated, already in the case of  $\Rep(\Ft,\PSLtC)$,
where the fractal behavior is reminiscent of holomorphic dynamics in one
complex variable.  However, for $\PSLtR$-representations, this complication
is absent~\cite{MR2026539}.    
Another  notable feature of the present work is the
appearance of fractal-type behavior in  $\Rep(\Ft,\PGLtR)$.

Another new phenomenon is that the closure of the Bowditch set $\B$ contains dihedral characters.

Another motivation for this study is that imaginary characters (for $k=-2$) arise in Bowditch's original investigation~\cite{MR1643429} of Markoff triples, 
in proving that orbits
accumulate at the origin in the Markoff surface.

Maloni-Palesi-Tan~\cite{MR3420542} prove related results for $3$-generator groups.

\section*{Acknowledgements}
We are grateful to Joan Birman, Jean-Philippe Burelle, 
Serge Cantat, Dick Canary,
Virginie Charette,
John H.\ Conway, 
 Todd Drumm, Nikolai Ivanov, Misha Kapovich, 
Sean Lawton, Frank Loray,  Sara Maloni, 
John Millson,
Yair Minsky,
Walter Neumann, Frederic Palesi, Adam Sikora, Domingo Toledo,
Anna Wienhard, 
Maxime Wolff, Scott Wolpert, Yasushi Yamashita,
Eugene Xia, 
and Ying Zhang for helpful
discussions on this material.

Goldman and Stantchev thank the anonymous referee for a careful and
thoughtful reading of the original manuscript~\cite{GoldmanStantchev}
(which was never resubmitted to {\em Experimental Mathmatics\/}), 
and for many useful suggestions and corrections.
Goldman and Stantchev gratefully acknowledge support from
National Science Foundation grants DMS-070781, DMS-0405605 and
DMS-0103889 as well as the Department of Mathematics at the
University of Maryland.  Goldman gratefully acknowledges support
from the General Research Board at the University of Maryland during
the Fall semester of 2005. Tan was partially supported by the
National University of Singapore academic research grant
R-146-000-186-112
We are grateful to the Institute of Mathematical Sciences at the National University
of Singapore in 2010, Institut Henri Poincar\'e in Paris in 2012 and the Mathematical
Sciences Research Institute in Berkeley, California in 2015 for their hospitality
during the final stages of this work. Finally we thank the GEAR Research Network
in the Mathematical Sciences funded by NSF grant DMS-1107367 for their financial
support.
Goldman and Tan also express their gratitude to the Mathematical Sciences Research Institute where much of this work was finally completed in Spring 2015.

\section*{Notation and terminology}
Denote the cardinality of a set $S$ by $\#(S)$.
For a given equivalence relation on a set $S$,
denote the equivalence class of an element $s\in S$
by $[s]$.
If $S \xrightarrow{~A~} S$ is a mapping, 
denote the set of fixed points of $A$ by $\Fix(A)$.
Denote the {\em symmetric group\/} of permutations of $\{1,2,\dots,n\}$ by
$\mathfrak{S}_n$.

If $G$ is a group, denote the group of automorphisms $G\to G$ by $\Aut(G)$.
If $g\in G$, denote the corresponding inner automorphism by:
\begin{align*}
G &\xrightarrow{~\Inn(g)~} G \\
x &\xmapsto{\phantom{~\Inn(g)~}} gxg\inv. \end{align*}
Denote the cokernel of the homomorphism
\[
G \xrightarrow{~\Inn~} \Aut(G)  \]
by 
\[
\Out(G) := \Aut(G)/\Inn(G). \]
If $G, H$ are groups and $G\longrightarrow H$ is a homomorphism,
denote the corresponding semidirect product by $G \rtimes H$ or
$H \ltimes G$. Then $G \normalsubgroup (G\rtimes H)$ where $\normalsubgroup$
denotes the relation of {\em normal subgroup.}
Denote the center of a group $G$ by $\Center(G)$.

Denote the free group of rank two by $\Ft$.

The usual Euclidean coordinates on $\R^3$ (respectively $\C^3$)
will be denoted $x,y,z$ (respectively $\xi,\eta,\zeta$). 
The corresponding coordinate vector fields will be denoted
$\dd{x},\dd{y},\dd{z}$  
(respectively $\partial_\xi,\partial_\eta,\partial_\zeta$).

Denote the identity matrix by $\Id$.
Denote the transpose of a matrix $A$ by $A^\dag$.
An {\em anti-involution\/} of a complex algebraic object is
an anti-holomorphic self-mapping of order two.

If $A$ is a ring and $n\in\N$  then $\GL(n,A), \SL(n,A)$ have their usual meanings. 
$\SLtRpm$ denotes the subgroup of $\GLtR$ comprising matrices of determinant $\pm 1$.
Denoting multiplicative subgroup of units in $A$ by $A^\times$;
then $\PGL(n,A)$ denotes the cokernel of the homomorphism
\begin{align*}
A^\times &\longrightarrow   \GL(n,A) \\
a &\longmapsto a \Id 
\end{align*}
and $\PSL(n,A)$ denotes the image of $\SL(n,A)$ under the restriction of the
quotient epimorphism $\GL(n,A)\twoheadrightarrow \PGL(n,A)$.

The {\em level two congruence subgroup\/} $\GLtZ_{(2)}$ is the
kernel of the homomorphism defined by reduction modulo $2$:
\[
\GLtZ \longrightarrow \GL(2,\Z/2) \]
and denote the corresponding subgroup of $\PGLtZ$ by:
\[
\PGLtZt :=  \GLtZ_{(2)} / \{\pm \Id\} \subset \PGLtZ. \]
For any field $\kk$, denote the {\em projective line\/} over $\kk$
by $\Po(\kk)$. The elements of $\Po(\kk)$ are 
one-dimensional linear subspaces of the plane $\kk^2$.
Denote the one-dimensional subspace of $\kk^2$ containing $(a_1,a_2)$ by
$[a_1:a_2]\in\Po(\kk)$. The scalars $a_1,a_2$ are the 
{\em homogeneous coordinates\/} of the point $[a_1:a_2]\in\Po(\kk)$.

If $S$ is a manifold, denote its boundary by $\partial S$.
If $V\subset S$ is a hypersurface, then denote the manifold-with-boundary
obtained by splitting $S$ along $V$ by $S|V$. 
The quotient map $S|V \xrightarrow{~q~} S$
restricts to $q\inv(S\setminus V)$ by a homeomorphism,
and to $q\inv(V)$ by a double covering-space.

\section{The rank two free group and its automorphisms}\label{sec:RankTwoFreeGroupAndAutos}
\noindent 
In this section and the next we describe the group $\Gamma$ which acts 
polynomially on 
\[
\R^3 \xrightarrow{~\cong~} i\R\times i\R \times\R \hookrightarrow \C^3, \]
preserving the  function $\kC$ and the Poisson structure $\BB_\Coloring$ 
defined in \eqref{eq:RealPoissonStructure}.
Both the function $\kC$ and the bivector field $\BB_\Coloring$ are polynomial tensor fields.

The dynamical system $(\Gamma,\R^3)$ arises from automorphisms of the rank two free group $\Ft$ which preserve a nonzero $\{\pm 1\}$-character $\Coloring$ corresponding to the change of orientation of the invariant hyperbolic plane $\Ht\subset \Hth$. 
It also contains a normal subgroup 
$\Sigma \cong \Z/2 \oplus \Z/2$  of {\em sign-changes\/} 
which ``twists'' representations
$\rho$ by homomorphisms $\Ft \longrightarrow \{\pm 1\}$,
where $\{\pm 1\}$ is the center of the group $\SLtC$.
These transformations do {\em not\/} correspond to automorphisms
of $\Ft$ and we postpone their discussion to the next section.
In this section we describe the group arising from $\Au$ which acts faithfully
on the character variety $\RepF$, and its isomorphism to 
the {\em modular group\/} $\PGLtZ$.

%
\subsection{The modular group and automorphisms}
Let $\Ft$ be a free group freely generated by elements $X,Y$. 
We call such an ordered pair $(X,Y)$ of generators an (ordered) {\em basis\/}
of $\Ft$. 
Every automorphism $\phi\in\Au$ induces an automorphism
of the abelianization of $\Ft$, which is the free abelian group $\Z^2$.
The corresponding epimorphism 
\[ \Au  \twoheadrightarrow  \GLtZ \]
has kernel $\In$, and induces  an isomorphism $\Ou\cong\GLtZ$.

The center of $\GLtZ$ is $\{\pm \Id\}$.
The quotient  $\PGLtZ :=  \GLtZ/ \{\pm \Id\}$ 
enjoys the structure of a semidirect product
\[
\SymThree \ltimes (\Z/2 \star \Z/2 \star \Z/2), \]
where the symmetric group $\SymThree$ acts by permuting the three free factors.

Further analysis of  $\Au$ involves a homomorphism
$\Au\longrightarrow\SymThree$ defined by reduction modulo $2$.
The $\Z/2$-vector space 
\begin{equation}\label{eq:PlaneOverZ2}
\Z/2\oplus\Z/2 \ \cong\    \Hom(\Ft, \Z/2). \end{equation}
has three nonzero elements. 
Since the field $\Z/2$ has only one nonzero element, 
this three-element set identifies with the projective line $\Po(\Z/2)$:
\begin{alignat}{2}\label{eq:ThreeElementSet}
[0:1]&\longleftrightarrow\frac{\mathsf{even}}{\mathsf{odd}}&&\longleftrightarrow0 \notag \\
[1:0] &\longleftrightarrow  \frac{\mathsf{odd}}{\mathsf{even}}&&\longleftrightarrow\infty \notag 
\\
[1:1]&\longleftrightarrow\frac{\mathsf{odd}}{\mathsf{odd}}&&\longleftrightarrow1.
\end{alignat}
Every permutation of this set is realized by a unique  projective transformation,
resulting in an epimorphism
\[ \PGLtZ  \twoheadrightarrow \PGL(2,\Z/2) \cong \SymThree.  \]
This homomorphism splits.
Its kernel is the {\em level two congruence subgroup\/}
of $\PGLtZ$, denoted $\PGLtZt$.  
It is freely generated by the three involutions
\[
\II_3 := \bmatrix 1 & 0 \\ 0 & -1 \endbmatrix,\quad
\II_1 := \bmatrix -1 & 0 \\ -2 & 1 \endbmatrix,\quad
\II_2 := \bmatrix -1 & 2 \\ 0 & 1 \endbmatrix. \]
Geometrically these involutions correspond to reflections 
in the sides of an ideal triangle 
in the Poincar\'e upper half-plane $\Ht$ 
with vertices $0,1,\infty$ respectively.
Specifically, on $\QPo$,
\[
\Fix(\II_3) = \{0,\infty\},\quad
\Fix(\II_2) = \{\infty,1\},\quad
\Fix(\II_1) = \{1,0\}. \]
These elements of $\GLtZ$ correspond to automorphisms of $\Ft$ as follows:
The central involution $-\Id\in\GLtZ$ corresponds to the {\em elliptic involution\/}
\begin{align}\label{eq:EllInvolution} 
\Ft &\xrightarrow{~\e~} \Ft \\
X & \longmapsto X\inv \notag\\
Y & \longmapsto Y\inv. \notag
\end{align}
This involution and $\In$  generate
a normal subgroup 
\begin{equation}\label{eq:DefInne}
\Inne(\Ft) := \In \rtimes \langle\e\rangle \ \unlhd\  \Au, \end{equation}
with quotient  $\Au/\Inne(\Ft) \cong \PGLtZ$.

The automorphisms
\begin{equation}\label{eq:Involutions}
\begin{aligned}
X &\xmapsto{~\tII_3~} X \\
Y & \xmapsto{\phantom{~\tII_3~}}  Y\inv 
\end{aligned} \qquad
\begin{aligned}
X &\xmapsto{~\tII_1~} Y\inv X\inv Y\inv \\
Y & \xmapsto{\phantom{~\tII_2~}}  \qquad Y  \\
\end{aligned} \qquad 
\begin{aligned}
X &\xmapsto{~\tII_2~} \ X\inv \\
Y & \xmapsto{\phantom{~\tII_1~}}  X^2Y \end{aligned}\end{equation}
are involutions of $\Ft$,  and they represent the involutions $ \II_3,\II_1,\II_2$ respectively.

Reduction modulo $2$ 
\[ \QPo \longrightarrow \Po(\Z/2) \]
partitions $\QPo$ into the three $\PGLtZt$-orbits, 
represented by $\{0,\infty,1\}$ as in \eqref{eq:ThreeElementSet}.
Points in the orbit of 
\[ 1\longleftrightarrow \frac{\mathsf{odd}}{\mathsf{odd}} \]
we call {\em totally odd;\/} points in the other two orbits we call {\em partially even.\/}

Finally observe that the homomorphism
\[ \GLtZ \xrightarrow{\Det}  \pmOne \]
extends to a homomorphism
\begin{equation}\label{eq:AuDet}  
\Au \xrightarrow{\Det}  \pmOne \end{equation} 
factoring through
\[ \SymThree \xrightarrow{~\sgn~}  \pmOne. \] 
Using the identifications $\Ft \cong \pi_1(\Sigma_{1,1})$
and $\Ou \cong \Mod(\Sigma_{1,1})$, 
this action corresponds to the induced action of $\Mod(\Soo)$ on the relative homology group 
\[H_2\big(\Soo,\partial\Soo; \Z\big)\  \cong\  \Z, \]
defined by the effect on orientation on $\Soo$.

\subsection{The tree of superbases}\label{sec:Graphical Representation}
We associate to $\Ft$ a natural 
{\em a trivalent  tree\/} $\T$ upon which
$\Au$ acts.
(Figure~\ref{fig:tree} depicts this tree.)
The tree $\T$ comes equipped with an embedding in the plane,
and a tricoloring of its edges. 
Associated to a character $[\rho]$ is a natural {\em flow\/} $\vT_{\rho}$ on $\T$
(that is, a directed tree whose underlying tree is $\T$).
This dynamical invariant  encodes the dynamics of the $\Gamma$-action,
and is due to Bowditch~\cite{MR1643429}.

Vertices of the tree correspond to {\em superbases,\/}
which define coordinate systems $\RepF\cong \C^3$.
The automorphisms $\tII_j$ defined in \eqref{eq:Involutions}
correspond to the edges, and generate our dynamical system.

An {\em (ordered) basis\/} is an ordered pair $(X,Y)$ where
$\{X,Y\}$ freely generates $\Ft$.
An element of $\Ft$ is {\em primitive\/} if it lies in a basis.
Denote the set of primitive elements of $\Ft$ by $\Prim$.
Primitive elements $X,Y$ are {\em equivalent\/} if $X$ is conjugate to
either $Y$ or $Y\inv$. This equivalence relation on $\Prim$ arises from 
the $\In\times \Z/2$-action where  
inversion $Y\mapsto Y\inv$ generates the $\Z/2$-action.
Geometrically, equivalence classes of primitives correspond to
isotopy classes of {\em unoriented\/} essential nonperipheral simple closed
curves on $\Soo$.

It is well known that equivalence classes $[Z]$ of primitive elements
correspond to points of $\QPo$: an equivalence class $[Z]$ 
corresponds to $[p:q]$ (that is, $p/q \in \Q$, assuming $q\neq 0$)
if and only if
\[
Z  \equiv  X^p Y^q  \mod{[\Ft,\Ft]}\]
in the abelianization \[\Ft/[\Ft,\Ft] \cong \Z^2 \] 
and \[ \infty \longleftrightarrow [1:0] \in \QPo. \]
Equivalently, the pair $(p,q)\in\Z^2$ corresponds to the homology class
of $Z$ in $H_1(\Ft,\Z)\cong\Z^2$, or the projectivized homology class
of an {\em oriented curve\/} representing the unoriented curve.

Define a {\em basic triple\/} to be an ordered triple $(X,Y,Z)\in 
\Prim\times\Prim\times\Prim$ such that:
\begin{itemize}
\item $(X,Y)$ is a basis of $\Ft$; 
\item $XYZ = \Id$. 
\end{itemize}
Clearly, an ordered basis $(X,Y)\in \Prim\times\Prim$ of $\Ft$ 
extends uniquely to a basic triple, and bases of $\Ft$ correspond bijectively to basic triples. 

A {\em superbasis\/} of $\Ft$ is an equivalence class of basic triples, 
where the equivalence relation is defined as follows.
Basic triples $(X,Y,Z)$ and $(X',Y',Z')$ are {\em equivalent\/} if and only if
\[ 
[X']=[X],\quad  [Y']=[Y], \quad [Z']=[Z]. \] 
Alternatively, a superbasis is an $\Inne(\Ft)$-equivalence class of
basic triples, where $\Inne(\Ft)$ is defined in \eqref{eq:DefInne}
and $\e$ is the elliptic involution defined in \eqref{eq:EllInvolution}.
That is, two basic triples $(X,Y,Z)$ and $(X',Y',Z')$ represent the
same superbasis if and only if 
\begin{equation}
\begin{aligned}
X' & =  W X W\inv \\
Y' & =  W Y W\inv \\
Z' & =  W Z W\inv \end{aligned}
\qquad \text{~or~} \qquad 
\begin{aligned}
X' & =  W X\inv W\inv \\
Y' & =  W Y\inv W\inv \\
Z' & =  W YZ\inv Y\inv W\inv\end{aligned}
\end{equation}
for some $W\in\Ft$. 
(Compare Charette-Drumm-Goldman~\cite{CDG}.)
Geometrically, a superbasis corresponds to an ordered triple of isotopy classes
of unoriented simple loops on $\Soo$ with mutual geodesic intersection numbers $1$.

The notion of {\em superbasis\/} is due to Conway~\cite{MR1478672}
in the context of Markoff triples. 
Our notion is a slight modification to $\Ft$, since Conway consider superbases
of $\Z^2$. 

Now we define the three elementary moves on a superbasis.
These correspond to the edges in the tree $\T$. 
Suppose $\mathfrak{b}$ is a superbasis represented by a basic triple 
$(X,Y,Z)$ as above. 
Then $(X,Y)$ is a basis of $\Ft$, and the pair $([X],[Y])$ correspond to two (unoriented) simple loops intersecting transversely in one point $p$, 
which we can conveniently use as a base point to define the 
fundamental group $\pi_1(\Soo,p)$. 
There are exactly two ways to extend 
$([X],[Y])$ to a superbasis, depending on the the respective orientations of representative elements of $\pi_1(\Soo,p)$. 
That is, there is a unique superbasis $(X',Y',Z')$
where 
\[ [X'] = [X],\quad   [Y'] = [Y], \quad [Z'] \neq [Z]. \]
Explicitly, this corresponds to the transformation of basic triples:
\begin{align*}
X' & = X \\
Y' & = Y\inv \\
Z' & = Y X\inv, \end{align*}
the automorphism $\tII_3$ defined in \eqref{eq:Involutions}.

Similarly there are unique superbases corresponding to basic triples $(X',Y',Z')$
with 
\[ [X'] \neq [X],\quad   [Y'] = [Y], \quad [Z'] = [Z] \]
and 
\[ [X'] = [X],\quad   [Y'] \neq [Y], \quad [Z'] = [Z], \]
which respectively correspond to $\tII_1$ and $\tII_2$.
We call these three superbases the superbases {\em neighboring \/}  
the superbasis $\big([X],[Y],[Z]\big)$.

\begin{figure}
\includegraphics[width=\textwidth]{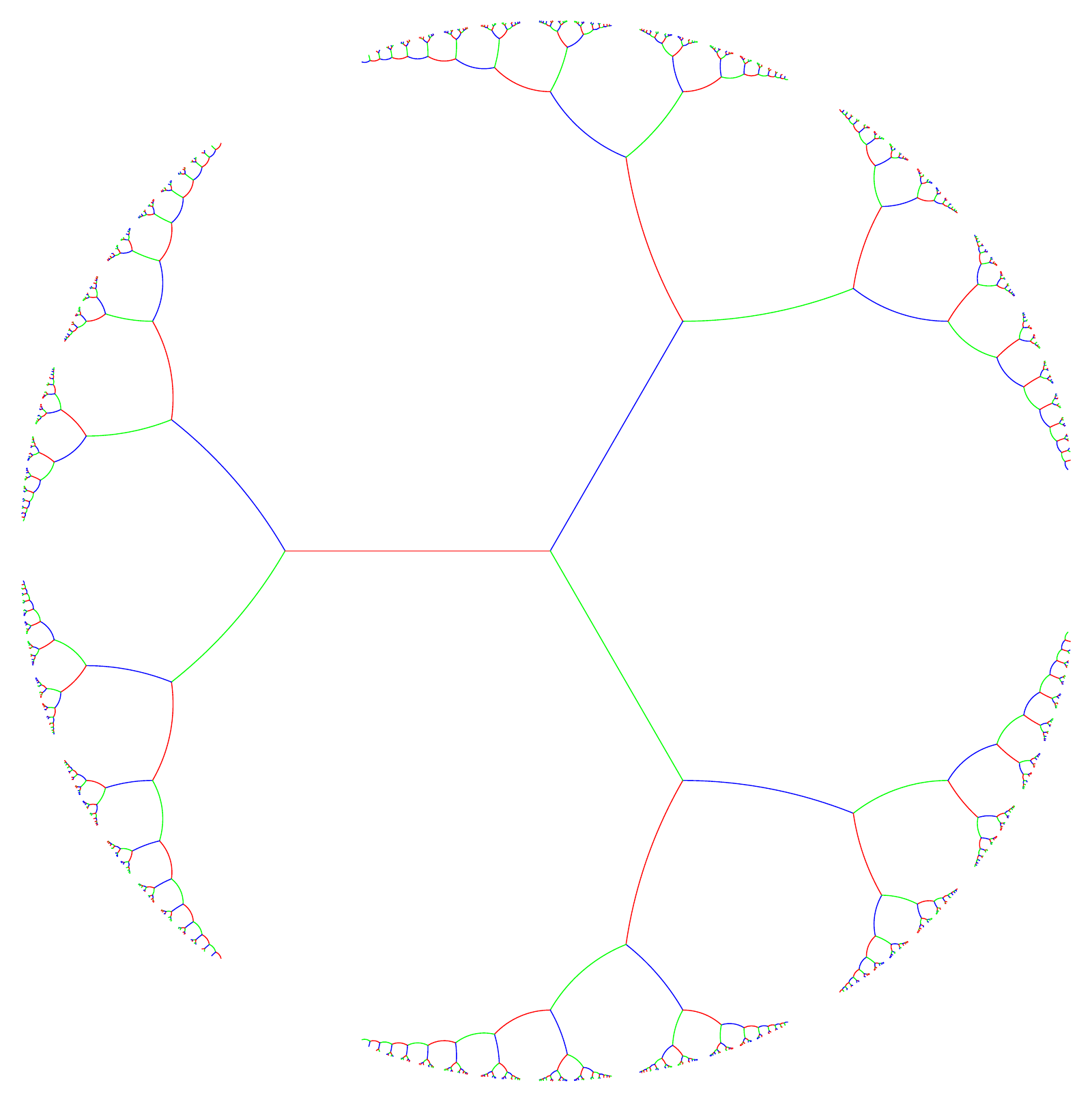}
\caption{A  trivalent tree embedded in $\Ht$}
\label{fig:tree}
\end{figure}

The tree $\T$ is defined as follows. 
Its vertex set $\V$ consists of superbases.
We denote the vertex corresponding to a superbasis
$\big([X],[Y],[Z]\big)$ by \[ v =  v(X,Y,Z) \in \V. \]
Denote the set of edges by $\E$. 
Edges correspond to pairs of neighboring superbases,
or,  equivalently, equivalence classes of bases,
and we write 
\[ e =  e^{X,Y} \in \E \]
for the edge corresponding to the basis $(X,Y)$.
This tree is just the Cayley graph of the {\em rank three free Coxeter group\/}
\[ \Z/2 \star \Z/2 \star \Z/2 \ \cong \PGLtZt. \]

\subsection{The tricoloring and the planar embedding}\label{sec:Tricoloring}
Bowditch begins with the tree $\T$ and a planar embedding of $\T$.
The extra structure of the planar embedding arises naturally as follows.
Since $\T$ is the Cayley graph of $\Z/2 \star \Z/2 \star \Z/2$,
its edges correspond to the three free generators $\II_1,\II_2,\II_3$.
Thus we may color the edges accordingly, 
obtaining a {\em tricoloring.\/}
Choosing a cyclic ordering of these three generators determines the 
germ of an embedding into a regular $2$-dimensional neighborhood of each vertex,
that is, the structure of a {\em fatgraph\/} on $\T$; 
this structure extends to a planar embedding of $\T$.
This planar embedding (and the set $\O$ of components of the complement
of $\T$) plays a central role in Bowditch's theory.

We introduce the following notation and terminology for $\T$.
Let $\E$ denote the set of edges in $\T$ and
$\O$ the set of complementary regions in the plane.
Then $\O$ bijectively corresponds to the set of equivalence classes
of primitive elements of $\Ft$. 
If $\big([X],[Y],[Z]\big)$ is a superbasis, 
then the three complementary regions around $v$ correspond to
$X, Y$ and $Z$. 
If $e = e^{X,Y}$, then we say that the edge $e$ {\em abuts\/} the regions
corresponding to $X$ and $Y$. 
One of the endpoints of $e$ is $v(X,Y,Z)$ and the other endpoint is
$v(X,Y,Z')$ where $\big([X],[Y],[Z']\big)$ is the neighboring superbasis.
We write 
\[ e = e^{X,Y}(Z,Z') \in \E, \] 
and say that the edge $e$ {\em ends\/} at the regions corresponding to $Z$ and $Z'$.
Compare Figure~\ref{fig:FourRegions}.


\begin{figure}
\includegraphics[width=\textwidth]{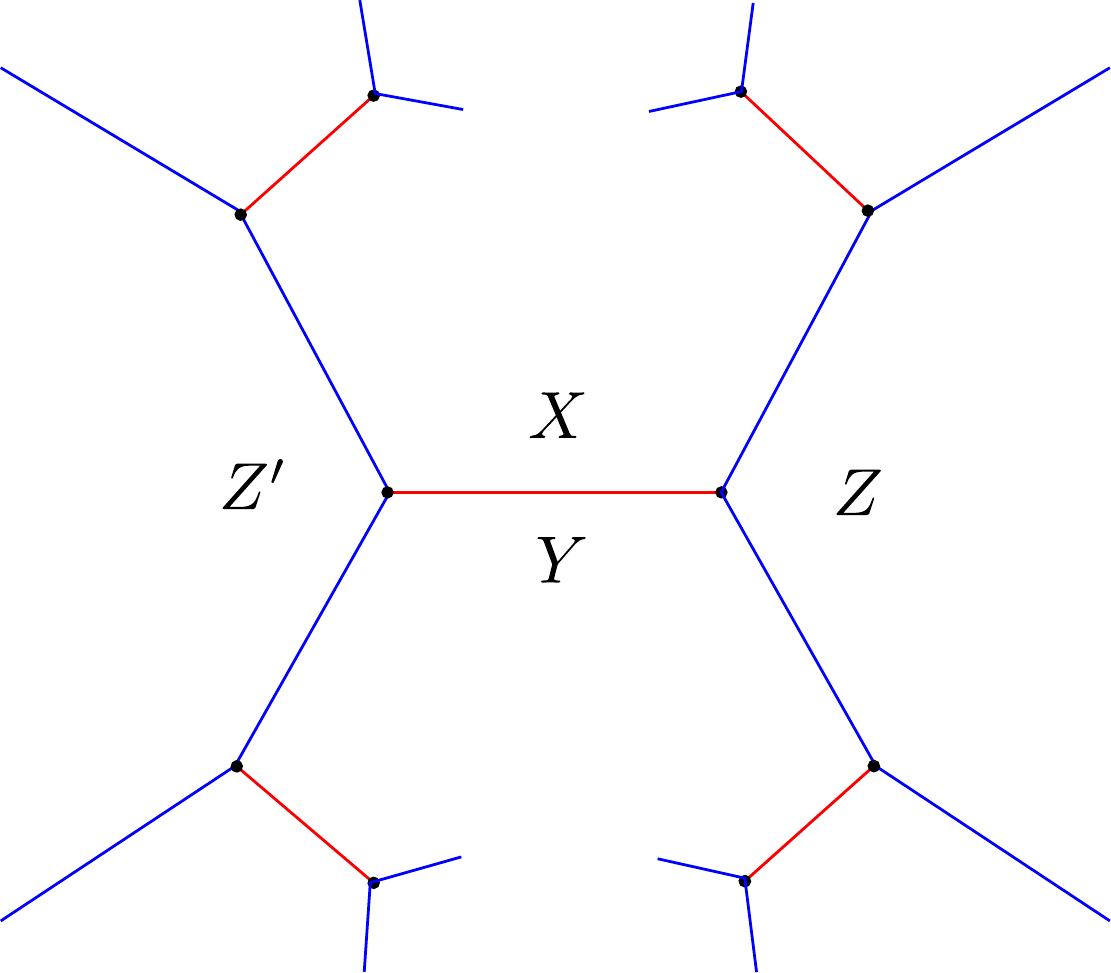}
\caption[Four regions associated to an edge]{The four regions associated to an edge.
The middle edge abuts the two regions labeled $X$ and $Y$, 
ends at the two regions labeled $Z'$ and $Z$.}  
\label{fig:FourRegions}
\end{figure}

Reduction modulo $2$ solidifies the role of the tricoloring.
Since 
\[\O \;:=\; \pi_0(\Ht\setminus\T) \longleftrightarrow 
\big(\Prim/\sim\big) ~\longleftrightarrow~ \QPo, \]
the trichotomy 
\[ \QPo  \twoheadrightarrow \Po(\Z/2) \cong \{\infty, 0, 1\} \]
tricolors $\O$.  
For every basic triple $(X,Y,Z)$, 
the mod $2$ reductions
\[ \{ [X]\mod{2}, [Y]\mod{2}, [Z]\mod{2} \} = \{\infty, 0, 1\}. \]
(This follows readily from the fact that the abelianizations of $X,Y,Z$ generate $\Z^2$.)
Thus for each vertex $v=v(X,Y,Z)$ as above,
the three regions around $v$ have three colors.

For an edge $e = e^{X,Y}(Z,Z')$ as above, then 
where $Z = Y\inv X\inv$ and $Z' = Y\inv X$,
and $Z$ and $Z'$ have the same parity and so determine the same 
element in $\Po(\Z/2) = \{\infty, 0, 1\}$, which is an invariant of the edge.
The resulting map
\[ \E \twoheadrightarrow \{\infty, 0, 1\}. \]
describes a {\em tricoloring\/} of the tree. 
That is, each edge is {\em colored\/} by one of $\infty, 0$, or $1$, as well.

Furthermore each edge $e$ is fixed by a unique involution of $\T$ which is
conjugate to a unique  $\II_j$ where $j=1,2,3$ indexes the tricoloring.
We write $j = j(e)$.

Hu-Tan-Zhang\cite{HuTanZhang} give an alternate approach to this structure,
not using a planar embedding. They define $\O$ using
{\em alternating geodesics,\/} which correspond to {\em alternating words\/} in the $\II_j$,
as described in \S\ref{sec:Paths}.

\subsection{Paths and Alternating Geodesics}\label{sec:Paths}

\begin{definition}\label{def:Path}
By a {\em finite path\/} we mean a sequence
\[ P = \big(
v_0 \stackrel{e_0}{\longdash}v_1
\stackrel{e_1}{\longdash}
\cdots 
\stackrel{e_{n-2}}{\llongdash}v_{n-1}
\stackrel{e_{n-1}}{\llongdash}v_n
\big)
\]
where each $v_i\in\V$, each $e_i\in\E$ and $\partial e_i = \{v_i,v_{i+1}\}$.
An {\em infinite path\/} is an infinite sequence
\[ 
v_0 \stackrel{e_0}{\longdash}v_1
\stackrel{e_1}{\longdash}
\cdots 
\stackrel{e_{n-2}}{\llongdash}v_{n-1}
\stackrel{e_{n-1}}{\llongdash}v_n
\stackrel{e_n}\longdash \cdots
\]
such that each
\[
v_0 \stackrel{e_0}{\longdash}v_1
\stackrel{e_1}{\longdash}
\cdots 
\stackrel{e_{k-2}}{\llongdash}v_{k-1}
\stackrel{e_{k-1}}{\llongdash}v_k
\]
is a finite path as above, for each $k>0$.
A path is a {\em geodesic\/} if and only if $e_i \neq e_{i+1}$ for all $i$. 
When the tree is directed, we shall consider directed geodesics, 
where $e_i$ points from $v_{i-1}$ to $v_i$.
\end{definition}

\noindent
Associated to a finite path in $\T$ is the word 
\[
w(P) := \II_{j(e_{n-1})} \cdots \II_{j(e_0)}  \ \in \  
\Z/2 \star \Z/2 \star \Z/2\ \cong \ 
\PGLtZt. \] 
If the path is a geodesic, then the word is {\em reduced,\/}
that is, the sequence $j(e_0), \dots,  j(e_n) $ contains no repetitions.
If $[(X_0,Y_0,Z_0)]$ is the superbasis corresponding to $v_0$, 
then the superbasis
\begin{equation}\label{eq:ActionOnSuperbases}
\Big[\big(w(P)(X_0),\ w(P)(Y_0),\  w(P)(Z_0)\big) \Big] \end{equation}
corresponds to $v_n$.

Given a complementary region $Z\in\O$,  
the set  $C(Z)$ of edges abutting $Z$ forms a geodesic
in $\T$.
The corresponding reduced word 
is an alternating sequence of $\II_i$ and $\II_j$ where
$i,j\in \{1,2,3\}$. 
Thus $\O$ defines a collection of preferred geodesics,
called {\em alternating geodesics.\/}
The set $\O$ bijectively corresponds to the set of alternating geodesics.

\subsection{Relation to the one-holed torus}
This tree arises from the topology of $\Soo$ as follows.
The {\em pants graph\/}
$\PG$ of $\Soo$ is the following graph:
\begin{itemize}
\item
The vertex set consists of the set of free homotopy classes of
unoriented essential, non-boundary parallel simple closed curves on $\Soo$
and identifies with $\Omega$.
\item
Two vertices are connected by an edge if and only if the geometric
intersection number of the corresponding curves is one.
\end{itemize}
The valence of every vertex is infinite, and every edge is
contained in precisely two triangles in the graph. The pants graph
identifies with the Farey tessellation on $\Ht$, with its
vertex set identified with $\QPo$.
This identification is fixed once
we fix an identification of $X,Y$ and $Z=(XY)\inv$ with $0$, $\infty$
and $1$, for a fixed basis $X,Y$ of $\Ft$.
The graph dual to $\PG$ is the tree $\T$,
with vertex set $\V$, edge set $\E$ and directed edge set $\vE$.

\subsection{Effect of a $\{\pm 1\}$-character}\label{sec:Effect}

This paper concerns actions $\rho\in\HomF$ which preserve a hyperbolic
plane $\Ht\subset\Hth$ but not an orientation on $\Ht$. 
Thus the image of $\rho$ lies in the stabilizer $\PGLtR$ of $\Ht$.
How $\rho$ fails to preserve an orientation is detected by the composition
\[
\Ft \xrightarrow{~\rho~} \PGLtR \longrightarrow \pi_0\big(\PGLtR\big) \cong 
\{\pm 1\}.\]
This defines an invariant of $\rho$, which is a nonzero element
\[
\Coloring \in \Hom(\Ft,\{\pm 1\}).\]
The three nonzero elements of 
$\Hom(\Ft,\{\pm 1\})$ are permuted by 
$\Au$ via the homomorphism to $\SymThree$.
We will choose one of these elements  $\Coloring$ once and for all, 
and define the subgroup 
$\PGLtZc$ to be the stabilizer of $\Coloring$  in $\PGLtZ$. 
Clearly $\PGLtZc$ has index three in $\PGLtZ$ and contains $\PGLtZt$ with index two.
Indeed the diagram
\[
\begin{CD}
\PGLtZt @>>> \PGLtZc @>>>  \PGLtZ \\
@VVV @VVV  @VVV \\
1  @>>> \langle (12)\rangle @>>>  \SymThree
\end{CD}
\]
commutes, where the horizontal arrows are inclusions and  the vertical arrows are quotient homomorphisms
by the congruence subgroup \[\PGLtZt = \Ker\big( \PGLtZ \to \SymThree\big).\] 

Since $\Ft = \langle X,Y\rangle$, the kernel $\Ker(\Coloring)$ must contain
at least one of $X,Y$. 
We choose $\Coloring$ to be nontrivial on both $X$ and $Y$:
\begin{align}\label{eq:ExampleColoring}
\Ft &\xrightarrow{~\Coloring~}  \{\pm 1\} \\
X &\longmapsto  -1  \notag\\
Y &\longmapsto  -1  \notag\end{align}
$\PGLtZc$ equals the inverse image
of the cyclic group 
$\langle (12) \rangle \subset \SymThree$ under the homomorphism
\[ \hGamma \twoheadrightarrow \SymThree \]
and is generated by $\PGLtZt$ and the automorphism corresponding
to the transposition $(12)\in\SymThree$:
\begin{align*}
\Ft &\xrightarrow{~\mathscr{P}_{(12)}~}  \Ft \\
X &\xmapsto{\phantom{~\mathscr{P}_{(12)}~}}  Y \\
Y & \xmapsto{\phantom{~\mathscr{P}_{(12)}~}}  X \end{align*}

\section{Character varieties and their automorphisms}

Our basic object of study is 
the {\em character variety\/} $\RepF$  of equivalence classes of 
$\SLtC$-representations of $\Ft$. 
Two representations are equivalent if the closures of 
their $\Inn\big(\SLtC\big)$-orbits intersect. 
Here the closures of $\Inn\big(\SLtC\big)$-orbits 
are encoded using the trace function
$\SLtC\xrightarrow{~\tr~}\C$ applied to the representation.

We begin with the spaces and their geometry; 
the main result is Vogt's theorem which identifies $\RepF$ with affine
space $\C^3$. 
This identification $\RepF\longleftrightarrow\C^3$ depends on a superbasis.
We compute the action of the group $\Gamma$ by polynomial automorphisms,
which we interpret as a group of automorphisms of the tree $\T$.
Then we discuss how the $\{\pm 1\}$-character $\Coloring$ defines a
{\em purely imaginary real form\/} of $\RepF$.
Finally we describe the invariant function $\kC$ and the Poisson bivector $\BB$.

\subsection{The deformation space}
The space $\HomF$ of $\SLtC$-representations of $\Ft$ identifies with
the Cartesian product $\SLtC\times\SLtC$ via:
\begin{align*}
\HomF &\longrightarrow \SLtC\times\SLtC \\
\rho\qquad &\longmapsto \big( \rho(X), \rho(Y) \big) \end{align*}
Composition of $\rho$ with an inner automorphism $\Inn(g)$
corresponds to the diagonal action by conjugation of $g$
on $\SLtC\times\SLtC$. 
The {\em character variety\/} is defined as the categorical quotient of 
this action. 

To describe this action, a more convenient presentation of $\Ft$ is the
following redundant presentation corresponding to a basic triple:
\begin{equation}\label{eq:presentation}
\Ft = \langle X, Y, Z \mid X Y Z = \Id \rangle ),
\end{equation}
(so that $Z = (XY)^{-1}$). 
The mapping
\begin{align*}
\HomF &\xrightarrow{~\mathfrak{X}~} \qquad \C^3 \\
\rho\qquad &\longmapsto  \bmatrix
\tr\big(\rho(X)\big) \\
\tr\big(\rho(Y)\big) \\
\tr\big(\rho(Z)\big) \endbmatrix.
\end{align*}
is a {\em categorical quotient\/}  for the 
$\Inn\big(\SLtC\big)$-action, that is:

\begin{prop}[Vogt~\cite{MR1508833}] 
Let 
\[\HomF\xrightarrow{~f~}\C,\]
be an $\Inn\big(\SLtC\big)$-invariant regular function. Then
\[ f(\rho) = F\Big(\tr\big(\rho(X)\big),  \tr\big(\rho(Y)\big),  \tr\big(\rho(Z)\big)\Big) \]
for a unique polynomial function 
\[ F(\xi,\eta,\zeta)\in\C[\xi,\eta,\zeta]. \]  
Furthermore the restriction of $\mathfrak{X}$ to 
the subset of irreducible representations is a quotient
mapping for the action of $\Inn\big(\SLtC)$ on  $\HomF$.
In particular two irreducible representations have the same character
if and only if they are $\SLtC$-conjugate. 
\end{prop}
\noindent(See Goldman~\cite{MR2497777} for a twenty-first century discussion.)
\subsection{Sign-changes}
The group $\PGLtZ$ acts faithfully and polynomially on the $\SLtC$-character variety of $\Ft$ and encodes the action of $\Au$. 
It preserves extra structure coming from the function $\kappa$ and the 
Poisson bivector $\BB$. 
However, a small extension of this action preserves all this structure.
This extension includes transformations arising from different ways of lifting representations from
$\PSLtC$ to $\SLtC$. 

Suppose that $\rho_1,\rho_2\in\HomF$ both project to the same representation
in $\Hom\big(\Ft,\PSLtC\big)$. 
Let $\gamma\in\Ft$. 
Since 
\[ \Ker\Big( \SLtC \longrightarrow \PSLtC\Big) = \pmI, \] 
the difference $\rho_1(\gamma)\rho_2(\gamma)\inv \in \pmI$.
Since $\pmI = \Center\big(\SLtC\big)$, 
\begin{align*}
\Ft &\xrightarrow{~\sigma~}  \pmI \\
\gamma &\longmapsto \rho_1(\gamma)\rho_2(\gamma)\inv \end{align*}
is a homomorphism. Conversely, any
homomorphism $\Ft \xrightarrow{~\sigma~}  \pmI$
acts on $\HomF$, preserving the projection
\[
\HomF \longrightarrow \Hom\big(\Ft,\PSLtC\big). \]
Indeed, this projection is a $\Sigma$-principal bundle,
where 
\[ \Sigma := \Hom\big(\Ft,\pmI\big) \cong \pmOne\times\pmOne \]
the {\em Klein four-group,\/}
with three nontrivial elements
$\sigma_1,\sigma_2,\sigma_3$, each of order two.

Let $\rho\in\HomF$. Here is how $\sigma\in\Sigma$ acts on $\rho$:
\begin{equation}\label{eq:SignChanges}
\begin{aligned}
X &\xmapsto{~\sigma_1(\rho)}  \rho(X) \\
Y &\xmapsto{\phantom{~\sigma_1(\rho)}} -  \rho(Y) \end{aligned}
\qquad
\begin{aligned}
X &\xmapsto{~\sigma_2(\rho)}  -\rho(X) \\
Y &\xmapsto{\phantom{~\sigma_2(\rho)}}  \rho(Y) \end{aligned}
\qquad\begin{aligned}
X &\xmapsto{~\sigma_3(\rho)} - \rho(X) \\
Y &\xmapsto{\phantom{~\sigma_3(\rho)}}   -  \rho(Y).\end{aligned}
\end{equation}
We call $\Sigma$ the group of {\em sign-changes.\/}

Both $\Au$ and $\Sigma$ act on $\HomF$, and they  generate an action
of the semidirect product $\Au \ltimes \Sigma$ on $\HomF$. 
Furthermore $\Inne(\Ft)$ remains normal in this group.
The quotient group
\[
\hGamma := 
\big( \Au \ltimes \Sigma\big) / \Inne(\Ft) 
\cong \PGLtZ \ltimes \Sigma
\]
is the group which acts faithfully on the $\SLtC$-character variety.
Note that the semidirect product is defined by the homomorphism
$\PGLtZ \longrightarrow \SymThree$, where $\SymThree$ acts 
by permutations of $\Sigma\setminus\{1\}$.
In particular $\hGamma$ is a split extension
\[
\PGLtZt \times \Sigma
\longrightarrow \hGamma
\longrightarrow \SymThree. \]

Observe that the homomorphism $\Det$ defined in \eqref{eq:AuDet}
extends to a homomorphism
\begin{equation}\label{eq:hGammaDet}
\hGamma \xrightarrow{\Det}  \pmOne. \end{equation}

\subsection{Action of automorphisms}\label{sec:FaithfulAction}
The sign-change group $\Sigma$ and the automorphism group $\Au$ 
act on $\RepF \cong \C^3$ preserving its algebraic structure.
They generate an action of the semidirect product $\Au\ltimes\Sigma$
on $\C^3$ by polynomial automorphisms.
Here we describe this action explicitly on some of the generators.

Observe that $\Au\ltimes\Sigma$ does not act faithfully on
$\RepF$. We describe its kernel. 
First note that the trace function
\[ \SLtC \xrightarrow{~\tr~} \C \] is $\Inn\big(\SLtC\big)$-invariant,
that is, it is a {\em class function.} 
Hence $\Inn(\Ft)$ acts trivially on $\RepF$. 
Therefore the $\Au\ltimes\Sigma$-action factors through
$\Ou\ltimes\Sigma$.

This action is still not faithful, since the 
elliptic involution $\e$ defined in \eqref{eq:EllInvolution}
acts trivially as well. 
First note that $\tr$ is {\em inversion-invariant:\/}\newline
If $A\in\SLtC$, then 
\begin{equation}\label{eq:InversionInvariance}
 \tr(A\inv) = \tr(A).
 \end{equation}
If $\rho\in\HomF$, then 
\[ \e(Z) = X\inv Y\inv = \Inn(X\inv)( Z\inv ),  \]
so:
\begin{align*}
\tr\Big(\rho\big(\e(X)\big) \Big) & = 
\tr\big(\rho(X)\inv \big)  = \tr\big(\rho(X) \big) \\
\tr\Big(\rho\big(\e(Y)\big) \Big) & = 
\tr\big(\rho(Y)\inv \big)  = \tr\big(\rho(Y) \big) \\
\tr\Big(\rho\big(\e(Z)\big) \Big) & = 
\tr\Big(\Inn\big(\rho(X)\inv\big) \rho(Z)\inv \Big)  = \tr\big(\rho(Z) \big).
\end{align*}
Therefore $\e$ acts trivially on $[\rho]$ as claimed.

Therefore the restriction of the $\big(\Au\ltimes\Sigma\big)$-action 
to the subgroup $\Inne(\Ft)\subset\Au$ is trivial, 
and the $\big(\Au\ltimes\Sigma\big)$-action factors through
\[
\hGamma := \big(\Au/\Inne(\Ft)\big)  \ltimes\Sigma \cong \PGLtZ\times\Sigma. \] 

The sign-changes $\Sigma$, defined in \eqref{eq:SignChanges}, 
act on characters by:
\[  
\bmatrix \xi \\ \eta \\ \zeta \endbmatrix \xmapsto{~\sigma_1~}  
\bmatrix \xi \\ -\eta \\ -\zeta \endbmatrix,\qquad
\bmatrix \xi \\ \eta \\ \zeta \endbmatrix \xmapsto{~\sigma_2~} 
\bmatrix -\xi \\ \eta \\ -\zeta \endbmatrix, \qquad 
\bmatrix \xi \\ \eta \\ \zeta \endbmatrix \xmapsto{~\sigma_3~}  \bmatrix -\xi\\ -\eta \\ \zeta \endbmatrix \]
The involutions $\II_i$, defined in \eqref{eq:Involutions}, 
freely generate $\PGLtZt$. 
They act by the following {\em Vieta involutions\/} on $\C^3$: 
\[ 
\begin{aligned}
\bmatrix \xi \\ \eta \\ \zeta \endbmatrix \xmapsto{~\II_1~}  
\bmatrix \eta\zeta - \xi \\ \eta \\ \zeta \endbmatrix \end{aligned}
\qquad
\begin{aligned}
\bmatrix \xi \\ \eta \\ \zeta \endbmatrix \xmapsto{~\II_2~}  
\bmatrix \xi \\ \xi\zeta -\eta \\ \zeta \endbmatrix \end{aligned}
\qquad
\begin{aligned}
\bmatrix \xi \\ \eta \\ \zeta \endbmatrix \xmapsto{~\II_3~}  
\bmatrix \xi \\ \eta \\  \xi\eta-\zeta \endbmatrix \end{aligned}
\]
Although $\PGLtZ\longrightarrow\SymThree$ splits,
the composition 
\[ \Au\longrightarrow\PGLtZ\longrightarrow\SymThree \] 
does not split. 
However, 
a left-inverse $\SymThree\longrightarrow\PGLtZ$ determines the usual
action of $\SymThree$ on $\C^3$ by permuting the coordinates.
For example the transposition $(12)\in\SymThree$ induces the automorphism:
\begin{align*}
\C^3 &\xrightarrow{~\mathscr{P}_{(12)}~} \C^3 \\
\bmatrix \xi \\ \eta \\ \zeta \endbmatrix &\xmapsto{\phantom{~\mathscr{P}_{(12)}~}}
\bmatrix \eta \\ \xi \\ \zeta \endbmatrix \end{align*}

For properties of this action see
Magnus~\cite{MR558891},
Bowditch~\cite{MR1643429}, 
\cite{MR2497777}, 
Cantat-Loray~\cite{MR2649343} and
Cantat~\cite{MR2553877}.

\subsection{Real forms of the character variety}\label{sec:RealForms}

Several {\em real structures\/} on $\RepF$ exist.
For example, complex-conjugation on $\SLtC$ induces an
anti-involution on $\HomF$ whose fixed-point set is  $\Hom(\Ft,\SLtR)$.
Another example is the {\em Cartan anti-involution\/}
\begin{align*} \SLtC & \longrightarrow \SLtC \\
A &\longmapsto  (\bar{A}^\dag)\inv, \end{align*}
which fixes $\SUt\subset\SLtC$.
This anti-involution induces an anti-involution of $\HomF$ whose fixed-point set is  $\Hom(\Ft,\SUt)$. 
They induce the same anti-involution
on $\RepF \cong \C^3$ (given by complex-conjugation on $\C^3$)
and whose fixed-point set is $\R^3\subset\C^3$.
This coincidence arises from the fact that
the two anti-involutions on $\SLtC$ are $\Inn\big(\SLtC\big)$-related:
\[
(\bar{A}^\dag)\inv \  = \ J {\bar A} J\inv \  = \  \Inn(J) {\bar A} 
\quad \text{~where~}\quad  J := \bmatrix 0 & -1 \\ 1 & 0 \endbmatrix. \]

We are interested in a slight variation of this real structure, 
when it is {\em twisted\/} by a character $\Coloring\in \Hom(\Ft,\{\pm 1\})$.
Namely, the map associating to a representation $\rho\in\HomF$
the representation
\[
\gamma \longmapsto \Coloring(\gamma) \ \overline{\rho(\gamma)} \]
is an anti-involution of $\HomF$ inducing an anti-involution of
$\RepF$. When $\Coloring$ is trivial, this is just the above anti-involution. 
When $\Coloring$ is the $\{\pm 1\}$-character defined in \eqref{eq:ExampleColoring},
then the corresponding anti-involution of $\RepF \cong \C^3$ fixes
\[ i\R \times i\R \times \R \subset \C^3. \]

\subsection{Real and imaginary characters}
Real characters correspond to representations conjugate to
$\SUt$-representations or $\SLtR$-rep\-re\-sen\-ta\-tions.
In terms of hyperbolic geometry,
$\SUt$-representations correspond to actions
of $\Ft$ on $\Hth$ which fix a point in $\Hth$ and
$\SLtR$-representations correspond to actions
which preserve an oriented plane $\Ht$ in $\Hth$.
A point $(\xi,\eta,\zeta)\in\R^3$ corresponds to an $\SUt$-representation
if and only if
\begin{equation*}
-2 \le \xi,\eta,\zeta \le 2, \quad \xi^2 + \eta^2 + \zeta^2 - \xi \eta \zeta \le 4.
\end{equation*}
Equivalence classes of $\SLtR$-representations correspond to
points in the complement in $\R^3$  of the interior of this set.
(The boundary of this set consists of representations in
$\mathsf{SO}(2) = \SUt \cap \SLtR$.)

The set of $\{\pm 1\}$-characters equals the four-element group 
\[ 
\Hom(\Ft,\{\pm 1\}) \cong H^1(\Ft,\Z/2) \cong \Z/2 \oplus \Z/2,
\]  
under the isomorphism
\begin{align*}
\Z/2 &\xrightarrow{~\cong~} \{\pm1\} \\
n &\xmapsto{\phantom{~\cong~}} (-1)^n. \end{align*}
The group $\Ou\cong\PGLtZ$ acts on the three-element set of nonzero elements, by the homomorphism
\[
\Ou\cong\PGLtZ \twoheadrightarrow \mathsf{PGL}(2,\Z/2) \cong \SymThree. 
\]

This paper concerns actions which preserve $\Ht\subset\Hth$
but do not preserve orientation on $\Ht$. Since $X,Y$ generate
$\Ft$, at least one of $\rho(X)$, $\rho(Y)$ reverse orientation.
Since
\begin{equation*}
\rho(Z) = \rho(Y)\inv\rho(X)\inv,
\end{equation*}
{\em exactly one\/}
of $\rho(X), \rho(Y), \rho(Z)$ preserves orientation.
Therefore three cases arise, which are all equivalent
under the cyclic permutation of $X,Y,Z$
(apparent from the presentation \eqref{eq:presentation}).
We reduce to the case that $\rho(Z)$ preserves orientation,
so that $\rho(X)$ and  $\rho(Y)$ each {\em reverse\/} orientation.
This is the case
defined in \eqref{eq:ExampleColoring}.
Then the  stabilizer of $\Coloring$ corresponds to the subgroup comprising  matrices
\[\bmatrix a & b \\ c & d \endbmatrix,\quad    a, b, c, d \in \Z,\quad   ab-cd = \pm 1 \]
where $a\equiv 1\mod 2$, $c \equiv 0\mod 2$.

Whether an element of $\Ft$ preserves or reverses orientation is detected
by the homomorphism $\Coloring$ defined in \S\ref{sec:Effect}.
In terms of traces, this means that $\xi = \tr\big(\rho(X)\big)$ and
$\eta = \tr\big(\rho(Y)\big)$ are purely imaginary
and that $\zeta = \tr\big(\rho(Z)\big)$ is real.
Thus we write
\begin{align*}
\tr\big(\rho(X)\big) & = ix \\
\tr\big(\rho(Y)\big) & = iy \\
\tr\big(\rho(Z)\big) & = z
\end{align*}
where $x,y,z\in\R$.
%

\subsection{Invariants of the action}
The function
\begin{equation}\label{eq:Kappa}
\kappa(\xi,\eta,\zeta) := \xi^2 + \eta^2 + \zeta^2 - \xi\eta\zeta -2
\end{equation}
arises as the
trace of the commutator $[X,Y] := XYX\inv Y\inv$:
\begin{equation*}
\tr\big(\rho([X,Y])\big) = \kappa(\xi,\eta,\zeta).
\end{equation*}
By Nielsen's theorem~\cite{nielsen1917} that $\Au$ preserves $[X,Y]$ up to
conjugacy and inversion (and the fact that $\tr(a) = \tr(a\inv)$
for $a\in\SLtC$), the function $\kappa$  is $\Au$-invariant.

The action also preserves the exterior  $3$-form
$d\xi \wedge d\eta \wedge d\zeta$ on $\C^3$ and its
dual exterior trivector field
$\dd{\xi} \wedge \dd{\eta} \wedge \dd{\zeta}$ on $\C^3$.
Interior product of this trivector field with the $1$-form
$d\kappa$ defines a bivector field
\begin{equation}\label{eq:PoissonStructure}
\BB := 
(2\zeta -\xi\eta) \ 
\dd{\xi}\wedge\dd{\eta} \ + \
(2\xi - \eta\zeta) \ 
\dd{\eta}\wedge\dd{\zeta}\  + \ 
(2\eta - \zeta\xi) \ 
\dd{\zeta}\wedge\dd{\xi}
\end{equation}
which defines a complex-symplectic structure on each level
set $\kappa\inv(t)$.
This bivector is $\hGamma$-invariant in the sense that if $\gamma\in\hGamma$,
then 
\begin{equation}\label{eq:InvarianceOfB} 
\gamma_* \BB = \Det(\gamma) \BB \end{equation}
where $\Det$ is the homomorphism defined by \eqref{eq:hGammaDet}.

We now restrict $(\kappa,\BB)$ to the real forms of $\RepF$ defined 
in \S\ref{sec:RealForms}.
Under the substitution
\[ \bmatrix  \xi \\ \eta \\ \zeta \endbmatrix := 
\bmatrix  i x \\ i y \\ z \endbmatrix \]
the restriction of $\kappa$ to $i\R\times i\R \times \R$ equals:
\begin{equation}\label{eq:KappaUpsilon}
\kC(x,y,z) := \kappa(i x,i y, z) = -x^2 - y^2 + z^2 +  xyz - 2
\end{equation}
where $\kappa$ is defined in \eqref{eq:Kappa}.
Similarly the restriction of $\BB$ equals:
\begin{equation}\label{eq:RealPoissonStructure}
\BB_\Coloring := 
(2z + xy) \ 
\dd{x}\wedge\dd{y} \ + \
(-2 x + y z) \  
\dd{y}\wedge\dd{z}\  + \ 
(-2 y + z x) \ 
\dd{z}\wedge\dd{x}.
\end{equation}
The action of $\Gamma$ restricts to a polynomial action
on $i\R\times i\R \times\R$, which preserves the function $\kappa_\Coloring$.
This action preserves area on each level set $\kCk$ 
with respect to the area form induced by \eqref{eq:PoissonStructure},
in the sense of \eqref{eq:InvarianceOfB}.
It is with respect to these measures that we studied
the ergodicity of the group actions in \cite{MR2026539}.

We first describe the topology of the level sets $\kCk$ and then relate the dynamical system
defined by the $\Gamma$-action to geometric structures on $C_{1,1}$ and $C_{0,2}$.

\section{Topology of the imaginary commutator trace}
\noindent
In this section we consider the {\em purely imaginary real form\/}
\begin{align*}
\RepF^\Coloring & \cong
\R^3 \xrightarrow{\cong}  i\R\times i\R \times \R \\ &\subset
\C^3 \cong \RepF \end{align*}
of $\RepF$ associated to the $\{\pm 1\}$-character $\Coloring$
as defined in \S\ref{sec:RealForms}.
The restriction of the commutator trace function  is the cubic polynomial 
$\R^3 \xrightarrow{~\kC~}\R$ defined in \eqref{eq:KappaUpsilon}.
In this section we analyze the topology of the level sets $\kCk$ for $k\in\R$:

\begin{thm4*}\label{thm:TopologyOf LevelSets}
\begin{enumerate}
\item
Suppose $k>-2$. 
\begin{itemize}
\item $\kCk$ is a smooth surface with two connected components.
\item The restriction of the projection 
\[\kCk\xrightarrow{~\Pixy~}\R^2\] to each component is a diffeomorphism onto $\R^2$.
\item The components of  $\kCk$ are graphs of  functions
$\R^2 \xrightarrow{~z_\pm~} \R$.
\item
The sign-changes $\sigma_1, \sigma_2$ permute the two components.
\item 
When $k\ge 2$, the region $\R\times\R\times[-2,2]$ separates the two components of
$\kCk$.
\end{itemize}
\item $\kCminusTwo$ is connected, with one singular point $(0,0,0)$, and
$\sigma_1, \sigma_2$ permute the two components
of $\kCminusTwo\setminus\{(0,0,0)\}$.
\item
Finally, when $k<-2$, the level set $\kCk$ is smooth and connected. 
The restriction of the projection $\Pixy$ to $\kCk$ is a branched double covering of the 
complement of the open topological $2$-disc
\[
\{ (x,y)\in\R^2 \mid (x^2+4)(y^2+4) < 4 (2-k) \}. \]
\end{enumerate}
\end{thm4*}
\noindent 
The only value of $k$ for which the level set $\kCk$ admits a rational parametrization is $k=2$;
compare \S\ref{sec:ImagCharkEqualTwo} for more details.
Figure~\ref{fig:ContoursXY} depicts the contours of the function $(x^2+4)(y^2+4)$.

\subsection{Preliminaries}

Let $\Pixy$ denote projection to the $xy$-plane:
\begin{align}\label{eq:projection}
\R^3 &\xrightarrow{~\Pixy~} \R^2 \notag\\
(x,y,z)& \longmapsto (x,y). \end{align}
Let $ \QQ_z(x,y)  :=   x^2 - z xy + y^2.$

\begin{lem}\label{lem:PosDef}
For $\vert z\vert < 2$,
 the quadratic form $\QQ_z$ is positive definite.
 \end{lem}
\begin{proof}
Write
\begin{equation}\label{eq:Qz}
\QQ_z(x,y) =
\frac{2+z}4 (x-y)^2 +
\frac{2-z}4  (x+y)^2. \end{equation}
Since $\vert z\vert < 2$, this is a positive linear combination of two squares.
\end{proof}
\begin{lem}\label{lem:CriticalPointsOfKappaColoring}
$(0,0,0)$ is the only critical point of $\R^3\xrightarrow{~\kC~}\R$.
\end{lem}
\begin{proof}
Since
\begin{equation}\label{eq:DifferentialOfKappaColoring}
d\kC = (-2 x + yz ) dx + (-2 y + z x ) dy + (2 z + x y) dz,\end{equation}
the point $p= (x,y,z)\in\R^3$ is critical if and only if
\begin{align}
-2 x + y z &= 0 \label{eq:ddx}\\
-2 y + z x &= 0 \label{eq:ddy}\\
2 z + x y &= 0 \label{eq:ddz}. \end{align}

First suppose that $z\neq 0$. 
Then \eqref{eq:ddz} implies $x y \neq 0$, that is,
$x\neq 0$ and $y\neq 0$. 
Apply \eqref{eq:ddx} and \eqref{eq:ddy} to obtain:
\[ \frac{x}{y} = \frac{y}{x} = \frac{z}2, \]
whence $x/y = \pm 1$, that is, $x = \pm y$ and $z=\pm 2$.
Now \eqref{eq:ddz} implies $x^2 = -4$, contradicting $x\in\R$.

Thus $z=0$. By \eqref{eq:ddx}, $x=0$ and
by \eqref{eq:ddy}, $y=0$. Thus $p=(0,0,0)$ as desired.
\end{proof}

\subsection{Projection when $k>-2$}

\noindent
The proof crucially uses: 
\begin{lem}\label{lem:BasisTangentSpace}
Suppose that  $\kC(x,y,z) = k$ and $2z + xy \neq 0$. 
Then the vector fields
\begin{equation}\label{eq:BasisTangentSpace}
\tdd{x}\  := \ 
\dd{x}\  +\  \frac{2 x - y z}{2 z + x y}\   \dd{z},
\qquad
\tdd{y}\ := \ 
\dd{y}\ +\ \frac{2 y - x z}{2 z + x y}\   \dd{z}
\end{equation}
form a basis of the tangent space $T_{(x,y,z)}\big(\kCk\big)$.
Furthermore 
$\tdd{x}, \ \tdd{y}$ project under
$d\Pixy$ to the coordinate basis 
\[\dd{x},\ \dd{y} \in T_{(x,y)}\R^2.\]
\end{lem}
\begin{proof}
Since
\begin{equation}\label{eq:DiffdKc} 
d\kC = 
(-2 x + y z)\  dx + (-2 y + z x)\  dy + (2 z + x y)\  dz, \end{equation}
the vector fields $\tdd{x}, \tdd{y}$
lie in the kernel of $d\kC$. 
Therefore they are tangent to $\kCk$. 
Their projections under $d\Pixy$  form the coordinate basis
of $T_{(x,y)}\R^2 \cong \R^2$. 
Thus they are linearly independent and base 
\[ T_{(x,y,z)}\big(\kCk\big)  = \Ker\big(d\kC\big) \] 
as claimed.
\end{proof}

\begin{prop}\label{prop:SmoothSurface}
$\kCk$ is a smooth surface and the restriction of $\Pixy$ is a local diffeomorphism.
\end{prop}

\begin{proof}
Since $\kC(0,0,0) = -2< k$, Lemma~\ref{lem:CriticalPointsOfKappaColoring}
implies that $\kCk$ contains no critical points of $\kC$.
Thus $\kCk\subset\R^3$ is a smooth surface.

If $2 z + x y \neq 0$, then 
Lemma~\ref{lem:BasisTangentSpace}
implies that $\Pixy$ restricts to a diffeomorphism.
Thus it suffices to show  $2 z + x y \neq 0$.

To this end, suppose $2 z + x y = 0$. 
Then $\kC(x,y,z) = k$ implies:
\[
(2 z + x y)^2 = (x^2 + 4)(y^2 + 4) + 4 (k-2).  \]
Since $k > -2$, 
\begin{align*}
0 &= (2 z + x y)^2 = (x^2 + 4)(y^2 + 4) + 4 (k-2) \\ 
& \ge 16 + 4 (k-2)  = 4 (k+2) > 0, \end{align*}
a contradiction. 
\end{proof}

\begin{prop}\label{prop:twocomponents}
Suppose that $k> -2$. 
Then the plane $\R\times\R\times\{0\}$ separates $\kCk\subset\R^3$ 
into two components.
\begin{itemize}
\item Each component projects diffeomorphically onto $\R^2$ under $\Pixy$.
\item Each of the sign-changes $\sigma_1,\sigma_2$ permutes the two components.
\item When $k > 2$, 
the  region $\R\times\R\times [-2,2]$ separates the two components of $\kCk\subset\R^3$. 
\item 
The open region $\R\times\R\times (-2,2)$ separates  the two components of $\kCtwo$.
\end{itemize}
\end{prop}
\begin{proof}
The identity
\begin{equation}\label{eq:KappaUpsilonInTermsOfz} 
\kC(x,y,z)  = z^2 - 2 - \QQ_z(x,y),
\end{equation}
implies (since $k > -2$),
\[ z^2 = k + 2 + \QQ_z(x,y) > \QQ_z(x,y). \]
If $\vert z\vert < 2$, then Lemma~\ref{lem:PosDef} implies
that $\QQ_z(x,y) \ge 0$ and $ z^2 > 0$. 

Thus $z\neq 0$. 
Consequently $\kCk\subset\R^3$ is the disjoint union of the two
subsets where $z> 0$ and $z< 0$ respectively.
These two subsets are interchanged by $\sigma_i$.

Suppose that $k>2$. 
If $\vert z\vert \le 2$, then, as above, Lemma~\ref{lem:PosDef} implies
$\QQ_z(x,y) \ge 0$, and 
\[ z^2 = k + 2 + \QQ_z(x,y) \ge  k+2 >  4, \]
contradicting $\vert z\vert \le 2$.
Thus $\R\times\R\times [-2,2]$ separates the two components of $\kCk$.
Proposition~\ref{prop:SmoothSurface} implies $\Pixy$ restricts to a diffeomorphism, as claimed.
The case $k=2$ is handled similarly.
\end{proof}
Writing
\begin{equation*}
\kC(x,y,z) =\Big(z + \frac{xy}2\Big)^2 -  \frac{(x^2+4)(y^2+4)}4  + 2,  
\end{equation*}
the two solutions of  $\kC(x,y,z) = k$ are:
\begin{equation}\label{eq:zPlusOrMinus} 
z = z_{\pm}(x,y) :=  \frac{-xy \pm \sqrt{(x^2+4)(y^2+4) + 4(k-2) }}2. \end{equation}
Since $\vert z\vert > 2$, and
\begin{align*}
-xy + \sqrt{(x^2+4)(y^2+4)+4(k-2)} & >  0, \\
-xy - \sqrt{(x^2+4)(y^2+4) +4(k-2)} & <  0,
\end{align*}
the solution $z=z_+(x,y)$ satisfies $z > 2$ and the solution $z=z_-(x,y)$ satisfies $z < -2$.
Therefore $\kCk$ consists of two components,
which are, respectively, the graphs of the functions $z_+, z_-$ defined in \eqref{eq:zPlusOrMinus}.

Solving \eqref{eq:KappaUpsilonInTermsOfz}, the level set
\[ \kCk \cap 
\big( \R\times\R\times\{z\}\big) \]
is a (nondegenerate) hyperbola for $z \neq \pm \sqrt{k+2}$
and the union of two crossing lines (a degenerate hyperbola) for
$z = \pm \sqrt{k+2}$.

First suppose that $-2<k<2$.
Then 
\begin{equation}\label{eq:WhySolnsDistinct}
(x^2 + 4)(y^2 + 4) \ge 16 > 4 (2 -k) \end{equation}
since $k>-2$. 

Suppose that $(x,y)\in\R^2$.
Then the preimages $\Pixy^{-1}(x,y)$ are points
$(x,y,z)$ where $z$ is one of the two solutions
$z = z_\pm(x,y)$ given by \eqref{eq:zPlusOrMinus}.
(These two solutions are distinct by \eqref{eq:WhySolnsDistinct}.)

Thus 
\[ 
\kCk = \graph(z_+) \coprod  \graph(z_-)  \]
and the sign-change
\begin{equation}\label{eq:xzSignChange}
\bmatrix x \\ y \\ z \endbmatrix \xmapsto{~\sigma_2~}
\bmatrix -x \\ y \\ -z \endbmatrix \end{equation}
interchanges these two preimages.

For $-2 < z < 2$, 
Lemma~\ref{lem:PosDef} implies that
the quadratic form $\QQ_z$ is positive definite.
As in the proof of Proposition~\ref{prop:twocomponents} again,
\eqref{eq:KappaUpsilonInTermsOfz} expresses 
$\kCk$ as a level set of $\QQ_z$ for fixed $z$.
Thus the nonempty level sets 
\[
 \kCk \cap z^{-1}(z_0) \]
are ellipses whenever $-2 < z_0 < 2$.

Now suppose $-2 < k$ as before. 
Since $\sqrt{2+k} < 2$, the expression \eqref{eq:KappaUpsilonInTermsOfz} 
of $\kCk$ in terms of $\QQ_z$ implies that
$\kCk$ does not intersect the level set $z=z_0$ when 
\[ \vert z_0\vert < \sqrt{2+k} < 2.\]
Thus the region $\R\times\R\times \big(-\sqrt{2+k},\sqrt{2+k}\big)$ separates
the two components of $\kCk$, 
which are interchanged by the involution 
$\sigma_2$ defined in \eqref{eq:xzSignChange}.

\subsection{The invariant area form}

Proposition~\ref{prop:twocomponents} implies that the coordinate functions
$x,y$ give global coordinates on $\kCk$ when $k > -2$. 
In these coordinates the invariant area form has the following simple expression:
\begin{equation}\label{eq:AreaForm}
dA_k\ :=\ 
\frac{\mathsf{sign}(z)\  dx \wedge dy}{\sqrt{ (x^2+4)(y^2+4) + 4(k-2)}}\end{equation}
\begin{lem}\label{lem:AreaForm}
The exterior $2$-form $dA_k$  defined in \eqref{eq:AreaForm} defines
a real-analytic $\Gamma$-invariant area form on $\kCk$. 
\end{lem}
\begin{proof}
Since $k>-2$ and $x,y\in\R$,
\[
(x^2+4)(y^2+4) + 4(k-2) \ge 4 (k+2) > 0 \]
so $dA_k$ is a real-analytic area form.

Since the complex exterior  bivector field $\BB$
on $\RepF$ defined by \eqref{eq:PoissonStructure}
is $\hGamma$-invariant, it suffices to show
that, on the real subvariety $\kCk$, the complex bivector field $\BB$ restricts to the real bivector field \begin{equation}\label{eq:AreaFormDual}
(dA_k)\inv\ :=\   
\mathsf{sign}(z)\  \sqrt{ (x^2+4)(y^2+4) + 4(k-2) }\ \dd{x} \wedge \dd{y}
\end{equation}
dual to $dA_k$. That is, we prove the embedding
\[
\R^2 \xrightarrow{~(\Pixy)\inv~} \kCk \subset \R^3 \]
pushes $(dA_k)\inv$ forward to $\BB_\Coloring$.

To this end, \eqref{eq:RealPoissonStructure} implies that $d\Pixy$ maps the Hamiltonian vector fields
\begin{alignat*}{2}
\Ham(x) &= -(2 z +  x y)\  \dd{y} +   (x z - 2y)\  \dd{z} 
&& 
= -(2z + xy)\  \tdd{y} \\
\Ham(y) &= (2 z +  x y)\  \dd{x} +   (2 x - y z)\  \dd{z} 
&& 
= (2z + xy) \tdd{x}  \end{alignat*}
to $-(2z+xy)\ \dd{y}$ and $(2z+xy)\  \dd{x}$ respectively.
Now \eqref{eq:zPlusOrMinus} implies
\[
2 z + xy =  \mathsf{sign}(z)\  \sqrt{ (x^2+4)(y^2+4) + 4(k-2) }, \]
completing the calculation.
\end{proof}
\noindent Similar calculations show:
\begin{align}\label{eq:Hamz}
\Ham(z) & = (2 y - x z )\ \dd{x}  + (y z - 2 x) \ \dd{y}  \\ &= 
(2 y - x z ) \ \tdd{x}  + (y z - 2 x)\  \tdd{y} .\notag
\end{align}

%


\subsection{The level set for  $k\le -2$}
Here more interesting topologies arise. 

\begin{prop}\label{prop:twocomponentsorannulus}
Suppose that $k\le-2$. 
\begin{itemize}
\item If $k=-2$, then $\kCminusTwo$ is connected, with one singular point  $(0,0,0)$.  
Each of the two components of $\kCk\setminus\{(0,0,0)\}$ 
projects diffeomorphically to $\R^2\setminus\{(0,0)\}$ under $\Pixy$.
 \item If $k<-2$, then $\kCk$ is homeomorphic to a cylinder. 
\end{itemize}
\end{prop}

\begin{figure}[htb]
\includegraphics[scale=.5]{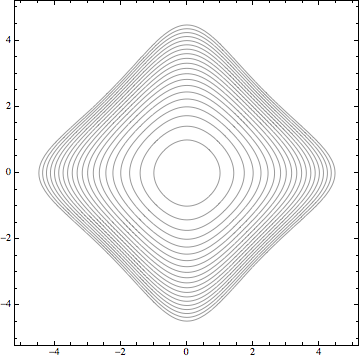}
\caption{Contours of $(x^2+4)(y^2+4)$ on $\R^2$ }
\label{fig:ContoursXY}
\end{figure}

\begin{proof}
%
%
%
%

When $k=-2$, the preimage $\kC^{-1}(-2)$ 
is defined by the equation
\[ \QQ_z(x,y) = z^2 \] 
which is singular only at $(0,0,0)$.
The intersection $\kC^{-1}(-2)\setminus\R^2\times\{0\}$ 
consists of two components, each being the graph of 
one of the functions $z_+, z_-$ defined in \eqref{eq:zPlusOrMinus}.
These two components of the complement of the origin are interchanged by $\sigma_2$.
This {\em purely imaginary real form\/} of the Markoff surface
plays an important role in Bowditch's paper~\cite{MR1643429}.

Finally consider the case $k < -2$. 
In this case the image $\Pixy\big(\kCk)\big)$ 
is the region of $\R^2$ defined by 
\[
\bigg( 1 + \Big(\frac{x}2\Big)^2 \bigg)
\bigg( 1 + \Big(\frac{y}2\Big)^2 \bigg)
\ge \frac{2-k}4 > 1.
\]
(Figure~\ref{fig:ContoursXY} depicts the level sets of this function.)
For $-2< z_0< 2$, the ellipses defined by $z= z_0$ form a cylinder.
As $ z_0 \rightarrow \pm 2$, their eccentricities tend to $\infty$,
and the limits (defined by $z= \pm 2$ 
are degenerate conics consisting of two parallel lines
 \[
(x \mp y)^2 = 2-k. \]
For $\vert z_0\vert > 2$, the level sets are hyperbolas. 
The entire level set $\kC^{-1}(-2)$ is homeomorphic to a cylinder.
\end{proof}

\section{Generalized Fricke spaces}
Hyperbolic structures on a surface $\Sigma$ with fundamental group $\Ft$
determine representations of $\Ft$ into the isometry group of $\Ht$.
This group identifies with the disconnected group $\PGLtR$,
and as before we lift to the double covering which we may identify with
\begin{equation*}
\SLtR \bigcup  \big( i \GLtR \cap \SLtC \big).
\end{equation*}
The homomorphism $\Coloring$ is the composition
\[ 
\Ft \xrightarrow{~\rho~} \PGLtR \twoheadrightarrow \pi_0\big(\PGLtR\big) \cong \Z/2. \]
It is trivial if and only if  $\Sigma$ is orientable.
In that case $\rho$ preserves orientation on $\Ht$, 
and the lift lies in $\SLtR$. 
This case is analyzed in \cite{MR2026539}. 
Otherwise, exactly one of $\rho(X)$,
$\rho(Y)$ and $\rho(Z)$ maps to $\SLtR$; by applying a permutation
in $\Au$, we may assume that $\rho(Z)\in\SLtR$,
that is, $\Coloring$ is the $\{\pm 1\}$-character with 
\[ \Coloring(X) = \Coloring(Y) = 1 \mod 2, \qquad \Coloring(Z) = 0 \mod 2. \]

In this case, if the structure is complete, the
representation is a {\em discrete embedding,\/} that is, an isomorphism
of $\Ft$ onto a discrete subgroup of $\PGLtR$ and  the hyperbolic
surface is the quotient $\Ht/\rho(\Ft)$.

There are two cases for the topology of the surface $\Ht/\rho(\Ft)$:
either it is homeomorphic to a $1$-holed Klein bottle $C_{1,1}$,
which corresponds to the case when $k > 2$, or a $2$-holed cross-surface
$C_{0,2}$, which corresponds to the case when $k < 2$. 
(See Charette-Drumm-Goldman~\cite{CDG} for more information
on the Fricke space of $C_{0,2}$.)
In the former case, we consider more general geometric structures, 
the boundary component of $C_{1,1}$ is replaced by a cone point;
from our viewpoint it is more natural to allow such an isolated singularity.
In this particular case the singular hyperbolic surfaces behave dynamically
like nonsingular complete hyperbolic structures on $C_{1,1}$.
\bigskip

\begin{tabular}{|c|c|c|}
  \hline
Surface $\Sigma$ & Peripheral elements & Boundary traces \\
  \hline
$\Sigma_{0,3}$ & $X, Y, Z :=(XY)^{-1}$ & $x, y, z$ \\
  \hline
$\Sigma_{1,1}$ & $K := XYX^{-1}Y^{-1}$ &  $ k:= x^2 + y^2 + z^2 + x y z - 2$\\
  \hline
$C_{0,2}$ & $Z := XY, \quad Z' := Y^{-1}X$ 
& $ z,\  z' := (i x) (i y)  - z$\\
  \hline
$C_{1,1}$ & $D := X^2 Y^2$ & $ \delta := 2 - (i x)^2 - (i y)^2 + (i x) (i y) z$\\
  \hline
\end{tabular}
\bigskip

\subsection{Geometric structures on the two-holed cross-surface}
We begin with analyzing the Fricke space $\Fricke(C_{0,2})$ of marked hyperbolic structures on $C_{0,2}$ in trace coordinates. 
Later,  in \S\ref{sec:ImagCharkLessThanTwo}, 
we characterize the corresponding  characters dynamically 
as those for which the flow 
possesses an attractor (see Definition~\ref{defn:Attractor}, 
in the case when $k < 2$).
Compare Goldman~\cite{MR2497777} and Charette-Drumm-Goldman~\cite{MR3180618} 
for further discussion.

\begin{prop}
Suppose $\rho$ is the holonomy representation for a hyperbolic structure $M$
on $C_{0,2}$. 
Then for some choice of basis for $\pi_1(M)$, 
the character of $\rho$ satisfies
\begin{align}\label{eq:2hx}
z & \le -2 \notag\\
z' &= - xy - z \le -2 . 
\end{align}
Conversely, if $(ix,iy,z)$ is a purely imaginary character satisfying
\eqref{eq:2hx}, then it corresponds to a complete hyperbolic structure
on $C_{0,2}$.
\end{prop}
\begin{proof}
Suppose $M$ is a complete hyperbolic surface diffeomorphic to $C_{0,2}$.
Let $X$ be a one-sided simple curve on $M$ and let $M|X$ be the hyperbolic
surface obtained by splitting $M$ along the closed geodesic corresponding to $X$. 
Since $X$ is one-sided, $M|X$ is connected and has one more boundary
component than $M$. Furthermore $\chi(M|X) = \chi(M) = -1$ and 
\[ \#\pi_0\big(\partial(M|X)\big) = \#\pi_0(\partial M) + 1 = 3. \]
By the classification of surfaces $M|X \approx \Sigma_{0,3}$.
Thus $M$ is obtained from the three-holed sphere $M|X$ 
by attaching a cross-cap to a boundary component $C\subset \partial(M|X)$.
Let $A,B$ denote the other two components of $\partial(M|X)$, 
which correspond to the components of $\partial M$.

The Fricke space of $\Sigma_{0,3}$ is parametrized by the three boundary traces 
$(a,b,c)$ where $a = \tr\big(\rho(A)\big)$, $b = \tr\big(\rho(B)\big)$ correspond
to the boundary components of $M$ and $C$ corresponds to $X^{-2}$. 
Since 
\begin{align*}
 \tr\big(\rho(C)\big) = \tr\big(\rho(X)^2)\big)&  = \tr\big(\rho(X)\big)^2 - \tr (\Id) \\ & =
-x^2 -2\; < -2, \end{align*}
the other two boundary traces $\tr\big(\rho(A)\big)$, 
$\tr\big(\rho(A)\big)$ have the same sign. 
Applying a sign-change automorphism if necessary, 
we can assume they are both negative. 

Let $Y$ be another one-sided simple curve with $i(X,Y) = 1$. 
Then, with appropriate choices of basepoint and representative curves,
we may assume that $A = XY$ and $B = Y^{-1}X$ (so that $ABC = I$). 
Their traces satisfy:
\[ a + b = (ix) (iy) = -xy. \]
These traces correspond to the traces $z, z'$ of $\rho(Z) = \rho(Y\inv X\inv)$ and
$\rho(Z') = \rho(Y X\inv)$, since 
\begin{align*} a & = \tr\big(\rho(A)\big)  = \tr\big(\rho(XY)\big)  \quad = z, \\ 
b & = \tr\big(\rho(B)\big)  = \tr\big(\rho(X\inv Y)\big)  = z', \end{align*}
as desired. 

Conversely, suppose that $(ix,iy,z)$ is a purely imaginary character satisfying
\eqref{eq:2hx}.  
Define \[ z' := - z - x y. \] 
Since 
\begin{align*} z & \le -2, \\
z' &\le -2, \\ 
-x^2 - 2 &< -2, \end{align*} 
there exists a complete hyperbolic surface $N$ diffeomorphic to 
$\Sigma_{0,3}$ with boundary traces $(-x^2 - 2, z, z')$.
The boundary components $A,B,C$ correspond to traces 
\begin{align*} a &= z, \\ b &= z', \\ c &= -x^2 -2. \end{align*}

Let $M$ be the complete hyperbolic surface obtained by attaching a cross-cap to $C$. 
Then $\pi_1(M)$ is obtained from \[\pi_1(N) = \langle A, B, C \mid A B C = I\rangle\]
by adjoining a generator $X$ satisfying $C = X^{-2}$. 
Define $Y := X^{-1} A$, so that
\begin{align*} XY &= A, \\ Y^{-1}X  &=  B \end{align*}
and the holonomy of $M$ has character $(ix,iy,z)$ as desired.
\end{proof}

\subsection{Geometric structures on the one-holed Klein bottle}
\label{sec:OHKB}
We now determine the Fricke space and generalized Fricke space for
$C_{1,1}$ in trace coordinates. 
The surface $C_{1,1}$ is rather special in that it contains a unique isotopy class
of nontrivial nonperipheral $2$-sided simple closed curve. 
We consider a marking of $C_{1,1}$ in which this curve has homotopy class $Z$,
where $(X,Y,Z)$ is a basic triple.
Represent $X$ and $Y$ by $1$-sided simple loops
intersecting tangentially at the basepoint (so $i(X,Y) = 0$).
With respect to a suitable choice of basepoint and defining loops,
$D = X^2 Y^2$ is an element of $\pi_1(C_{1,1})$ generating 
$\pi_1(\partial C_{1,1})$. For a given representation $\rho$
defining a hyperbolic structure on $C_{1,1}$, we denote the corresponding
trace function by 
\[ \delta = \tr\big(\rho(D)\big). \]
The basic trace identity implies that
\[ \delta + k = z^2. \]

Suppose that $\mu \longleftrightarrow (ix,iy,z)\in i\R\times i\R \times \R$ is a character with $k>2$, which corresponds to the holonomy of a hyperbolic structure on $\Coo$. 
Proposition~\ref{prop:twocomponents} implies that $\vert z \vert > 2$.
Furthermore both $x,y \neq 0$ since otherwise $0 > z^2 - (k+2) = \QQ_z(x,y) > 0$.

\begin{prop}\label{prop:GeneralizedFrickeSpaceOfCoo}
 Suppose that $k > 2$ and $\vert z \vert < \sqrt{k+2}$. 
Let \[ \delta := z^2 - k < 2.\] 
Then $\mu$ corresponds to the holonomy of a hyperbolic structure on $C_{1,1}$
\begin{itemize}
\item with geodesic boundary if $\delta < -2$, 
and the boundary has length $2 \cosh^{-1}(-\delta/2)$;
\item with a puncture if $\delta = -2$;
\item with a conical singularity of cone angle $\cos^{-1}(-\delta/2)$ if $-2<\delta < 2$.
\end{itemize}
\end{prop}
\begin{proof}
Depending on the three cases for $\delta$,
the character \[ (-2-x^2,-2-y^2,\delta)\] corresponds to the holonomy of
a hyperbolic surface $S \approx \Sigma_{0,3}$ where the boundary with trace $\delta$ 
either is a closed geodesic, a puncture, or a conical singularity,
respectively, depending on whether
$\delta < -2$, $\delta = -2$ or $2 > \delta > -2$ respectively.
Attaching a cross-cap to the boundary components corresponding to $X$ and $Y$
yields a hyperbolic surface $S'\approx C_{1,1}$ whose holonomy has character 
$(ix,iy,z)$. 
(Compare \cite{MR2497777} for a discussion of attaching cross-caps to geodesic boundary components of hyperbolic surfaces.)
\end{proof}
\begin{cor}\label{cor:GenFrSp}
 The generalized Fricke space of the one-holed Klein bottle
$C_{(1,1)}$ is defined by the inequalities
\begin{align*}
\vert z \vert \quad &> 2 \\
x^2 - z x y + y^2\  & < 0.
\end{align*}
Here the boundary trace $\delta$ and commutator trace $k$ are:
\begin{align*}
\delta & := \big( x^2 - z x y + y^2\big) + 2 \\
k &:= z^2  - \delta  \end{align*}
\end{cor}
The mapping class group of $C_{1,1}$ is realized as the subgroup of 
$\Au$ generated by the elliptic involution $\e$, the transposition $(12)\in\SymThree$ and the involution $\II_1$ defined in \eqref{eq:Involutions}. 
As $\e$ acts trivially on $\RepF$, the action on $\kCk$ 
of the subgroup of $\Gamma$ corresponding to 
$\Mod(C_{1,1})$ is generated by the transposition $\mathscr{P}_{(12)}$ and the Vieta involutions $\II_1$ and $\II_2$ defined in \S\ref{sec:FaithfulAction}.
Figure~\ref{fig:GenFrickeSpaces} depicts 
these spaces.

\subsection{Lines on the Markoff surface}\label{sec:MarkoffSurface}
Here we interpret lines on the imaginary real form on  the 
Markoff surface ($k=-2$) as Fricke spaces of a certain $2$-orbifold.
Lee and Sakuma~\cite{MR283209,MR3109863} 
studied the complex lines on the Markoff surface. We consider the real lines on the real form and show they contain
characters of non-injective representations with discrete image
and identify the quotient hyperbolic orbifold.

Recall that the Markoff surface is defined by $(\xi, \eta,\zeta)\in\C^3$
satisfying $\xi^2 + \eta^2 + \zeta^2 = \xi\eta\zeta$. 
We consider the lines defined by the vanishing of one of the coordinates,
for example $\xi = 0$. In that case the lines are defined by $\zeta = \pm i \eta$,
which yield imaginary characters $(0,iy, y)$, where $y\in\R$. 
We show that, when $y\ge 2$, these correspond to representations of $\Ft$
with discrete image.

Write $(\xi,\eta,\zeta) = \big(0,2 i\sinh(\ell/2), \sinh(\ell/2)\big)$,
so that $Y$ corresponds to a glide-reflection of
displacement $\ell$, or a reflection when $\ell = 0$. 
Representative matrices are, for example:
\begin{align*}
\rho(X) & = i \bmatrix -1 & 0 \\ 0 & 1 \endbmatrix, \\
\rho(Y) & = i \bmatrix 2\sinh(\ell/2) & 1 \\ 1 & 0 \endbmatrix, \\
\rho(Z)   & = \bmatrix 0 & -1 \\ -1 & 2\sinh(\ell/2) \endbmatrix \end{align*}
Then $\rho(X)$ represents the reflection and
$\rho(Y)$ represents a glide-reflection (or reflection when $\ell=0$),
whose respective fixed-point sets are:
\begin{align*}
\Fix\big(\rho(X)\big) & = \{0,\infty\} \\
\Fix\big(\rho(Y)\big) & = \{e^{\ell/2},-e^{-\ell/2}\} \end{align*}
If $\ell > 2\log(1+\sqrt{2})$, then $\rho(Z)$ is a transvection 
with
\[ \Fix\big(\rho(Z)\big)  = \Big\{-\sinh(\ell/2)\pm\sqrt{\cosh^2(\ell)-3/2}\Big\}.
 \]
When $\ell = 2\log(1+\sqrt{2})$, then $\rho(Z)$ is a parabolic fixing $-1$.

In these latter two cases, $\Image(\rho)$ is discrete, and the quotient
hyperbolic orbifold $\Ht/\Image(\rho)$ is a punctured disc with a mirrored
arc on its boundary. Its orbifold double covering-space 
\[  \Ht/\langle \rho(Y),\,\rho(X)\rho(Y)\rho(X)\rangle\] is a two-holed cross-surface $\Czt$
with reflection symmetry defined by the reflection $\rho(X)$.
When $\ell = 2\log(1+\sqrt{2})$, the corresponding character is $(2i,2i,2)$. 
The closed ray defined by $\ell\ge 
2\log(1+\sqrt{2})$ corresponds to the Fricke space of this nonorientable
hyperbolic orbifold $\Ht/\Image(\rho)$.

\section{Bowditch theory}

Bowditch~\cite{MR1643429} introduced objects
(which he called {\em Markoff maps\/})
\begin{equation*}
\O \xrightarrow{~\mu~} \C
\end{equation*}
which are equivalent to  characters in $\RepF\cong\C^3$
where the commutator of a pair of free generators has trace $-2$.
(Tan-Wong-Zhang~\cite{MR2370281} extended this to arbitrary
characters in $\RepF\cong\C^3$, calling them {\em generalized Markoff maps.\/})
Bowditch used Markoff maps to study the orbit of a
character under the action of the automorphism group $\Au$.
Departing from Bowditch's terminology, we 
call these maps {\em trace labelings.\/}  
We apply Bowditch's approach to the dynamics of the $\Gamma$-action 
on  {\em purely imaginary characters\/}
\[ 
(\xi, \eta, \zeta)\in i\R\times i\R \times \R \subset \C^3.\]  
This section reviews the theory in the more general setting of $\SLtC$-characters,
that is, when $(\xi, \eta, \zeta)\subset \C^3$ before specializing to the
purely imaginary characters.

Bowditch's method was used and extended by Tan, Wong and Zhang in
\cite{MR2247658,MR2191691,MR2405161} and Maloni-Palesi-Tan~\cite{MR3420542}
to study various aspects of the
dynamics of this action, and obtain variations of 
McShane's identities. See the above papers for more details.

Following Bowditch~\cite{MR1643429},
encode the dynamics of the action of $\Ou$ 
on characters, described by trace labelings  $\mu\in\Rep(\Ft,\SLtC)$, 
in terms of a dynamical system on the tree $\T$.
The character provides a way of directing $\T$, and the dynamics of the associated 
directed tree $\vT_\mu$ (the {\em flow\/}) is Bowditch's invariant. 

\subsection{The trace labeling associated to a character}

Suppose that 
\begin{equation*}
\Ft \xrightarrow{~\rho~} \SLtC
\end{equation*}
is a representation. 
Define the corresponding 
{\em trace labeling\/} $\mu_\rho$:
\begin{align*}
\O  &\xrightarrow{~\mu_\rho~} \C \\
\omega & \longmapsto \tr\big(\rho(\gamma_\omega)\big)
\end{align*}
where $\gamma_\omega$ is an element of $\Ft$ 
corresponding to $\omega$.
%
%
The trace labeling $\mu$ associated to a character in 
$\RepF\cong\C^3$
satisfies a equation for each edge and an equation for each vertex as follows.

First, fix the parameter 
\[ 
k = \kappa(\mu):=
\kappa(\xi, \eta, \zeta)
\]
throughout the discussion. Vertex Relation \eqref{eqn:vertexrelation} 
implies that $\kappa$  is constant on vertices of $\T$ and
Edge Relation \eqref{eq:edgerelation} provides the basic inductive step
for navigating through $\T$.
In particular, once $\mu$ is defined on the three regions around a vertex,
then \eqref{eq:edgerelation} forces a unique extension 
$\O  \to \C$.

Suppose that $e^{X,Y}(Z,Z') \in \E$ is an edge. Then
\begin{equation}\label{eq:edgerelation}
   \zeta+\zeta'=\xi\eta
\end{equation}
Suppose $v = v(X,Y,Z)\in \V$ is a vertex.
Then
\begin{equation}\label{eqn:vertexrelation}
    \xi^2+\eta^2+\zeta^2-\xi\eta\zeta -2\;=\; k
\end{equation}
where 
the lower case Greek letters represent the trace labeling of the region denoted by the corresponding upper case Roman letters.
Figure~\ref{fig:LabeledTree} depicts an example of a trace labeling.

The Edge Relation \eqref{eq:edgerelation} and the Vertex Relation
\eqref{eqn:vertexrelation} intimately relate. 
Fixing a pair $X,Y\in\O$ which abut one another, 
there are exactly {\em two\/} elements of $\O$ which abut $X$ and $Y$; 
these are $Z$ and $Z'$ respectively. 
Their corresponding traces $\zeta,\zeta'$ are the two roots of the quadratic equation
obtained from \eqref{eqn:vertexrelation} by fixing $\xi,\eta$.
This follows directly from \eqref{eq:edgerelation}.

The trace labeling $\mu$ also intimately relates to the $\Gamma$-action on the character variety $\RepF$.
In particular, the triple of trace values about each vertex is the image of $\mu$ under an element of $\Gamma$.
This is the action of $\Gamma$ on superbases described in 
\eqref{eq:ActionOnSuperbases}
where the superbasis determines a coordinate system $\RepF\cong\C^3$.

Conversely, for any $\gamma\in\Gamma$, 
the image $\gamma(\mu)$ is a triple of trace values around some vertex of the tree.

\begin{figure}
\includegraphics[width=\textwidth]{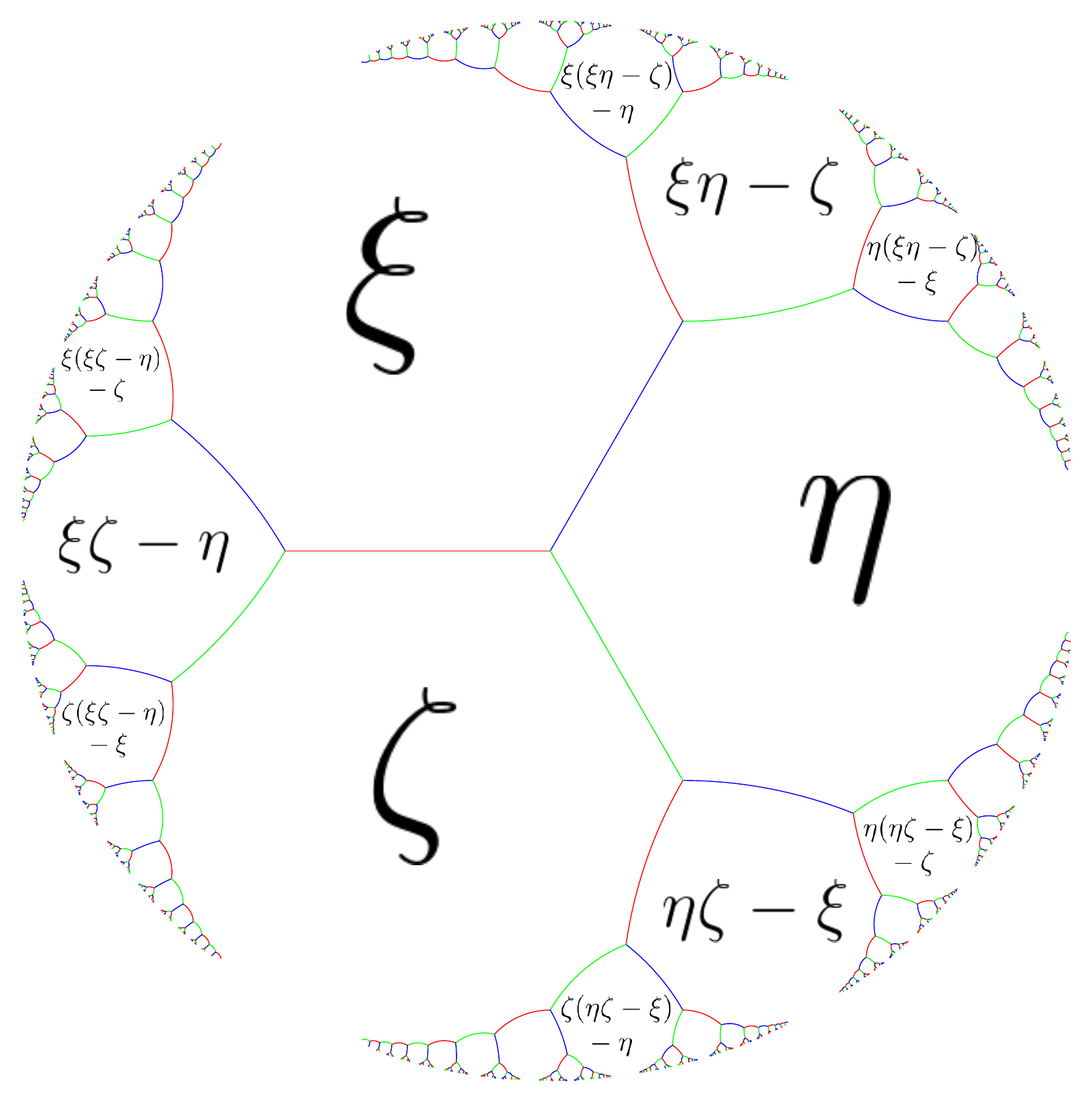}
\caption{Labeling the  trivalent tree by traces}
\label{fig:LabeledTree}
\end{figure}

\subsection{The flow associated to a character}
From the trace labeling $\mu$ associated to a character
in $\Rep(\Ft,\SLtC)$, 
Bowditch~\cite{MR1643429} associates a directed tree $\vT$ 
(which Tan-Wong-Zhang~\cite{MR2370281} call a {\em flow\/} on $\T$)
as follows.

Suppose that 
\[ e = e^{X,Y}(Z,Z')\in\E \]
is an edge.
Say that $e$ is {\em decisive\/} (relative to $\mu$) if
\begin{equation*}
\vert \zeta \vert  \neq \vert \zeta' \vert;
\end{equation*}
in that case direct the edge as follows:
\begin{equation*}
\begin{cases}
e^{X,Y}(Z\rightarrow Z') & :\Longleftrightarrow \vert \zeta \vert  > \vert \zeta' \vert \\
e^{X,Y}(Z\leftarrow Z')  & :\Longleftrightarrow \vert \zeta \vert  < \vert \zeta' \vert.
\end{cases}
\end{equation*}
Otherwise, say that $e$ is {\em indecisive\/} and assign either direction to the edge $e$.
Say that an edge $e$  {\em points decisively\/} towards a vertex $v$
if and only if $e$ is decisive and points towards $v$.

If every edge is decisive, say that $\mu$ is {\em completely decisive.\/}
If every edge is indecisive, say that $\mu$ is {\em completely indecisive.\/} 
For many cases of interest, only finitely many edges are indecisive.
The ambiguity at an indecisive edge does not affect the subsequent discussion. 
In particular, if the arrow 
points  %
from $Z$ to $Z'$, then $|\zeta| \ge |\zeta'|$.
Denote the above directed  tree by $\vT_{(\xi,\eta, \zeta)}$ or $\vT_{\mu}$.

Sign-change automorphisms do not affect the directed  tree 
$\vT_\mu$ associated to a character $\mu$.

\bigskip
The vertices $\Vt$ admit the following classification:

\bigskip
\begin{tabular}{|c|c|}
  \hline
  Number of edges pointing towards $v$ & Type of $v$ \\
  \hline
  0 & source \\
  \hline
  1 & fork \\
  \hline
  2 & merge \\
  \hline
  3 & sink \\
  \hline
\end{tabular}
\medskip\bigskip 
\begin{figure}
\centering
        \subfloat[][Source]{\includegraphics[scale=\scalefig]{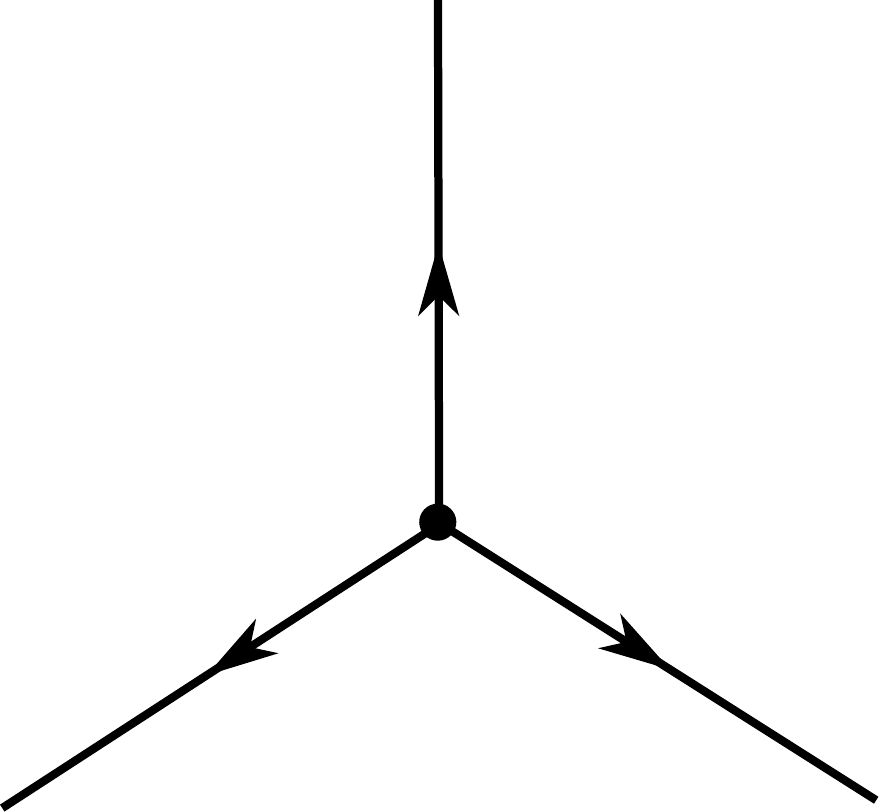}\label{F:Source}} \quad
        \subfloat[][Fork]{\includegraphics[scale=\scalefig]{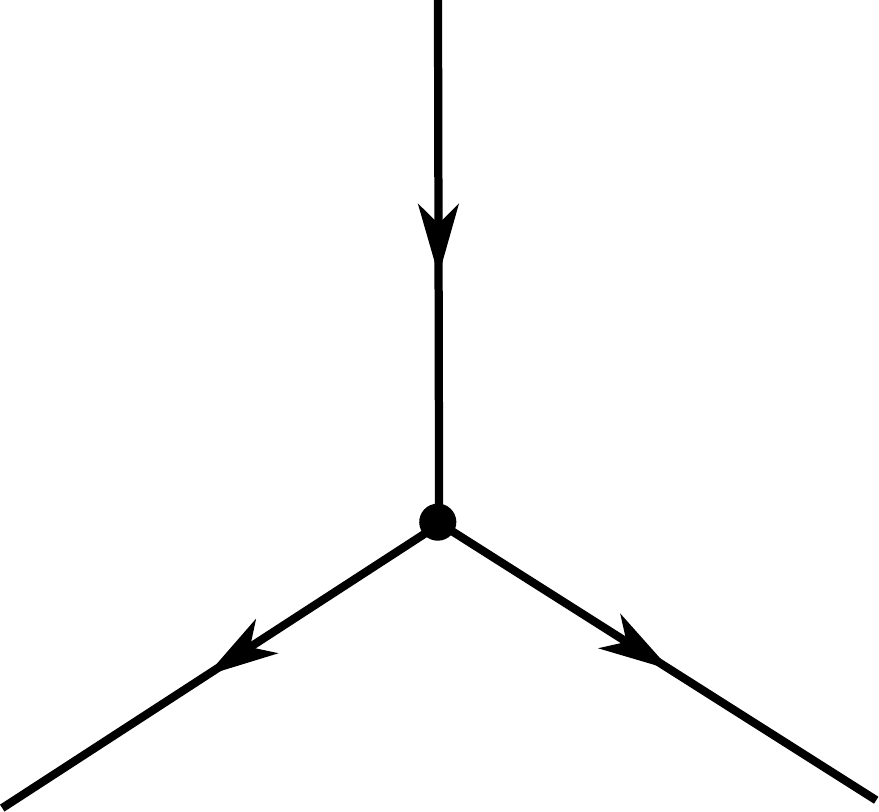}
\label{F:Fork}} \quad
        \subfloat[][Merge]{\includegraphics[scale=\scalefig]{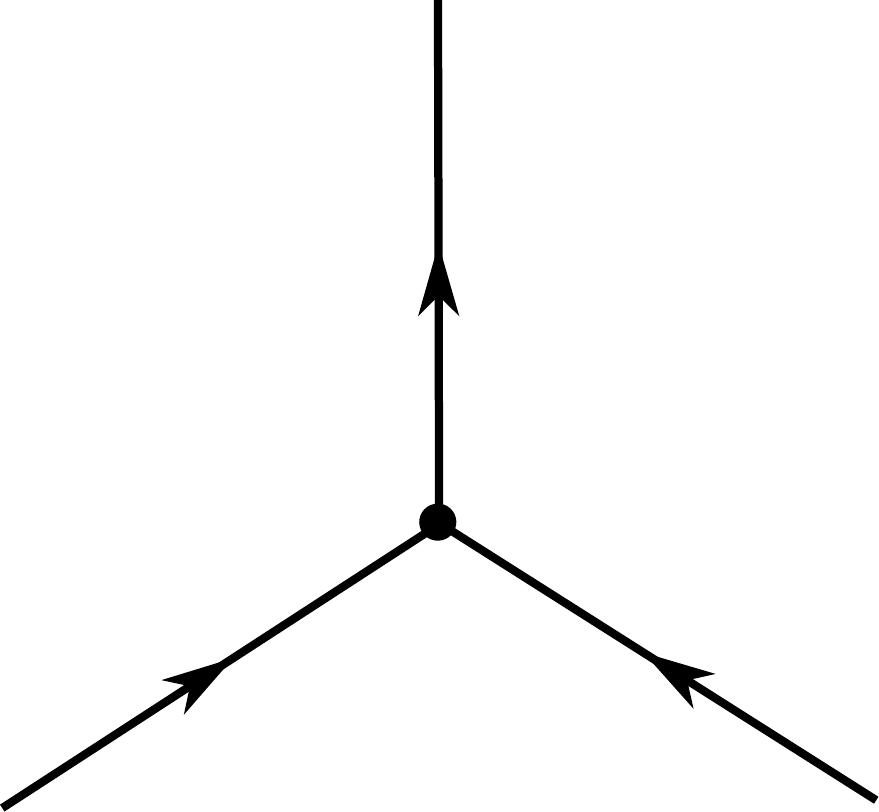}\label{F:Merge}} \quad
        \subfloat[][Sink]{\includegraphics[scale=\scalefig]{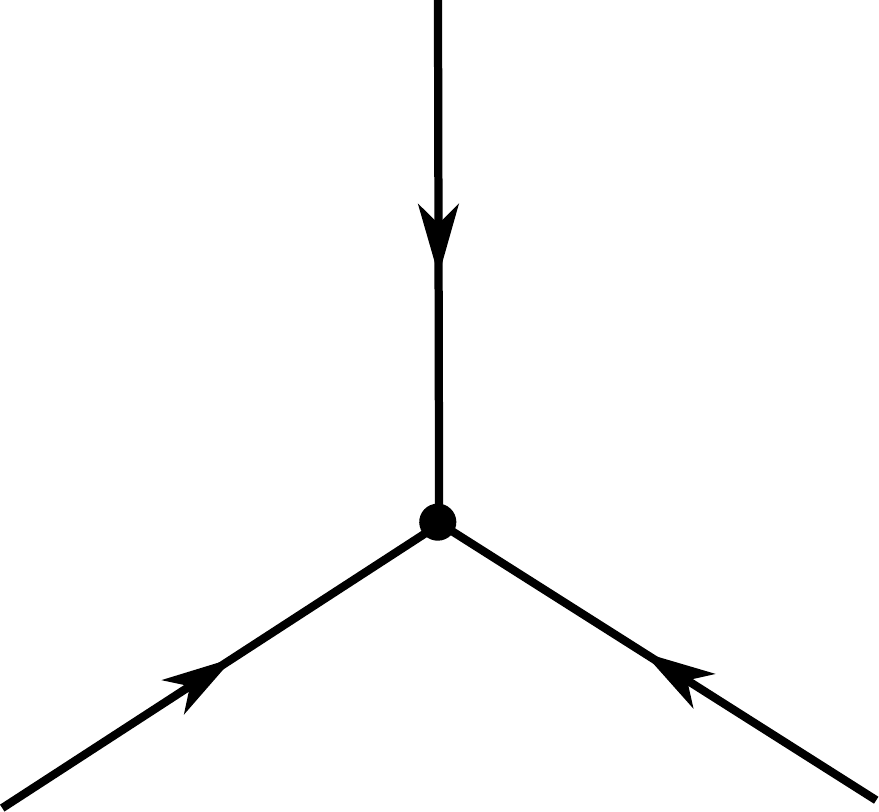}
\label{F:Sink}}
\caption{Types of vertices in a directed tree}
\end{figure}

\noindent
Each merge possesses a unique trace-reducing direction
since exactly one  directed edge incident to $v$ points away from $v$.

The following definition is due to Tan-Wong-Zhang~\cite{MR2370281}, \S 2.18 and
is a crucial idea in Bowditch~\cite{MR1643429}:

\begin{definition}\label{defn:Attractor}
Let $\vT$ be a directed tree. 
An {\em attractor\/} in $\vT$ is a minimal subtree $\vec{\mathsf{A}}\subset\vT$ satisfying
the following two properties:
\begin{itemize}
\item $\vec{\mathsf{A}}$ is finite;
\item Every edge $e \notin \vec{\mathsf{A}}$ points inward towards
$\vec{\mathsf{A}}$.
\end{itemize}
\end{definition}
\noindent
Examples of attractors include sinks 
and {\em attracting indecisive edges.\/}
For imaginary characters,
these are the only types of attractors.

\begin{prop}\label{prop:ProperAction}
$\Gamma$ acts properly on the set $\mathscr{S}$ of characters $\mu$ for which
$\vT_\mu$ has an attractor. 
\end{prop}
\begin{proof}
Denote by 
 $\mathsf{Finite}\big(\V\big)$ the set of finite subsets of $\V$.
Then the map 
\[\mathscr{S} \longrightarrow  \mathsf{Finite}\big(\V\big)\]
which associates to $\mu\in\mathscr{S}$ the attractor
of $\vT_\mu$ is $\Gamma$-equivariant.
Since $\Gamma$ acts properly on $\mathsf{Finite}\big(\V\big)$, 
it acts properly on $\mathscr{S}$.
\end{proof}
\begin{definition}
A {\em descending path\/} in the directed tree $\vec\T_{\mu}$ is
a path in $\T$ as above, where each edge $f_i$ points
away from $v_i$ and towards $v_{i+1}$:
\begin{equation}\label{eq:Path}
\vec P = \big( u_0 \xrightarrow{f_0} u_1 \xrightarrow{f_1}  \cdots   \xrightarrow{f_{n-1}} u_n \xrightarrow{f_n} \cdots
 \big) \end{equation}
where $u_i\in\V$ and $f_i\in\E$, which forms a geodesic in $\T$.
\end{definition}
If $P$ is a descending path as above, then the sequence of 
values $(\xi_n,\eta_n,\zeta_n)$ at $u_n$ 
are decreasing in absolute value, that is:
\begin{align}\label{eq:DescendingDecreasing}
\vert \xi_{n+1} \vert & \le  \vert \xi_n \vert \notag \\
\vert \eta_{n+1} \vert & \le  \vert \eta_n \vert \notag \\
\vert \zeta_{n+1} \vert & \le  \vert \zeta_n \vert. \end{align}
This follows since at each step, two of the three coordinates $\xi,\eta,\zeta$
remain constant and the absolute value of the remaining coordinate
does not increase.

\begin{figure}
\centerline{\includegraphics[width=\textwidth]{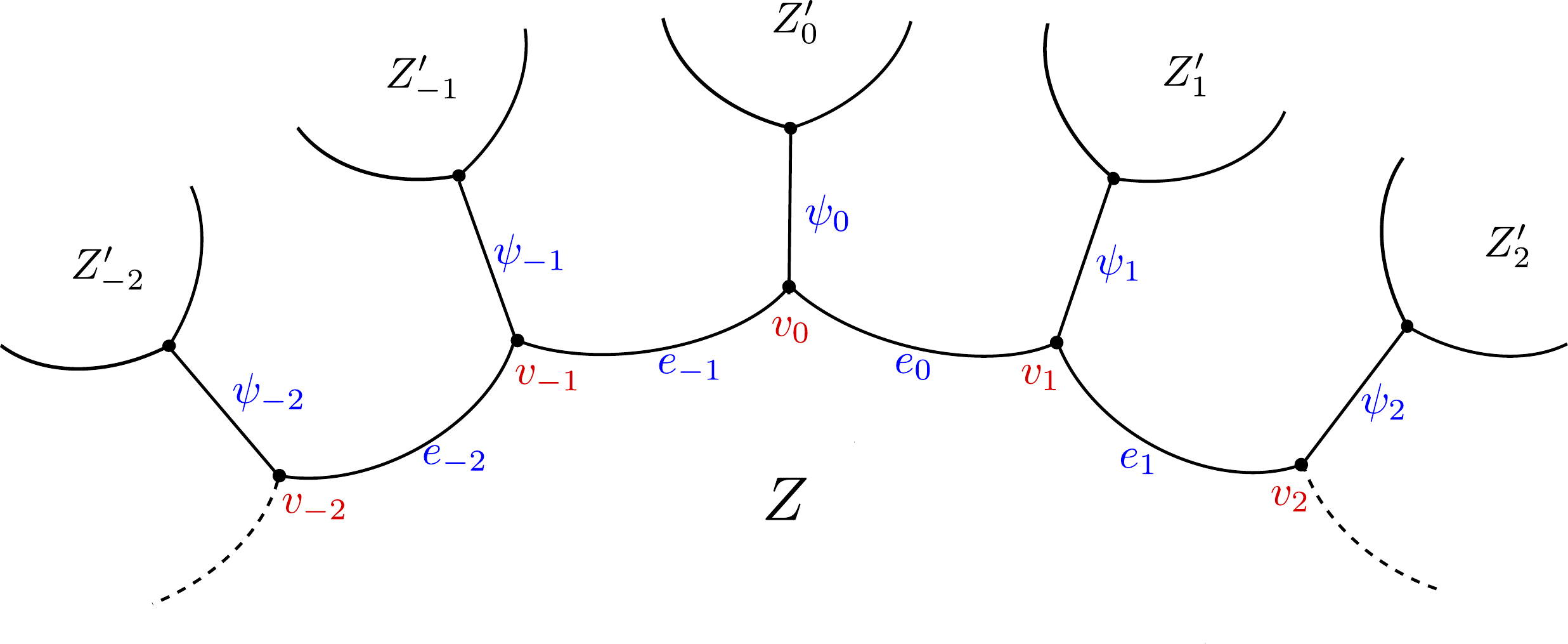}}
\caption{The regions around an alternating geodesic}
\label{fig:GeodesicAbuttingZ}
\end{figure}

\subsection{Exceptional characters}
We state here some of the basic results associated to the trace labelings 
$\mu$ as well as the directed trees $\vec \T_{\mu}$;  see
\cite{MR1643429}, 
\cite{MR2405161},
and \cite{MR2370281} 
for more details. 
For ease of exposition, we exclude
certain exceptional cases from our discussion below.
 
To begin, exclude the characters arising from
{\em reducible representations,\/} namely characters in 
$\kappa\inv(2)$. 
The dynamics on the set of reducible characters is
discussed in 
\cite{MR2497777}.
Next, exclude the $\hGamma$-orbit 
\[ \hGamma \cdot \big( \{0\}\times\C\times\C\big).\]
For each $k\in \C$, this set
intersects $\kappa\inv(k)$ in a set of measure zero, and hence does not
affect our claims of ergodicity. Nonetheless this case presents quite
interesting dynamics. 

Among the exceptional characters are {\em dihedral characters.\/} 
These are characters of representations where both elements of some basis
map to involutions in $\SLtC$. 
Taking that basis to be $(X,Y)$, the corresponding character equals
$(0,0,\zeta)$ for some $\zeta\in\C$. 
Therefore dihedral characters lie in the $\Gamma$-orbit of the coordinate
line $\{0\}\times\{0\}\times\C$.
The orbit of $(0,0,\zeta)$ equals 
\[ \big\{ (0,0,\zeta),(0,0,-\zeta),(0,\zeta,0),(0,-\zeta,0),(\zeta,0,0),(-\zeta,0,0) \big\}. \]
Dihedral characters are the most degenerate characters since 
directing the tree $\T$ is completely arbitrary:
{\em The directed trees corresponding to dihedral characters are completely indecisive\/}.

We call the elements of the set
\[ \kappa \inv (2)\  \cup\  (\hGamma \cdot(0,\eta,\zeta))\  \subset\  \C^3 \]
{\em exceptional} characters, and elements of its
complement {\em unexceptional} characters. 
The associated
trace labelings and directed trees are called {\em exceptional} or {\em
unexceptional} accordingly. Note that an {\em unexceptional} trace labeling
$\mu$ takes values only in $\C \setminus \{0\}$.


\subsection{The Fork Lemma}
The following fundamental result is proved in 
\cite{MR1643429}      
(Lemma 3.2(3)) for characters $(\xi,\eta,\zeta) \in \kappa^{-1}(-2)$
and extended in \cite{MR2370281}     
(Lemma~3.7) 
and \cite{MR2405161} (Proposition~3.3) 
for general characters.
Recall that if the two edges are directed away from a vertex $v$, 
then that vertex must either be a fork or a source.
We will see that such vertices are rare in our applications.

We supply a brief proof for the reader's convenience.

\begin{lem}[The Fork Lemma]\label{lem:Fork}
Suppose that $\mu$ is a trace labeling with directed tree $\vec\T_{\mu}$. 
If for the vertex $v(X,Y,Z) \in\V$, the edges opposite
to $X$ and $Y$ are directed away from $X$ and $Y$ respectively, 
then $|\zeta| \le 2$ or $\xi=\eta=0$.
\end{lem}
\begin{proof}
Write $e
^{Y,Z}(X \rightarrow X')$ and $e
^{X,Z}(Y \rightarrow Y')$ for the two edges directed away from $v$,
with corresponding traces $\xi,\eta,\zeta,\xi',\eta'$.
By assumption,
\begin{align*}
\vert \xi \vert &\ge \vert \xi'\vert \\
\vert \eta \vert &\ge \vert \eta'\vert,
\end{align*}
which implies:
\begin{align*}
 2\vert \xi\vert \ge  \vert \xi \vert + \vert \xi' \vert  &\ge  \vert \xi+\xi' \vert = \vert \eta\zeta \vert   \\
 2 \vert \eta \vert  \ge  \vert \eta \vert + \vert \eta' \vert  &\ge  \vert \eta+\eta' \vert = \vert \xi\zeta \vert  \end{align*}
Summing these two equations yields:
\[   2\  \big( \vert \xi \vert + \vert \eta \vert \big)\; \ge \;  \vert \zeta \vert\  
\big( \vert \xi \vert + \vert \eta \vert \big)  \]
Thus $2 \ge\vert \zeta\vert$ or $\xi=\eta=0$.
\end{proof}
%
%

\noindent
Since characters vanishing on a pair of two adjacent complementary regions 
are exceptional, a region $Y$ adjacent to a source or a fork arising from an unexceptional character must satisfy
$\vert\mu(Y)\vert \le 2$. 

We prove properness of the action using the
two {\em Q-conditions\/}
introduced by Bowditch in \cite{MR1643429}.
These conditions describe an open subset of the 
relative character variety with proper $\Gamma$-action.
Tan-Wong-Zhang~\cite{MR2370281} extended this set 
to an open subset of the {\em full\/} character variety 
with proper $\Gamma$-action.
For our purposes, we relax the conditions
slightly, allowing finitely many parabolics among the set of primitive elements
as in \cite{MR2247658}.
The resulting subset is no longer open, 
but the $\Gamma$-action remains proper.

\begin{definition}
Given a character $\mu$ and constant $C>0$, define $\Omega_{\mu}(C)$
to be the union of the set of closures of complementary regions $Y\in\O$ with
$\vert \mu(Y)\vert \le C$.
 \end{definition}
\medskip

\noindent
Lemma~\ref{lem:Fork} immediately implies:

\begin{lem}
[Bowditch~\cite{MR1643429}, Theorem 1(2))]
\label{lem:quasiconvexity}
Suppose $C\ge 2$. 
For every character $\mu$, the set $\Omega_{\mu}(C)$ is connected.
\end{lem}
\noindent(See also \cite{MR2370281}, Theorem~3.1(2).)

The following {\em (extended) BQ-conditions\/} were first defined by Bowditch in \cite{MR1643429}, and generalized by Tan-Wong-Zhang~\cite{MR2370281,MR2247658}):  
\begin{align}
\mu(X) &\not\in [-2,2] \text{~for all~}X \in \O; \label{eq:BQ1} \\
\mu(X) &\not\in (-2,2) \text{~for all~} X \in \O; \label{eq:BQ1'} \\
\Omega_{\mu}(2) &= \{Y \in \O: |\mu(Y)| \le 2\} \text{~is finite.} \label{eq:BQ2}
\end{align}
Condition~\eqref{eq:BQ1} means that every primitive element of $\Ft$ maps to a loxodromic element,
and Condition~\eqref{eq:BQ1'} means that no primitive element of $\Ft$ maps to an elliptic.
Condition~\eqref{eq:BQ2} is equivalent to the condition that 
$\Omega_{\mu}(C)$ is finite for any $C\ge 2$.
See \cite{MR1643429} and \cite{MR2247658} for details.

\begin{definition}\label{def:BQ}
A character $\mu$ is a {\em Bowditch character\/} if and only if $\mu$  
is irreducible and   satisfies \eqref{eq:BQ1} and \eqref{eq:BQ2} above.
A character $\mu$  is an  {\em extended Bowditch character\/} if and only if $\mu$  
is unexceptional and irreducible and satisfies 
\eqref{eq:BQ1'} and \eqref{eq:BQ2} above.

The {\em Bowditch set $\B$\/} comprises all Bowditch characters,
and the 
{\em extended Bowditch set $\B'$\/} comprises all extended Bowditch characters.
\end{definition}
\noindent
Condition 
\eqref{eq:BQ1'} 
allows for  finitely many such parabolics among the primitives.
The extended Bowditch set $\B'$  is no longer open in $\C^3$ but $\Gamma$ still acts properly on it.

Exceptional characters do not lie in $\B$: 
in the case of the reducible characters where $k=2$ by definition and  otherwise, 
Condition~
\eqref{eq:BQ1} is not satisfied if 
\[ \mu \in  \hGamma \cdot (0,\eta,\zeta). \] 
Similarly, if 
\[ \mu(Z) =\pm \sqrt{k+2} \] 
for some $ Z \in \O$, then 
the proof of Proposition~\ref{prop:EdgesAboutX} implies that
$\mu$ cannot satisfy Condition~\eqref{eq:BQ2}.

\medskip
Bowditch~\cite{MR1643429} and Tan-Wong-Zhang~\cite{MR2370281,MR2247658}   
proved that $\B$ is an open subset of the character variety 
and, for every $\mu\in\B'$,  
the flow $\vT_\mu$ possesses an attractor.
Thus Proposition~\ref{prop:ProperAction} implies:

\begin{thm}[
Bowditch~\cite{MR1643429}, 
Tan-Wong-Zhang~\cite{MR2370281}]    
$\Gamma$ acts properly on the extended Bowditch set $\B'$.\end{thm}

For general $\mu\in\RepF$, Tan-Wong-Zhang~\cite{MR2405161} associate 
a closed subset $\mathcal{E}(\mu)$ of the projectivized measured lamination space
of $\Soo$, the set of {\em end invariants\/} of $\mu$. 
(This space identifies with $\RPo$, the completion of $\QPo$.
A character $\mu$ lies in $\B'$ if and only if 
$\mathcal{E}(\mu) = \emptyset$ (\cite{MR2405161}, Theorem~1.3). 
We complete their computation of $\mathcal{E}(\mu)$ for imaginary characters
in this paper.   
We show that, for $k>2$, either
$\mathcal{E}(\mu) = \emptyset$ (when $\mu$ lies in the domain of discontinuity
of the $\Gamma$-action) or $\mathcal{E}(\mu)$ consists of only one point.

 \subsection
{Alternating geodesics}\label{sec:AlternatingGeodesics}
For a fixed $Z\in\O$, the values of a trace labeling $\mu$ on the other regions abutting 
$C(Z)$ admits an explicit description by inductively iterating \eqref{eq:edgerelation}.

First we establish some notation.
Choose a trace labeling $\mu$ associated to a character $[\rho]\in\RepF$.
Denote the trace $\mu(A)$ associated with $A\in\O$ by the corresponding 
lower-case letter $a$.

Let $ \{ v_n \mid n \in \Z \} $
be the ordered set of vertices adjacent to $Z$,
labeled so that 
\[ v_n = v(Y_{n-1},Y_n,Z), \]
inductively beginning at $Y_0 = X$ and $Y_1 = Y$.
Let  $e_n\in \E$ be the edge between $v_{n}$ and $v_{n+1}$.
We write
\[
\cdots 
\stackrel{e_{-1}}{\rule[.05in]{.2in}{.01in}}
v_0 \stackrel{e_0}{\rule[.05in]{.2in}{.01in}}v_1  
\stackrel{e_1}{\rule[.05in]{.2in}{.01in}}v_2 
\stackrel{e_2}{\rule[.05in]{.2in}{.01in}}
\cdots
\]
for this geodesic.
Let  $Y_n \in \O$ be the region adjacent to $Z$ sharing
the edge $e_n$ with $Z$. 
Finally, let  $\vedge{n}$ be the edge incident to $v_n$
not adjacent to $Z$, and $Z_n' \in \O$ the other region
opposite to the edge $\vedge{n}$.
Denote the corresponding traces by $\zeta = \mu(Z)$ and $\eta_n = \mu(Y_n)$.
Compare Figure~\ref{fig:GeodesicAbuttingZ}.

The indices $j(e_n)\in\{1,2,3\}$ for the tricoloring described in
\S\ref{sec:Tricoloring} 
 alternate between two elements in 
$ \{1,2,3\}$, depending on the mod $2$ class associated to $Z$.
This leads to a simple formula  \eqref{eq:ynFormula}
for the set of traces $\eta_n$.
In particular the sequence of characters forms a {\em lattice\/} in a hyperbola.
Recall that the hyperbola $\{ (\xi,\eta) \mid \xi \eta = 1\}$ is naturally an algebraic group
(the multiplicative group of the ground field), 
and a lattice is a discrete subgroup $\Lambda$ with compact quotient; 
such a lattice is necessarily a cyclic group.

\begin{prop}\label{prop:abut}
Suppose $\zeta\neq \pm 2$. Let $\lambda,\lambda^{-1}$ be the two (distinct) solutions
of 
\[ \lambda^2 - \zeta \lambda + 1 = 0 \]
so that 
\[ \zeta = \lambda + \lambda^{-1}.\]
Then, for $n\in\Z$, 
\begin{equation}\label{eq:ynFormula}
\eta_n\  =\;  A \lambda^n\ +\   B \lambda^{-n}  \end{equation}
where
\begin{align*}
A & =  \frac{\eta -\lambda^{-1}\xi}{\lambda - \lambda^{-1}}  \\
B & = \frac{\lambda \xi - \eta}{\lambda - \lambda^{-1}}.
\end{align*}
Furthermore
\[ AB = \frac{\zeta^2-k-2}{\zeta^2-4}. \]
\end{prop}
\begin{proof}
\eqref{eq:edgerelation} implies that $\eta_n$ satisfies the recursion
\[ \eta_{n+1} =  \zeta \eta_n + \eta_{n-1}. \]
Rewrite this recursion as:
\[ \mathbf{Y}_{n+1}  = \mathbf{Z}   \mathbf{Y}_n\]
in terms of a vector variable $\mathbf{Y}_n$ and matrix variable
$\mathbf{Z}$ defined by:
\[ \mathbf{Y}_n := \bmatrix  \eta_n \\ \eta_{n-1} \endbmatrix,\ 
\mathbf{Z}  = \bmatrix   \zeta & -1 \\ 1 & 0 \endbmatrix. \]
Diagonalize $\mathbf{Z}$ as:
\[ \mathbf{Z}  =  \frac1{\lambda-\lambda^{-1}}
\bmatrix \lambda^{-1} & \lambda \\ 1 & 1 \endbmatrix
\bmatrix \lambda^{-1} & 0 \\ 0 & \lambda \endbmatrix
\bmatrix -1 & \lambda \\ 1 & -\lambda^{-1} \endbmatrix \]
where
\[
  \frac1{\lambda-\lambda^{-1}}
\bmatrix \lambda^{-1} & \lambda \\ 1 & 1 \endbmatrix
\bmatrix -1 & \lambda \\ 1 & -\lambda^{-1} \endbmatrix
= \bmatrix 1 & 0 \\ 0 & 1 \endbmatrix.
\]
Then
\[
\mathbf{Y}_n = \mathbf{Z}^n \mathbf{Y}_0 = 
\frac1{\lambda-\lambda^{-1}}
\bmatrix \lambda^{-1} & \lambda \\ 1 & 1 \endbmatrix
\bmatrix \lambda^{-n} & 0 \\ 0 & \lambda^n \endbmatrix
\bmatrix -1 & \lambda \\ 1 & -\lambda^{-1} \endbmatrix
\]
from which the result follows.
\end{proof}
%
%
%

\noindent
When $\zeta = \pm 2$, the situation is even simpler:
\begin{prop}\label{prop:abutPar}
If $\zeta = 2$, then 
\[  \eta_n =  \xi + n (\eta-\xi) \]
and if $\zeta = -2$, then
\[ \eta_n =  (-1)^n \big( \xi - n (\xi + \eta) \big). \]
\end{prop}
\noindent

When one of the coordinates equal $\pm 2$, the corresponding level set is a union of two parallel lines. 
For example, $\zeta = \pm 2$ defines the {\em degenerate conic:\/}
\[ (\xi \mp \eta)^2  =  k - 2. \]
The characters corresponding to vertices
$v(Y_{n-1},Y_n,Z)$ then form a lattice on one of these lines
(in the sense of \S\ref{sec:AlternatingGeodesics}).

Propositions~\ref{prop:abut}~and~\ref{prop:abutPar} easily imply:
\begin{cor}\label{cor:abutgrowth} Suppose that $k\neq 2$.
\renewcommand{\labelitemii}{$\bullet$}
\begin{itemize}
\item[(a)] If either
\begin{itemize}
\item
$\zeta \not\in [-2,2]$ 
and $\zeta^2 \neq k +2$, or
\item
$\zeta =\pm 2$,  
\end{itemize}
then $\lim_{n \rightarrow \pm \infty}  \vert \eta_n \vert  = \infty$.
\item[(b)] If $\zeta \in (-2,2)$, then $\vert \lambda \vert =1$.
In particular 
$\vert \eta_n \vert$ is bounded for all $n \in \Z$. 
\item[(c)] 
If $\zeta=0$, then $\eta_{n+2}=-\eta_{n}$ for all $n \in \Z$.
Every edge abutting $Z$ is indecisive.
\item[(d)] 
If   $\zeta^2 = k+2$ and $(\xi,\eta) \neq (0,0)$, 
then 
$\lim_{n \rightarrow + \infty}  \vert \eta_n \vert$ equals either $0$ or $\infty$,
and $\lim_{n \rightarrow - \infty}  \vert \eta_n \vert$ equals either $\infty$ or $0$ respectively.
\end{itemize}
\end{cor}
\begin{proof}
For (a), note that $\vert\lambda\vert \neq 1$ and $AB \neq 0$. 
For (d), note that $AB = 0$. 
Now apply \eqref{eq:ynFormula}.
\end{proof}

Here is an important consequence of the Fork Lemma~\ref{lem:Fork}.

\begin{prop}\label{prop:EdgesAboutX}
Let $\mu$ be an unexceptional character with $\kappa(\mu)=k$.
Suppose that  $\vert \zeta\vert =|\mu(Z)|> 2 $ for $Z \in \O$.
Let 
\[ \vec C(Z) :=  \{\vec e_n\}_{n \in \Z}  \]  
be the sequence of directed edges of $\vT_{\mu}$ surrounding $Z$. 
Then either:
\begin{enumerate}
  \item [(a)] 
$\zeta^2 \neq k+2$,
and $\partial Z$ contains a vertex $v$ such that 
all $\vec e_n$ point towards $v$, or
  \item [(b)] 
$\zeta^2 = k+2$, and $\vec e_n$ all point in one direction.
\end{enumerate}
\end{prop}

\begin{proof}
Suppose 
that neither the $\vec e_n$ all point in one direction,  nor  is there a vertex on 
$\partial Z$ towards which all the $\vec e_n$ point.
Then $\partial Z$ contains a vertex $v$ which is a fork or source, 
with $Z$ adjacent to two edges directed out of $v$. 
Since $\mu$ is not dihedral and $|\zeta|>2$ by assumption, 
this contradicts Proposition~\ref{lem:Fork}.

Now let $Y_n$, $n \in \Z$ be the (ordered) neighboring regions of $Z$,
with $Y_n$ adjacent to $e_n$. 
By Proposition~\ref{prop:abut}, 
\[ \eta_n=A\lambda^n+B\lambda^{-n}, \]
where
\[ AB=\frac{\zeta^2-k-2}{\zeta^2-4}, \quad \lambda+\lambda^{-1}=\zeta. \]
Note that
$|\lambda| \neq 1$ since $\zeta \not\in [-2,2]$. 
If 
\[ \zeta^2 \neq k+2,\]
then $A,B \neq 0$ so that $|\eta_n| \rightarrow \infty$ as
$n \rightarrow \pm \infty$.
Thus  $\vec e_n$ must eventually point
inwards from both directions.
Hence all 
$\vec e_n$  point towards $v$ 
for a unique vertex $v$ adjacent to $Z$.

If 
$\zeta^2 = k+2$,
exactly one of $A$, $B$ equals 0 (since $\mu$ is not dihedral).
Then $|\eta_n|$ is monotone in $n$, increasing to
$\infty$ in one direction and decreasing to $0$ in the other. 
Therefore all $\vec e_n$  point in the same direction.
\end{proof}
\noindent
When $\zeta^2 = k+2$, the $\zeta$-level set on $\kappa^{-1}(k)$ 
consists of two lines which intersect in the dihedral character $(0,0,\zeta)$ as follows.
Choose $\phi$ so that $\phi^2 = k -2$.
Then $\zeta = \lambda + \lambda\inv$ where
\[ \lambda := \frac{\zeta + \phi}2, \qquad
\lambda\inv := \frac{\zeta - \phi}2. \]
Then the solutions of $\kappa(\xi,\eta,\zeta) = k$ for fixed $\zeta$ with $\zeta^2 = k+2$ as above form the degenerate hyperbola
\[ ( \xi  - \lambda  \eta) (\xi  - \lambda\inv  \eta) = 0 \]
which is the union of the two lines
\[ \xi  - \lambda  \eta  = 0, \qquad \xi  - \lambda\inv  \eta = 0. \]
This is a degenerate conic section of the cubic surface
$\kappa^{-1}(k)$.
The sequence of vertices $v(Y_{n-1},Y_{n},Z)$ then correspond to a
lattice in the group $\C^*$ acting on the complement of $(0,0,\zeta)$ 
in one of the above lines.

The case when $\zeta=\pm 2$ is similar to the case where $|\zeta|>2$ using 
Proposition~\ref{prop:abutPar}:  
in this case all $\vec e_n$ point towards some vertex $v \in \partial Z$.

\subsection{Indecisive edges and orthogonality}
Characters where two loxodromic generators have orthogonal axes
play a special role in this theory, especially for imaginary characters.
Recall that the natural complex-orthogonal pairing on $\sltc$
\begin{align*}
\sltc \times \sltc &\longrightarrow \C \\
(X,Y) &\longmapsto X \cdot Y := \tr(XY) \end{align*}
is invariant under the adjoint representation $\SLtC \xrightarrow{~\Ad~} \Aut\big(\sltc\big)$.
The {\em traceless projection\/} 
\begin{align*}
\SLtC &\longrightarrow \sltc \\
X &\longmapsto  X - \frac{\tr(X)}2 \Id \end{align*}
is $\Ad$-equivariant.
We may normalize to obtain $\hat{X}\in\SLtC \cap \sltc$ by dividing by  a square root of
\[
\Det\Big( X - \frac{\tr(X)}2 \Id \Big) = \frac{\tr(X)^2 - 4}4 \]
to obtain the {\em involution\/} fixing the invariant axis $\Axis(X)$ when $X$ is loxodromic;
compare \cite{MR2497777}, \S 3.2.

\begin{definition} Loxodromic elements $X,Y\in\SLtC$ are {\em orthogonal\/} 
if and only if $\Axis(X)$ and $\Axis(Y)$ are orthogonal
lines in $\Hth$. \end{definition}

\begin{prop}\label{prop:OrthogonalAxes}
Let $(\xi,\eta,\zeta)\in\C^3 = \RepF$ and consider the effect of a Vieta involution,
for example
\[ (\xi,\eta,\zeta) \longmapsto 
(\xi,\eta,\zeta') = (\xi,\eta,\xi\eta -\zeta). \]
If $\rho\in\HomF$ is a representation corresponding
to $(\xi,\eta,\zeta)$, then $\zeta' = \zeta$ if and only $\rho(X)$ and $\rho(Y)$ are orthogonal.
\end{prop}
\begin{proof} 
The traces of $\rho(X), \rho(Y)$ and $\rho(XY)=\rho(X)\rho(Y)$ are, respectively,
$\xi, \eta$ and $\zeta$. 
Now $\rho(X)$ and $\rho(Y)$ are orthogonal if and only if
\[
\widehat{\rho(X)} \cdot \widehat{\rho(Y)} = 0, \]
or equivalently,
\begin{align*}
0 & = \tr\bigg( 
\Big(\rho(X) - \frac{\tr\big(\rho(X)\big)}2\Id \Big) 
\Big(\rho(Y)- \frac{\tr\big(\rho(Y)\big)}2\Id  \Big) \bigg) \\ & =
\tr\big(\rho(X)\rho(Y)\big)
- \frac{\xi\eta}2 - \frac{\xi\eta}2 + \frac{\xi\eta}4\   \tr(\Id) \\ & = 
\zeta - \frac{\xi\eta}2 = \frac{\zeta - \zeta'}2. \end{align*}
Thus $\zeta' = \zeta $ if and only $\rho(X)$ and $\rho(Y)$ are orthogonal, as desired.
\end{proof}
\noindent
We now construct {\em indecisive edges\/} in $\vT_\mu$, 
and characterize them, at least when $\mu$ is unexceptional and real or imaginary:

\begin{cor}
Suppose that $\mu$ is a real or purely imaginary unexceptional trace labeling.
Then an edge joining $(\xi,\eta,\zeta)$ to 
$(\xi,\eta,\zeta')$ is indecisive if and only if the corresponding generators $\rho(X),\rho(Y)$
are orthogonal.
\end{cor}
\begin{proof}
The edge is indecisive if and only if
$\vert\zeta'\vert = \vert\zeta\vert$.
Since $\zeta,\zeta' \in \R \cup i\R$),
this occurs  if and only if $\zeta' = \pm \zeta$.
If $\zeta' = -\zeta$, then $\xi\eta = 0$, contradicting $\mu$ being unexceptional.
Thus $\zeta' = \zeta$ and 
Proposition~\ref{prop:OrthogonalAxes} implies
$\Axis\big(\rho(X)\big) \perp \Axis\big(\rho(Y)\big)$ as claimed.
\end{proof}
\noindent
Attracting indecisive edges and sinks are the only attractors
arising for real and imaginary characters. 
For hyperbolic structures on $\Sthz$, 
the only attractors are sinks, but for $\Soo$ attracting indecisive edges occur
exactly when two simple geodesics intersect orthogonally in one point.
For the nonorientable surfaces $\Coo$ and $\Czt$, the situation is similar:
A hyperbolic structure homeomorphic to $\Coo$ (with a possible cone point) corresponds
to an attracting indecisive edge if and only if the two-sided interior simple geodesic
meets a one-sided simple geodesic once, and orthogonally.
A hyperbolic structure homeomorphic to $\Czt$ corresponds to 
an attracting indecisive edge if and only if the two one-sided simple geodesics
intersect once, and orthogonally.

\section{Imaginary trace labelings}

We apply the results of the previous section to trace labelings arising
from characters of the form $(ix,iy,z)\in \C^3$ where $x,y,z \in \R$
which is the case of interest in this paper. 
These are the trace labelings which are adapted to the
$\{\pm 1\}$-character $\Coloring$ defined by
\begin{align*}
X &\longmapsto -1  \\
Y &\longmapsto -1  \\
Z &\longmapsto 1 . \end{align*}
Following 
\cite{MR1643429, MR2370281}, 
we call these {\em imaginary trace labelings\/} (or 
{\em imaginary characters\/}).

When $k>2$, Proposition~\ref{prop:ProperAction2} characterizes imaginary Bowditch characters in terms of the associated directed tree, which is later related to the generalized Fricke space $\Fricke(C_{1,1})$ (\S\ref{sec:ProjectionFor2<k}).
The main step is Theorem~\ref{thm:WellDirectedCriterion}.
Similarly, when $k<2$,  the Bowditch characters are characterized in terms of the associated directed tree which relate to the Fricke space $\Fricke(C_{0,2})$.
Away from $\Fricke(C_{0,2})$ are separating curves mapping to elliptic elements.
The existence of simple closed curves mapping to elliptic elements  implies ergodicity on the complement of the wandering domains.
Existence of {\em elliptic primitives\/} is the main qualitative difference between the cases $k>2$ and $k<2$. 
\subsection{Well-directed trees}
\begin{definition}
A directed tree with no forks or sources is said to be {\em well-directed.\/}
In this case, the edges are either well-directed towards an
attractor or towards a unique end of the tree. 
If the attractor is a sink, then  $\vT$ is said to be 
{\em well-directed towards a sink.}
\end{definition}
Each case yields a tree well-directed towards a sink or an indecisive edge for Bowditch characters.

We adopt the following conventions. 
We use the letters $Z, W$ to denote elements of $\O$ with real trace labels 
and $X$ or $Y$ for elements of $\O$ with purely imaginary trace labels. 
Furthermore, write: 
\begin{align*}
z &:=\mu(Z), \\ 
w &:=\mu(W), \\ ix &:=\mu(X), \\
iy &:=\mu(Y),
\end{align*} 
where $x,y,z,w \in \R$. 

\begin{prop}\label{prop:ProperAction2}
Let $\mathscr{W}$ denote the set of unexceptional imaginary trace labelings $\mu$ for which $\vT_\mu$ is well-directed towards a sink $v(X,Y,Z)$ 
where $z \in \R$,  $|z|\ge 2$. 
Then $\mathscr{W}\subset\B'$.
In particular $\Gamma$ acts properly on $\mathscr{W}$.

Conversely, the directed tree associated to an imaginary Bowditch character is well-directed towards such a sink.  

\end{prop}
\noindent
We defer the proof to the latter part of this section. 
Our analysis separates naturally into two cases, depending on whether $k>2$ or $k<2$ where $k =\kappa(\mu)$. 
Proposition~\ref{prop:ProperAction2} will be proved separately for these two cases.
The case $k>2$ includes the generalized Fricke space for the one-holed Klein bottle $C_{1,1}$ and the case $k<2$ includes the Fricke space of the 
two-holed cross-surface $C_{0,2}$.

We first discuss how a $\{\pm 1\}$-character $\Coloring$ affects the  tree $\T$.




 \subsection{$\{\pm 1\}$-characters on $\Ft$}

A nontrivial homomorphism $\Ft \xrightarrow{\Coloring} \Z/2$ divides the
complementary regions in $\O = \Prim$ 
into two classes:
\begin{itemize}
\item
$\O_\R$: 
Those in the kernel of $\Coloring$ correspond to orientation-preserving curves, 
and have real trace labels. 
\item 
$\O_{i\R}$:
Those upon which $\Coloring$ is nonzero correspond to orientation-reversing curves
and have purely imaginary trace labels. 
\end{itemize}

%
\noindent
Similarly, $\Coloring$ determines a dichotomy of the edges of $\T$ into
{\em real\/} edges and {\em imaginary\/} edges:

\begin{definition} \label{defn:RealImaginaryEdges}
Let $\Coloring$ be the $\{\pm 1\}$-character and let $e\in\E$ be an edge.
Define:
\begin{itemize}  \item $e$ is {\em real with respect to $\Coloring$\/} 
if neither of the two complementary regions in $\Omega$ abutted by $e$ 
lie in $\O_{\R}$. 
\item $e$ is {\em imaginary with respect to $\Coloring$\/} 
if exactly one of the two complementary regions in $\Omega$ abutted by $e$  lies in $\O_{\R}$. 
\end{itemize}
\end{definition}
In terms of the tricoloring $\E\xrightarrow{~j~}\{1,2,3\}$ defined in 
\S\ref{sec:Tricoloring}, an edge $e$ is an $\R$-edge if and only if
$j(e) = 3$ and is an $i\R$-edge if and only if $j(e) = 1 \text{~or~} j(e) = 2$.

When $\Coloring$ is understood, we say that $e$ is an {\em $\R$-edge\/} 
if $e$ is real with respect to $\Coloring$ and $e$ is an {\em $i\R$-edge\/} if 
$e$ is imaginary with respect to $\Coloring$.

\begin{figure}[h]
\includegraphics[scale=.5]
{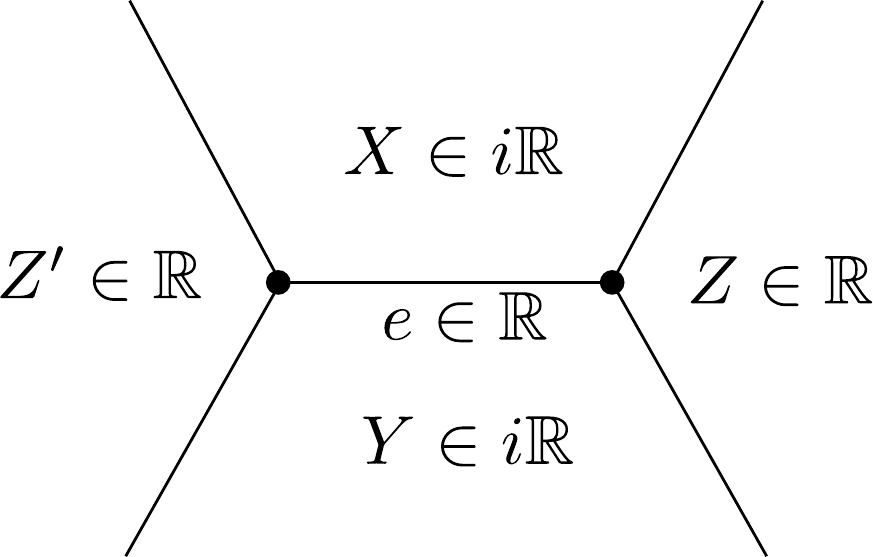}
\qquad \qquad \quad
\includegraphics[scale=.5]
{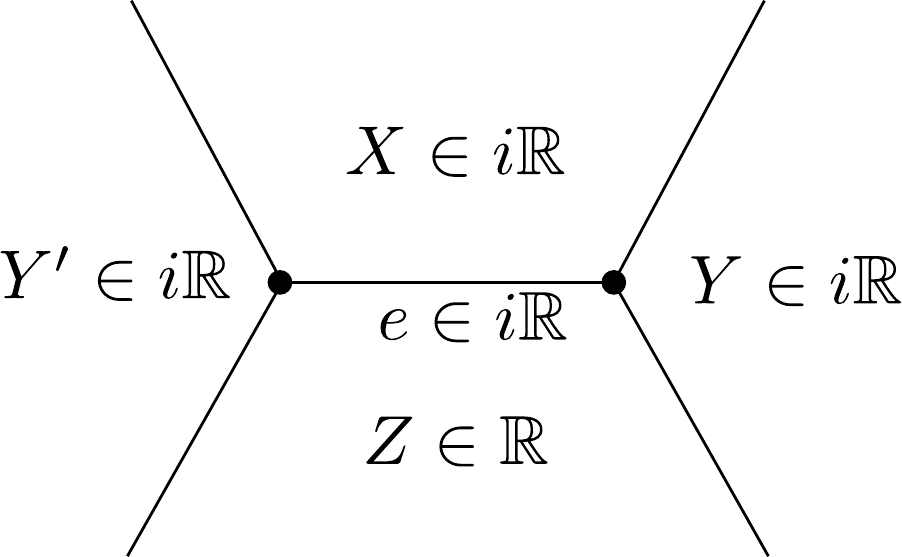}
\caption{An $\R$-edge and an $i\R$-edge}
\label{fig:RealImaginaryEdges}
\end{figure}

\noindent
An $\R$-edge ends at two regions in $\O_{\R}$ and
an $i\R$-edge ends at two regions in $\O_{i\R}$, 
in the sense of \S\ref{sec:Graphical Representation}.
Compare Figure~\ref{fig:RealImaginaryEdges}.
\begin{prop}\label{prop:ColoredEdges}
For every $v\in \V$, exactly two edges incident to $v$ are $i\R$-edges.
The union  of $i\R$-edges is a disjoint union
of geodesics, namely the alternating geodesics corresponding to  $\O_{\R}$.
Thus the edges composing alternating geodesics of elements of $\O_{\R}$ consist 
entirely of  $i\R$-edges.
The edges composing alternating geodesics of elements of $\O_{i\R}$ alternate
between $i\R$-edges and $\R$-edges.
\end{prop}
\begin{proof}
The unique $\R$-edge at $v$ is the edge $e$ with $j(e)=3$.
Thus exactly two $i\R$-edges emanate from $v$.
Thus
%
%
the union of the $i\R$-edges
is a disjoint union of geodesics,
one through each vertex of $\V$:
Start at a vertex $v$ as above, and denote it $v_0$.
Denote  the endpoints of the two $i\R$-edges through $v_0$ by $v_{-1}$ and $v_1$ respectively.
Inductively define $v_n$ for integers $\vert n\vert > 1$ as follows. 
If $n > 1$, let $v_n$ be the endpoint of the unique $i\R$-edge through $v_{n-1}$  not containing $v_{n-2}$.
If $n < -1$, let $v_n$ be the endpoint of the unique $i\R$-edge through $v_{n+1}$  not containing $v_{n+2}$.
Then $v_n$ form the vertices of a geodesic through $v$ all of whose edges are 
$i\R$-edges.
Thus $\T$ decomposes into a disjoint union of geodesics, containing all vertices and
all $i\R$-edges. 
\end{proof}


\subsection{Positive and negative vertices}
Let $v = v(X,Y,Z)$ be a vertex with trace labeling $(ix,iy,z)$ where $x,y,z\in\R$. 
Define $v$ to be {\em positive \/} (respectively {\em negative\/}) if 
$xy z > 0$ (respectively $x y z < 0$). 
If $\mu$ is unexceptional, then every vertex is positive or negative.
Furthermore applying a sign-change automorphism to $\mu$ 
preserves the dichotomy of vertices into positive and negative vertices.

\begin{lem}\label{lem:PosNegVert}
If $v$ is positive, then the $\R$-edge adjacent to $v$ directs inwards.
\end{lem}
\begin{proof}
Apply the Edge Relation to the $\R$-edge so that
$z + z ' = - x y$. 
Since $xy$ has the same sign as $z$,
\[ \vert z'\vert  = \vert- z   -  x y \vert > \vert z \vert\]
as desired. 
\end{proof}

\section{Imaginary characters with $k>2$ } 

\noindent
We summarize the main results of this section:
\begin{thm8*}\label{thm:WellDirectedCriterion}
Suppose that $\mu$ is an unexceptional imaginary trace labeling with $k>2$. 
\begin{itemize}
\item The associated directed tree $\vT_{\mu}$ is well-directed. 
\item $\vT_{\mu}$ is well-directed towards a unique sink 
if and only if 
\[ \vert\mu(Z)\vert < \sqrt{k+2} \]
for some $Z \in \O_{\R}$. 
Equivalently, $\mu$ corresponds to a hyperbolic structure on the $1$-holed
Klein bottle $\Coo$ with a funnel, cusp, or cone point. 
\item
The Bowditch set $\B$ 
equals the generalized Fricke orbit $\gFO{\Coo}$,
and intersects the slice $\kCk$ in a dense open subset of $\kCk$.
\end{itemize}
\end{thm8*}

\subsection{Alternating geodesics when $k>2$.}\label{sec:MostEdgesMerge}
The proof begins by analyzing the alternating geodesic $C(Z)$ surrounding
a complementary region $Z\in \O_\R$. 
This geodesic contains a unique sink if and only if $\vert z\vert < \sqrt{k+2}$,
in which case the character lies in the generalized Fricke space
$\Fricke'(\Coo)$. When $\vert z\vert > \sqrt{k+2}$, 
one edge points out of $Z$ towards another $Z'\in\O_\R$
and we obtain a descending path.
If this descending path ends, then $\mu$ lies in $\gFO{\Coo}$.

First we record a useful observation:
\begin{lem}\label{lem:zBiggerThan2}
If $\mu$ is an unexceptional imaginary trace labeling with 
$k>2$ and $Z \in \O_{\R}$, 
then $\vert z\vert >2$ where  $z=\mu(Z)$.
\end{lem}
\begin{proof}
Proposition~\ref{prop:twocomponents} implies that 
$\rho(\gamma)$ is hyperbolic  if $\gamma$ is primitive and
$\Coloring(\gamma) = 1$. 
\end{proof}
\noindent The discussion divides into the three cases: 
\begin{itemize}
\item $2 <\vert z\vert< \sqrt{k+2}$;
\item $z = \pm \sqrt{k+2}$;
\item $\vert z\vert > \sqrt{k+2}$. 
\end{itemize}
We retain the notations of \S \ref{sec:AlternatingGeodesics} for
$e_n$, $v_n$, $Y_n$, $\vedge{n}$ and $Z_n'$, $n \in \Z$. 

  
\begin{prop}\label{prop:mergesaboutZa}
Suppose $ 2 <|z|< \sqrt{k+2}$.
\begin{itemize}
\item Each $v_n$ is positive;
\item All $e_n$'s point towards $v_m$, for a unique  $m \in \Z$.   
\item Each $\vedge{n}$ points towards $Z$;
 \end{itemize}
$v_m$ is a sink and all the other $v_i$'s where $i \neq m$ are merges. 
\end{prop}
\begin{proof} 
When $ 2 <|z|< \sqrt{k+2}$,
Proposition~\ref{prop:abut} implies:
\[ iy_n=A\lambda^n+B \lambda^{-n}\] where $\lambda+\lambda^{-1}=z$ and 
\[ AB = \frac{z^2-k-2}{z^2-4}. \]
Applying a sign-change if necessary, 
assume
\[ 2<z<\sqrt{k+2},\qquad \lambda>1,
\]
and that
\[ A=ai,\qquad B=bi \] 
where $a,b>0$.  
Then 
\[ y_n=a\lambda^n+b\lambda^{-n}>0 \] for all $n \in \Z$ and 
$y_n \rightarrow \infty$ as $n \rightarrow \pm \infty$. 

Since $z>0$ and all $y_n > 0$, each $v_n$ is a positive vertex.
Lemma~\ref{lem:PosNegVert} implies that  
$\vedge{n}$ points decisively towards $Z$ for all $n$.
Proposition \ref{prop:EdgesAboutX} implies all the $e_n$'s point towards a vertex, 
say $v_0(Z, Y_{-1},Y_0)$. 

Hence $v_0$ is a sink and $v_n$ is a merge for $n \neq 0$. 
At most one 
$e_n$   %
is indecisive. 
For example, if $y_1=y_{-1}$,  
the sink is either $v_0$ or $v_1$
(depending on the direction of $e_0$).
Another case occurs when $y_{-2}=y_0$. 
All other $e_n$'s are decisive.
\end{proof}

%
%
\begin{prop}\label{prop:mergesaboutZb}
Suppose $z=\pm \sqrt{k+2}$. 
\begin{itemize}
\item Each vertex $v_n$ is positive;
\item All the $e_n$'s point in the same direction.
\item Each $\vedge{n}$ points towards $Z$ for all $n$;
\end{itemize}
All the vertices $v_n$  are merges. 
\end{prop}
\begin{proof} 
When $ z = \pm \sqrt{k+2}$,
assume  \[ z=\sqrt{k+2}, \qquad  \lambda>1 \]
as in Proposition~\ref{prop:mergesaboutZb} above.
Since $AB=0$ and $\mu$ is not dihedral, one of $A, B \neq 0$, 
we may assume $B=0$ and $A=ia$ where $a > 0$. 
Therefore 
\[ y_n\ =\ a\lambda^n\ >\ 0 \] 
for all $n$ and increases monotonically as  $n\to\infty$. 
As in Proposition~\ref{prop:mergesaboutZa}, 
each $v_n$ is positive, 
$\vedge{n}$ points decisively towards $Z$ for all $n$ and 
$e_n$ points towards $e_{n-1}$ for all $n$. 
Hence all the $v_n$'s are merges, as claimed.
\end{proof}

\begin{prop}\label{prop:mergesaboutZc}
Suppose $\vert z \vert> \sqrt{k+2}$.
\begin{itemize}
\item $v_m$ is negative for a unique $m \in \Z$;
\item Each $e_n$ points towards $v_m$;
\item $\vedge{m}$ points away from $Z$ while  all the other  
$\vedge{i}$'s, with $i\neq m$,  point towards $Z$. 
\end{itemize}
All the $v_n$'s  are merges. 
$v_m$ is the only merge directed away from $Z$.
\end{prop}
\begin{proof}
If $ |z |> \sqrt{k+2}$, 
assume (as above)
\[ 
z>\sqrt{k+2},\quad \lambda>1. \]
Since $AB>0$, further assume 
\[ A=ia,\qquad B=ib \]
where $a>0$ and $b<0$. 
Hence $y_n$ increases monotonically and  
\[
\lim_{n\to -\infty} y_n = - \infty, \qquad
\lim_{n\to +\infty} y_n =  +\infty. \]
\noindent
Re-indexing if necessary, assume  $y_n>0$ for all $n \ge 0$ and $y_n<0$ for $n <0$. 
Then $v_0$ is negative and all other $v_n$ are positive.
Edge Relations~\eqref{eq:edgerelation} imply
that all the $e_n$'s point  decisively towards $v_0$.
Lemma~\ref{lem:PosNegVert} implies that all the $\vedge{n}$'s  point decisively towards $Z$ for $n \neq 0$. 
For  $\vedge{0}$, 
\[ z_0'=-y_{-1}y_0-z<0, \qquad -y_{-1}y_0>0 
\]
so $|z|>|z_0'|$. 
Hence  $\vedge{0}$ points decisively away from $Z$, 
as claimed.
\end{proof}
\noindent
Proposition~\ref{thm:WellDirectedCriterion} 
now follows from Lemma~\ref{lem:zBiggerThan2} 
and Propositions~\ref{prop:mergesaboutZa},
\ref{prop:mergesaboutZb}, 
\ref{prop:mergesaboutZc}: 
Since every vertex is adjacent to some $Z$ with $\vert z\vert > 2$,
there are neither forks nor sources. Thus every vertex is a merge or
a sink and $\vT$ is well-directed.

\subsection{The Bowditch set}\label{sec:BowditchSet}

Recall that a {\em peripheral structure\/} on a free group $G$ is a collection of 
conjugacy classes of cyclic subgroups of $G$ corresponding to the components of $\partial S$, where $S$ is a surface-with boundary with $\pi_1(S) \cong G$. 
Nielsen proved that an automorphism of $G$ corresponds to a mapping class
of a surface $S$ if and only if it preserves a peripheral structure on $\pi_1(S)$.
For example, peripheral structures on $\Ft$ corresponding to $S\approx \Sthz$
correspond to superbases, 
and Nielsen's theorem on $\Soo$ is equivalent to the statement that every automorphism
of $\Ft$ preserves the peripheral structure defined by the simple commutators of bases.

The peripheral structures on $\Ft$ corresponding to 
a one-holed Klein bottle $S$ with fundamental group $\Ft$ are determined
by the equivalence class in $\Ft$ corresponding to $\partial S$. 
In terms of a basis $(X,Y)$ of $\Ft$, such an equivalence class
is $\Au$-related to that of $X^2Y^2$.
Its homology class in $H_1(\Ft)\cong \Z^2$ is twice that of $XY$, 
which is a primitive element of $\Z^2$.
Therefore the peripheral structures corresponding to one-holed Klein bottles
correspond bijectively to inversion-conjugacy classes of primitive elements
in $\Ft$, that is, to points in $\QPo$.
When the coloring $\Coloring$ is specified, these primitive classes
correspond to 
\[ 1 \longleftrightarrow \frac{\mathsf{odd}}{\mathsf{odd}} \in \QPo\]
under reduction modulo $2$; see  \eqref{eq:ThreeElementSet}.
We refer to such an element of $\QPo$ (or the corresponding primitive class) 
as {\em totally odd.} 

Apply Proposition~\ref{prop:GeneralizedFrickeSpaceOfCoo} to
Proposition~\ref{thm:WellDirectedCriterion} to obtain:

\begin{cor} $\vT_{\mu}$ is well-directed towards a unique sink if and only if
$\mu$ lies in 
$\gFO{C_{1,1}} =  \Gamma\cdot \Fricke'(C_{(1,1)})$.
\end{cor}

\noindent
We identify this Fricke orbit with the Bowditch set:
\begin{prop}\label{prop:WellDirectedToSinkImpliesBowditch}
Suppose that $\mu$ is an unexceptional imaginary  trace labeling with $k>2$. 
Then $\vT_{\mu}$ is well-directed towards a sink 
if and only if  $\mu\in\B$.
\end{prop}
\begin{proof}
Suppose that $\vT$ is well-directed toward a sink 
$v_0(X_0,Y_0,Z_0)$ where 
\begin{align*}
\mu(X_0)& =ix_0, \\
\mu(Y_0)& =iy_0 \\
\mu(Z_0) & =z_0,
\end{align*}
and $x_0,y_0,z_0 \in \R$. 
We verify the Bowditch conditions.

Since $\mu$ is unexceptional and $|z_0|>2$, 
\[ \vert x_0\vert,\ \vert y_0\vert,\ \vert z_0\vert -2\; >\;\varepsilon \]
for some $\varepsilon>0$.

Propositions~\ref{prop:mergesaboutZa}, \ref{prop:mergesaboutZb} and \ref{prop:mergesaboutZc}
imply that %
all but possibly one edge of $\T_{\mu}$ 
point decisively towards $v_0$.
Furthermore the difference in the absolute values of the regions at the opposite ends of a decisive edge is $> \varepsilon^2$, 
except possibly for one $\R$-edge adjacent to $v_0$.  
Hence, moving away from the sink, 
each step (except for possibly the first step) increases the absolute value of the new trace label by at least  $\varepsilon^2$. 
Hence $\Omega_{\mu}(2)$ is finite and $\mu$ satisfies Condition~\eqref{eq:BQ2}.

Similarly, since the trace labeling $\mu$ is unexceptional, no trace value is zero. 
Now by Proposition \ref{prop:twocomponents},  $\vert\mu(Z)\vert > 2$ for $Z\in \Omega_\R(\T)$. Hence, no trace value lies in $[-2,2]$ which implies 
Condition~\eqref{eq:BQ1}.
Thus each Bowditch condition is satisfied and $\mu\in\B$, as claimed.

\end{proof}
\noindent
Another proof follows from Proposition \ref{prop:ProperAction}.
In this case the only attracting finite subtrees are indecisive attracting 
edges.

%
\subsection{Planar projection of the Bowditch set}
\label{sec:ProjectionFor2<k}
Now we describe in more detail the structure of the Bowditch set (or Fricke
orbit) and its complement in the level surface $\kCk$.
Proposition~\ref{prop:SmoothSurface} implies that the two components of the 
level surface $\kCk$ are graphs of the functions 
$z_{\pm}$ defined in \eqref{eq:zPlusOrMinus}
and the projection $\Pixy$ maps each component diffeomorphically onto
the $xy$-plane.

%
%
First observe that the imaginary characters $(ix,iy, \sqrt{k+2})$ 
lie on the lines in the plane $z = \sqrt{k+2}$
\[y=m^\pm  x \]
having slopes
\begin{equation}\label{eq:DefSlopes}
m^\pm = m^\pm(k) :=\frac{\sqrt{k+2}\pm \sqrt{k-2}}{2}. \end{equation}
These lines intersect in the dihedral character $\Delta= (0,0,\sqrt{k+2})$.
Moreover characters $(ix,iy, z)$ with $2<z<\sqrt{k+2}$ 
project to the sectors in the first and third quadrant bounded by these lines.
Similarly, the characters $(ix,iy, -\sqrt{k+2})$ lie on lines 
\[ y=-m^\pm x \] 
in the plane $z = -\sqrt{k+2}$. 
These lines intersect in the dihedral character $-\Delta= (0,0,-\sqrt{k+2})$.
In particular the tangent space at $\pm\Delta$ equals the horizontal plane
defined by $z=0$:
\begin{equation}\label{eq:HorizontalTangentSpace}
 T_{\pm\Delta} \big(\kCk\big)  = \R^2 \oplus 0 \subset \R^3. \end{equation}
 
Characters $(ix,iy, z)$ with $-\sqrt{k+2}<z<-2$ project to sectors in the second and fourth quadrant bounded by these lines. 
The infinite dihedral group generated by $\II_1,\II_2$  preserves these sectors.

Each sector identifies with the generalized Fricke space
$\Fricke'(C_{1,1})$ of the one-holed Klein bottle.
The four components represent the orbit of $\Fricke'(C_{1,1})$ 
under sign-change automorphisms. 

Recall 
that  the group $\Gamma$ is generated by sign-changes $\Sigma$ and
\[
\PGLtZc\  \cong\  \big(\Z/2\star\Z/2\star\Z/2\big) \ltimes \Z/2 \]
and this group acts faithfully on $i\R\times i\R\times\R$.
The projection $\Pixy(\B)$ is then the projection of the $\Gamma$-orbit of these sectors.
Each sector is $\hGamma_\Coloring$-invariant,
and therefore $\hGamma_\Coloring$ acts on the set of sectors.
Thus  the pair consisting of $(0,0,\pm z)$  is $\Gamma$-invariant.
These correspond to dihedral characters.
Therefore they meet the boundaries of components of $\Pixy(\B)$. 
All the other components are bounded by curves which are 
the images of the lines $y=m^{\pm} x$ under $\Gamma/\Gamma_{\Coloring}$. 
Except for the lines
$y=m^{\pm} x$,  these curves are not straight lines.

The projection $\Pixy$ maps $\B$ onto countably many connected components in the $xy$-plane,
all of whose closures contain the origin.

The components are indexed by the {\em totally odd\/} rationals,
that is, rational numbers $p/q$ where both $p,q$ are odd integers.
These components bijectively correspond to the subset  $\O_{\R}\subset\O$.
Furthermore 
a cyclic ordering on the set  
respects the cyclic ordering of the rationals in $\hat{\R}$. 
The complement of $\Pixy(\B)$ in $\R^2 \setminus \{(0,0)\}$ consists of uncountably many connected components indexed by $\hat{\R}$ and has empty interior.


These can be understood in terms of the {\em end invariants} introduced by Tan-Wong-Zhang~\cite{MR2405161} which we describe briefly. 
Recall that the tree of superbases $\T$ can be identified with the dual graph of the Farey triangulation so that the set of complementary regions $ \Omega$ of $\T$ is identified with $\QPo$. 
The trace labeling $\mu$ associated to a character $\rho \in \HomF$ is a map from $\Omega$ to $\C$. 
The equivalence classes of ends of the tree $\T$  identifies 
with $\RPo$, where the identification is two-to-one on the rationals 
(corresponding to the two ends of the alternating 
geodesic about $Z \in \Omega$) and one-to-one on the irrationals. 
We identify 
the ends of $\T$ with $\RPo$. 
The topology of $\RPo$ induces a topology on the set of ends.
The cyclic ordering on $\RPo$ induces a cyclic ordering on the set of ends. 

Let $\lambda$ denote an end of $\T$. 
Then, following \cite{MR2405161}, call $\lambda$  an {\em end invariant\/} of $\rho$ if,
for some constant $K>0$ and infinitely many distinct $Z_n \in \Omega$ adjacent to $\lambda$,
\[ \vert\mu(Z_n)\vert <K. \]  
Denote the set of end invariants of $\mu$ by $\EndInv \subset \RPo$.





For characters on $\kCk$ when $k>2$, at most one end invariant exists.
This follows from the fact that the tree is well-directed.
Note also that,  in this case, $\B'=\B$ since for all $Z \in \Omega$ with real trace, $|\mu(Z)|>2$.

As proved in \cite{MR2405161}, $\mathcal{E}(\mu)=\emptyset$ if and only if
$\mu\in\B=\B'$. 
Otherwise $\mathcal{E}(\mu)= \{\lambda\}$ for a unique  $\lambda\in\RPo$.  %
End invariants of characters on the boundary of $\B$ correspond to totally odd rationals.
In particular end invariants of characters corresponding to points on the lines of slope $m^{\pm}$
correspond to $1\in\QPo$
These lines bound the two sectors $\B_0$ defined by 
\[  2 < z < \sqrt{k+2},\] 
or, equivalently,  
\[ m^- < y/x < m^+\] 
The other components of $\B$
correspond to the images of $\B_0$ under $\gamma\in\Gamma$, and are indexed
by the totally odd rational numbers $\gamma(1)$. Points on the boundary of these components correpond to characters with end invariant the corresponding totally odd rational numbers $\gamma(1)$.

The complement of $\Pixy(\B)$ in $\R^2 \setminus \{(0,0)\}$ 
consists of uncountably many connected components easily 
described in terms of the end invariants: 
Components 
with end invariant indexed by totally odd $p/q \in \Q$ ($p,q$ odd)
correspond to the boundaries of the components of $\Pixy(\B)$, 
Components 
with end invariant indexed by $p/q \in \Q$ with either $p$ or $q$ even correspond to the the exceptional characters $\Gamma_{\Coloring}(0, -iy,z)$.
Components 
with end invariant indexed by an irrational correspond to characters whose 
trace labeling $\mu$ has directed tree $\vT$ 
directed towards a fixed irrational end.
%
%

As in \S\ref{sec:FaithfulAction}, 
three involutions  $\II_1,\II_2,\ \sigma_1\circ\II_3\ \subset\Gamma$ act on 
$\R^3$ by:
\[
\bmatrix x \\ y \\ z  \endbmatrix \xmapsto{~\II_1~}
\bmatrix yz-x \\ y \\ z  \endbmatrix,\qquad 
\bmatrix x \\ y \\ z  \endbmatrix \xmapsto{~\II_2~}
\bmatrix x \\ xz-y \\ z  \endbmatrix, \qquad
\bmatrix x \\ y \\ z  \endbmatrix\xmapsto{~\sigma_1\circ\II_3~}
\bmatrix x \\ -y \\ xy+z  \endbmatrix
\]
(We compose the Vieta involution $\II_3$ with the sign-change $\sigma_1$
to preserve each of the two components of $\kCk$.)
Each involution fixes the line $\{0\}\times\{0\}\times\R$.
This line consists of dihedral characters.
Their derivatives at $p = (0,0,\pm\sqrt{k+2})$ determine three linear involutions
as follows. 
Since both $\{0\}\times\{0\}\times\R$ and $\kCk$ are invariant
and intersect transversely at $p$, the tangent spaces decompose as direct sums:
\begin{alignat*}{3}
T_\Delta(\R^3) & =  T_\Delta\Big(\kCk\Big)\; &&\oplus\; T_\Delta\Big(\{(0,0)\}\times\R\Big)  \\
& = \; \Big( \R^2 \oplus 0\Big) &&\bigoplus \; \Big( 0\oplus 0 \oplus \R\Big) 
\end{alignat*}
With respect to this decomposition, $\II_1,\II_2,\sigma_1\circ\II_3$ act infinitesimally at $\Delta$ by
their derivatives:
\begin{align*}
D_\Delta(\II_1) & = \bmatrix -1 & \sqrt{k+2}\\ 0  & 1 \endbmatrix 
\oplus \bmatrix1\endbmatrix 
\\
D_\Delta(\II_2) & = \bmatrix 1 & 0 \\ \sqrt{k+2} & -1 \endbmatrix 
\oplus \bmatrix1\endbmatrix 
\\
D_\Delta(\sigma_1\circ\II_3) &=\bmatrix 1 & 0 \\ 0 & -1 \endbmatrix
\oplus \bmatrix1\endbmatrix  
\end{align*}
Projectivizing these $2\times 2$ matrices yield  reflections of the 
Poincar\'e upper halfplane $\Ht$:
\[
z \longmapsto \sqrt{k+2} - \bar{z} , \qquad
z \longmapsto \frac{\bar{z}}{\sqrt{k+2}\ \bar{z}-1} , \qquad
z \longmapsto -\bar{z}, \]
respectively. 
These are reflections in the geodesics with pairs of endpoints:
\[
\{ \infty, \sqrt{k+2}/2 \},\qquad
\{ 2/\sqrt{k+2},0  \},\qquad
\{ 0,\infty \}, \]
respectively. 
These generate a Fuchsian group whose quotient is doubly covered by a disc with
two cusps and geodesic boundary of length $2\cosh^{-1}(k/2)$.
(Compare Figure~\ref{fig:FuchsianGroup}.)
\begin{figure}[thb]
\centerline{\includegraphics[scale=.7]{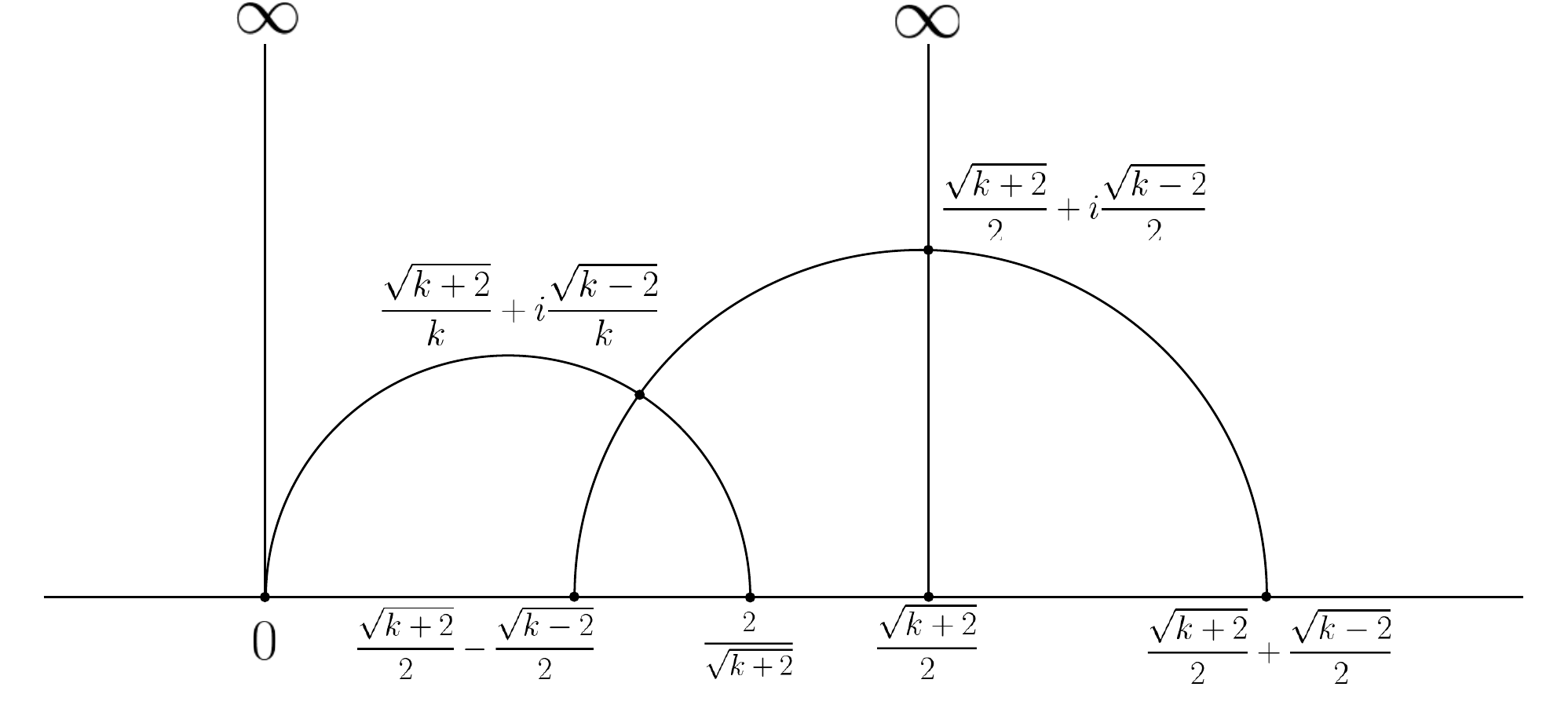}}
\caption{The Fuchsian group generated by Vieta involutions}
\label{fig:FuchsianGroup}
\end{figure}

\begin{figure}[thb]
\centerline{\includegraphics[width=\textwidth]{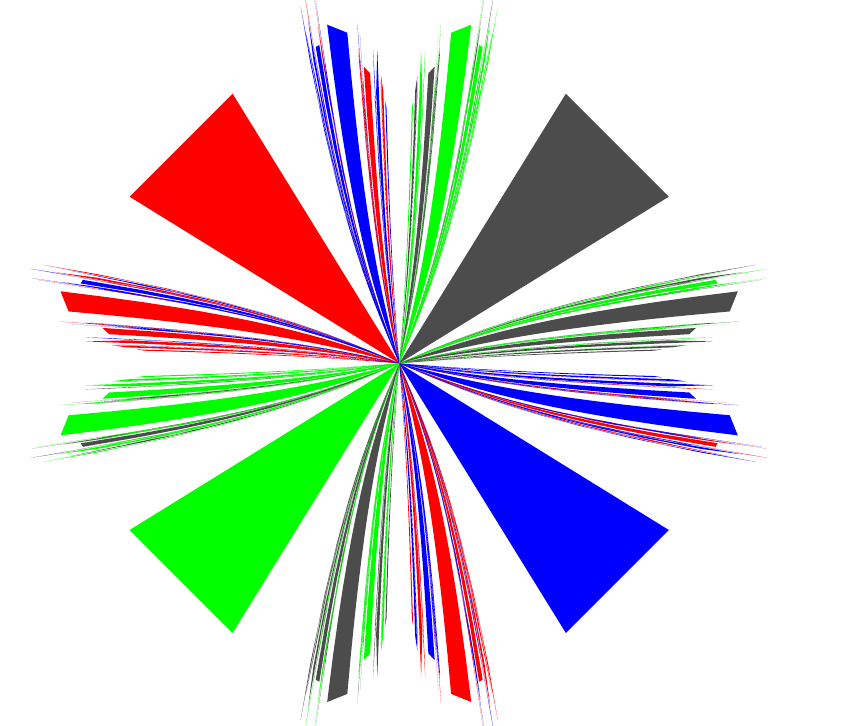}}
\caption{$xy$-projection of domain of discontinuity for $k=3$}
\label{fig:xyDisc}
\end{figure}

The limit set $\Lambda$ is a Cantor set in $\RPo$. 
Its complement is a countable collection of open intervals,
corresponding to the orbit of  $\Fricke'(C_{1,1})$.
The fixed points of the composition 
\[D_\Delta(\II_1)\circ D_\Delta(\II_2) \]
exactly correspond to the slopes of the lines $y=m^\pm x$ 
bounding $\Fricke'(C_{1,1})$ 
(where $m^\pm$ is defined as in \eqref{eq:DefSlopes}).
The points of $\Lambda$ correspond to 
the components of the complement of the domain of discontinuity.

For an exceptional imaginary trace labeling 
\[ \mu\ \in \ \hGamma \cdot (0, iy, z), \]
where $y \neq 0$ and $|z|>2$, only one region $X\in\Omega$ with
$\mu(X) = 0$ exists. 
Furthermore all edges of $\vec T_{\mu}$ 
not adjacent to $X$ point towards $X$. 
Edges adjacent to $X$ are indecisive and can therefore be directed arbitrarily. 

%
\subsection{Density of the Bowditch set}
By Bowditch~\cite{MR1643429} and Tan-Wong-Zhang~\cite{MR2370281}, 
$\B$ is open. 
Now we prove that it is dense:

\begin{thm}\label{thm:Density}
The  interior of the complement of the Bowdtich set $\B$
is empty.
\end{thm}
The proof involves the $\Gamma$-invariant smooth measure
provided by \linebreak Lemma~\ref{lem:AreaForm} arising from the 
Poisson bivector $\BB_\Coloring$ defined in \eqref{eq:RealPoissonStructure}.
Denote the measure of a Lebesgue-measurable subset $S\subset\kCk$ by $\area(S)$. 

\begin{proof}
Suppose not; let 
\[
U\subset\kCk \setminus \big( 
\B \cup \{(0,0,z)\} \big)
\] 
be a nonempty connected open set.


Since $U\cap
\B =\emptyset$, a character $u\in U$ determines a descending path
\begin{equation*}
{\vec P}(u) = \big( u =  u_0 \xrightarrow{f_0} u_1 \xrightarrow{f_1}  \cdots   \xrightarrow{f_{n-1}} u_n \xrightarrow{f_n} \cdots
 \big) \end{equation*}
 where each $f_i$ corresponds to $\II_{j(i)}$ for some $j(i) = 1,2,3$ and $j(i)\neq j(i+1)$.

Moreover, $\vec P$ depends continuously on the character $u$. 
Suppose $\mu$ is a character lying two different translates of $\B_0$.
Then there are two different descending paths leading to two different $Z,Z'\in \Omega$ 
with $\vert z\vert,\vert z'\vert <\sqrt{k+2}$. 
That would imply existence of two sinks, a contradiction.
Therefore $\vec P$ depends continuously on the character $u$.
Since $U$ is connected, $\vec P$ is independent of $u\in U$.
 
 Then $\vec P$ determines a sequence 
 \[
 \gamma_n := \II_{j(n)} \circ \dots \circ \II_{j(1)} \circ \II_{j(0)} \in \Gamma. \]

\begin{lem}
The  $\gamma_n(U)$ lie in finitely many connected components of the complement  of $\B'$. 
\end{lem}
\begin{proof}
We may assume that $U$ lies in 
\[ \Box_R := \Big\{ \big(x,y,z_\pm(x,y)\big) \in \kCk \ \Big|\  \vert x\vert, \vert y\vert \le R \Big\} \]
for some $R > 0$. 
Observe that $\area(\Box_R) < \infty$.
By \eqref{eq:DescendingDecreasing}, 
the coordinates $\vert x\vert, \vert y\vert$ are non-increasing,
so that $\gamma_n(U) \subset \Box_R$ for all $n > 0$.

Suppose first that the $\gamma_n(U)$ lie in infinitely many components of the complement of $\B'$. Choosing a subsequence if necessary, we may assume these components are all distinct. 
Then  the $\gamma_n(U)$ are all disjoint. 
Since 
$\area\big(\gamma_n(U)\big) = \area(U) > 0$,
\[
\area(\Box_R) \ge \sum_{n=1}^\infty \area\big(\gamma_n(U)\big)  = \infty,\]
a contradiction. 
\end{proof}
\begin{lem}
The descending path corresponding to the component of the complement containing $U$ is infinite.
\end{lem}
\begin{proof}
Otherwise $U$ is contained in a component which is the image of either the $x$ or $y$ axis under 
some element of $\Gamma$.
Since these sets have measure zero, this contradicts $\area(U)>0$. 
\end{proof}

\begin{lem}
The descending path is eventually periodic.
\end{lem}
\begin{proof}
Since there are only finitely many components, 
$\gamma_n(U)$ and $\gamma_m(U)$ lie in the same component of the complement for some $n>m$. 
Replacing $U$ by $\gamma_m(U)$, we may assume that $\gamma_N(U)$ and $U$ 
lie in the same component for some $N>1$. 
Since all elements of the same component of the complement have the same infinite descending path, 
the infinite descending path for $U$ must be periodic.  
\end{proof}
\noindent
Thus we may write $\gamma_{mn} = (\II_{j_1}\II_{j_2}\dots\II_{j_m})^n$
where $\II_{j_1}\II_{j_2}\dots\II_{j_m}$ is the {\em period.\/} 
\begin{lem}  The descending path cannot consist entirely of imaginary edges.
That is, at least one  $j_n$ must equal $3$. \end{lem}
\begin{proof}
Otherwise the (minimal) period would be $\II_1\II_2$ or $\II_2\II_1$. 
Then $U$ would be contained in the two lines which bound the Fricke space.
Since these lines have measure zero, this contradicts $\mu(U)>0$. 
\end{proof}
\noindent
Thus we may assume that $j_1 = 3$, that is, the period begins with $\II_3$ and the path begins with an
$\R$-edge. Equivalently the vertex corresponding to  $(x,y,z)\in U$ is {\em negative.}

Choose $\delta > 0$ sufficiently small so that the lift 
\[\Nd := \big(\Pixy\big)\inv \big(\mathfrak{N}_{\delta}(0,0)) \big) \]
of the Euclidean $\delta$-ball 
${\mathfrak N}_\delta(0,0) \subset \R^2$ to $\kCk$ has area less than $\area(U)/2$.
Then
\[
\area(U\setminus\Nd) \ge \area(U) - \area(\Nd) > \area(\Nd). \]
Thus by replacing $U$ by $U\setminus\Nd$ we may suppose that
$U$ is disjoint from $\Nd$ and 
$\area(\Nd) < \area(U)$. 

\begin{lem}\label{lem:EventuallyDisjoint}
$\gamma_n(U) \setminus \Nd = \emptyset$ for $n$ sufficiently large. 
\end{lem}
\begin{proof}
First bound $\gamma_n(U)\setminus\Nd$ away from the coordinate axes:
The coordinate axes describe four components of the complement
of $\B$, so we may assume the infinite descending path misses these axes.
Furthermore 
$\gamma_n(U)\setminus\Nd$ 
lie in finitely many closed subsets disjoint from these axes,
and inside the compact subset $\Box_R$. 
Thus 
we may assume that the coordinates
of $\gamma_n(U)\setminus\Nd$ are bounded away from $0$. 
That is, there exists $\epsilon > 0$ with the following properties:
For each $u_n\in \gamma_n(U)$, 
either \[
\vert x_n\vert \ge \epsilon, \quad
\vert y_n\vert \ge \epsilon, \]
or $u_n\in \Nd$, in which case we are done.
Furthermore by Lemma~\ref{lem:zBiggerThan2}, $\vert z_n \vert > 2$ for all  $n$. 

Now \eqref{eq:zPlusOrMinus} implies that if $\vert x\vert, \vert y\vert \le R$,
then
\[ 
\vert z_\pm(x,y) \vert \le C \]
where
\[
C := \frac12 \Big( R^2 + \sqrt{(R^2+4)^2 + 4 (k-2)} \Big). \]
 The infinite descending path contains infinitely many words containing $\II_3$. Otherwise eventually it becomes an alternating sequence of 
 $\II_1$ and $\II_2$, which corresponds to totally odd rationals.  
For each $\gamma_n$, the value of the $z$-coordinate either remains constant or changes by:
\[ 
\vert z_{n+1} \vert = 
\vert - x_n y_n - z_n \vert \le \vert z_n \vert - \epsilon^2. \] 
Furthermore, since the period contains an $\II_3$, the value of $z$ changes at least once over each period. 
The value of $\vert z_n\vert$ therefore decreases by at least $\epsilon^2$ over each period.
Thus, for  $n>mC/\epsilon^2$ where $m$ is  the period, 
\[ \gamma_n(u) \in \Nd,\]
since  otherwise $\vert z_n \vert <0$, a contradiction.
\end{proof}
\begin{proof}[Conclusion of proof of Theorem~\ref{thm:Density}]
Lemma~\ref{lem:EventuallyDisjoint} implies that 
\[ \gamma_n(U) \subset \Nd \]
for sufficiently large $n$. Thus
\[
\area(U) = 
\area\big(\gamma_n(U)\big) \le \area(\Nd) 
< \area(U), \]
a contradiction. 
\end{proof}\let\qed\relax\end{proof} 

\begin{conj} If $k>2$, then $\area\big(\kCk\setminus\gFO{\Coo}\big) = 0$.
\end{conj}

%
%

\section{Imaginary characters with $k<2$.}\label{sec:ImagCharkLessThanTwo}

For an imaginary trace labeling $\mu$ with $k<2$,  
possibly  $z=\mu(Z) \in (-2,2)$ for some $Z \in \O_{\R}$. 
Then the action of $\Gamma$ on an open subset of $\kCk$ is ergodic.
In particular if $-14 \le k < 2$, 
the $\Gamma$-action is ergodic on all of $\kCk$. 

If $k < -14$, then the Fricke orbit is a disjoint union of wandering domains,
and the action on the complement is ergodic.
For $k=-14$, the Fricke orbit is a countable discrete set, 
upon whose complement the action is ergodic.

\subsection{Existence of elliptics}

The first step will be to show that some primitive element is mapped
to an elliptic.  

\begin{prop}\label{prop:eliipticnearepsilon}
Suppose that $k<2$. 
Then  there exists $\epsilon = \epsilon(k) > 0$ 
with the following property:
For a non-exceptional trace labeling $\mu$ with 
\[ \vert\mu(Y)\vert =\vert iy\vert <\epsilon\]  
for some 
\[ Y\in \O \setminus \O_{\R}, \] 
then 
$Y$ has a neighbor $Z \in \O_{\R}$  such that 
\[ \vert \mu(Z)\vert = \vert z\vert\  < 2.\]
\end{prop}

\begin{proof}
For 
$y$ sufficiently small, 
some $\vert\mu(Z_n)\vert < 2$ 
where 
\[ \dots, Z_{n-1}, X_n, Z_n, X_{n+1}, \dots \]
denotes the sequence of regions abutting $Y$.
Since $\Coloring(Z_n) = 0$ and $\Coloring(X_n) = 1$, the corresponding sequence
of traces alternates between real and purely imaginary:
\begin{align*}
\mu( X_{n} ) & = ix_n \\
\mu( Z_{n} ) & = z_n \end{align*}
where $x_n,z_n\in\R$.
(Compare Figure~\ref{fig:GeodesicAbuttingZ}.)
%

Analogously to \eqref{eq:KappaUpsilonInTermsOfz},
for fixed $y$, 
the $y$-level sets of $\kCk$ are hyperbolas  
\begin{align}\label{eq:KappaUpsilonInTermsOfy}
H_{k,y} :=\  &\kCk \cap \;
\big(\R\times\{y\}\times\R\big) \notag \\
 = & (\QQ_y')^{-1}( y^2 - k + 2)
\end{align}
where $\QQ_{y}'$ denotes the quadratic form
\begin{equation}\label{eq:yQuadraticForm}
 \QQ_y'(x,z) := z^2 + y x z - x^2. \end{equation}
Choose $\beta\in\R$ such that the trace $i y = (i\beta) + (i\beta)^{-1}$,
that is, \[   y = \beta - \beta^{-1}. \]
By Edge Relation \eqref{eq:edgerelation},
the trace sequence \[\dots,z_{n-1},i x_n, z_n, i x_{n+1},\dots\] satisfies:
\begin{align}\label{eq:GeodesicForSmallk}
\bmatrix i x_{n+1} \\ z_{n+1} \endbmatrix &= 
\MM_y \bmatrix z_n \\ i x_{n+1} \endbmatrix  \notag\\
\bmatrix z_n \\ i x_{n+1} \endbmatrix &= 
\MM_y \bmatrix i x_n \\ z_n \endbmatrix  
\end{align}
where
\begin{equation}\label{eq:TwoByTwoMatrixy}
\MM_y := \bmatrix 0 & 1 \\ -1 & i y\endbmatrix.
\end{equation}  
(Compare Proposition~\ref{prop:abut}.)
Moreover  $\MM_y$ has eigenvalues $\pm i \beta$
and preserves the quadratic form $\QQ_y'$ defined in \eqref{eq:yQuadraticForm}.
Furthermore 
$\QQ_y'$ factors as the product
\[ \QQ_y'(x,z) =   l_\beta(x,z) l_\beta'(x,z) \]
of homogenous linear functions
\begin{align*}
l_\beta(x,z) & := z + \beta x \\
l_\beta'(x,z) & := z - \beta^{-1} x.
\end{align*}
These linear factors are eigen-covectors under $(\MM_y)^2$:
\begin{align*}
l_\beta \circ (\MM_y)^2 &= -\beta^{-2} l_\beta \\
l_\beta' \circ (\MM_y)^2 &= -\beta^2 l_\beta' 
\end{align*}
The points $(x_n,z_n)$ form a lattice in the hyperbola $H_{k,y}$
(in the sense of \S\ref{sec:AlternatingGeodesics}) 
defined by \eqref{eq:KappaUpsilonInTermsOfy}.

For $\vert y \vert < \sqrt{k+2}$, 
this hyperbola  meets the strip defined by $\vert z \vert < 2$.
(Since $k < 2$, the level value is negative: $ 2 - k + y^2 < 0$.)
The multiplicative increment of the lattice formed by the points $(x_n,z_n)\in H_{k,y}$ 
equals $\beta$. 
For $y$ small, that is, when $\beta\sim 1$,
the lattice points are so closely spaced that at least one lattice
point lies in the strip. 

\end{proof}

%

\subsection{Alternating geodesics 
for $k<2$}
Suppose that $\mu$ is an unexceptional imaginary trace labeling with $k<2$ and let $Z\in \O_{\R}$.
We refine Proposition~\ref{prop:EdgesAboutX} 
when $k<2$. 
Propositions~\ref{prop:mergesaboutZtwoa},\ref{prop:mergesaboutZtwob} 
sharpen
Propositions~\ref{prop:mergesaboutZa},\ref{prop:mergesaboutZb},\ref{prop:mergesaboutZc}.

\begin{prop} \label{prop:mergesaboutZtwoa}
Suppose that $\vert z\vert < 2$.
Then $Z$ corresponds to an elliptic element and
$\vert y_n\vert$ is bounded for  all $n \in \Z$. 
The points $(y_n,y_{n+1})\in\R^2$ lie on an ellipse in $\R^2$.
They comprise a  dense subset of the ellipse 
if and only if $Z$ corresponds to an 
elliptic element of infinite order, that is, if $z =2 \cos(\alpha \pi)$ 
where $\alpha$ is irrational.
\end{prop}
\begin{proof}
Apply Corollary \ref{cor:abutgrowth}(b).
\end{proof}

\begin{prop} \label{prop:mergesaboutZtwob}
Suppose that $ \vert z\vert\ge 2$.
\begin{itemize}
\item $v_m$ is negative for a unique $m \in \Z$; 
\item Every $e_n$ points towards $v_m$;
\item Every $\vedge{i}$, with $i\neq m$,  points towards $Z$;
\item The edge $\vedge{m}$ points away from $Z$ if 
\[
\vert z \vert>\vert -y_my_{m+1}-z\vert,\] in which case $v_m$ is a merge.
\item
The edge $\vedge{m}$ points towards $Z$ if  
\[
\vert z \vert<\vert -y_my_{m+1}-z\vert,\]  
in which case $v_m$ is a sink. 
\item
If  \[ \vert z \vert=\vert -y_my_{m+1}-z\vert,\]
then $\vedge{m}$ is indecisive, with  two possible choices of the sink. 
\end{itemize}
\end{prop}

%
%
%

\begin{proof}
We only consider the case $\vert z\vert >2$. 
The case when $|z|=2$ is similar, 
using Proposition~\ref{prop:abutPar} instead of Proposition~\ref{prop:abut}. 
As in the proof of Proposition~\ref{prop:EdgesAboutX} (a) assume 
\[ 
z>2,\quad \lambda>1. \]
Since $AB>0$, further assume 
\[ A=ia,\qquad B=ib \]
where $a>0$ and $b<0$. 
Hence $y_n$ increases monotonically and  
\[
\lim_{n\to -\infty} y_n = - \infty, \qquad
\lim_{n\to \infty} y_n =  \infty. \]
\noindent
Re-indexing if necessary, assume  $y_n>0$ for all $n \ge 0$ and $y_n<0$ for $n <0$. 
Hence $v_0$ is negative and $v_n$ is positive for all $n\neq 0$.
Edge Relation~\eqref{eq:edgerelation} imply
that all the $e_n$'s point decisively towards $v_0$.
Lemma~\ref{lem:PosNegVert} implies that each 
$\vedge{n}$'s point decisively towards $Z$ for $n \neq 0$. 
For the case of $\vedge{0}$, 
\begin{align*} 
z_0'\; & =\ -y_{-1}y_0-z, \\ 
-y_{-1}\ y_0\; & > \qquad 0. \end{align*}
If $z_0'<0$ then $\vert z \vert >\vert z_0'\vert$ so $\vedge{0}$ points towards $Z_0'$ and $v_0$ is a merge. Similarly, if \[ z>z_0'>0, \] 
then $\vedge{0}$ is also directed towards $Z_0'$ and $v_0$ is a merge. 
If $z_0'>z>0$ then $\vedge{0}$ points towards $Z$ and $v_0$ is a sink.  
This only happens if $k< -14$, in the case of the Fricke space $\Fricke(C_{0,2})$ 
(compare \cite{{MR2705402}}).
Finally, if $z_0'=z$ then $\vedge{0}$ is indecisive and either vertex at its endpoints can be chosen to be sinks.
\end{proof}

\subsection{Descending Paths}
We use Proposition~\ref{prop:eliipticnearepsilon} to find a simple closed
curve $\gamma$ such that $\rho(\gamma)$ is elliptic. 
Equivalently, we find a complementary region
$\omega\in\Omega_\R$ such that $\mu(\omega) \in (-2,2)$. 
We accomplish this by following a descending path
to such an elliptic region. Otherwise, we fall into a sink.
Sinks correspond to the Fricke orbit of $C_{0,2}$.
This case is analogous to the case treated in Goldman~\cite{MR2026539} when
$k > 18$, and the Fricke space of $\Sigma_{0,3}$ appears.
The following key result implies that an {\em infinite\/} descending path eventually contains an elliptic or meets the Fricke orbit:

\begin{prop}\label{prop:descendingpath}
Suppose that $\mu$ is a trace labeling with $\kappa(\mu)<2$.
Let 
\[ \vec P = \big( v_0 \xrightarrow{f_0} v_1 \xrightarrow{f_1}  \cdots   \xrightarrow{f_{n-1}} v_n \xrightarrow{f_n} \cdots
 \big) \]
be a descending path in $\vec T_\mu$. Then there exists a vertex
$v_m = v(X,Y,Z)$ with either:
\begin{itemize}
\item 
$\vert\mu(Z)\vert<2$, or: 
\item 
$v(X,Y,Z)$ is a sink with $|\mu(Z)|\ge 2$. 
\end{itemize}
In the latter case, if 
\begin{align*}
\mu(X) &=ix, \\ 
\mu(Y) & =iy, \\ 
\mu(Z) &=z, \end{align*}
where $x,y,z \in \R$ with $z\ge 2$, then $-xy-z \ge z$.
\end{prop}

\noindent The proof breaks into a series of lemmas. 
Consider an infinite descending path $\vec P$ of \eqref{eq:Path}
such that $\vert z_n\vert \ge 2$ for every vertex $u_n(X,Y,Z)$ on this path.
\begin{lem}\label{lem:InfManyColoringEdges}
Infinitely many $f_n$ are $\R$-edges.
\end{lem}
\begin{proof}
Suppose  $P$ contains only finitely many $\R$-edges.
Then there exists $N$ such that edge $f_n$ is a $i\R$-edge for $n\ge N$. 
Furthermore, 
Proposition~\ref{prop:ColoredEdges} implies that 
$f_n$ lies in a geodesic bounding a complementary region $Z\in\O_\R$
for all $n\ge N$.

By assumption $\vert z \vert \ge 2$.
By Proposition~\ref{prop:mergesaboutZtwob}, 
each $f_m$ points towards some $u\in\V$,
a contradiction since $P$ is a descending path.
\end{proof}

\begin{lem}\label{lem:DefiniteDecrease}
Choose $\epsilon(k)>0$ as in Proposition~\ref{prop:eliipticnearepsilon}.
For each $\R$-edge $f_n$ as above,
let $X,Y,Z,W$ be the complementary regions so that
\[ f_n = e^{X,Y}(Z\to W) \]with corresponding trace labels $ix, iy, z, w$.
If $\vert x\vert, \vert y\vert > \epsilon(k)$ and 
$\vert z\vert, \vert w\vert \ge 2$, 
then 
\[ \vert w \vert < \vert z \vert - \epsilon(k)^2.\]
\end{lem}
\begin{proof}
Since $f_n$ is a $\R$-edge directed away from 
the vertex $v_n(X,Y,Z)$,
Proposition~\ref{prop:mergesaboutZtwob} implies
$v_n(X,Y,Z)$ is negative. 
Therefore $xy$ and $z$ have opposite signs.

We claim that $v(X,Y,W)= v_{n+1}$ must be positive.
Otherwise, Proposition~\ref{prop:mergesaboutZtwob} 
applied to the neighbors of $W$ implies that
$v(X,Y,W)= v_{n+1}$ is a sink. This contradicts $P$
being an infinite descending path. 

Thus $xy$ and $w$ have the same sign, but $z$ has the opposite sign.
Applying a sign-change if necessary, assume $z > 0$. 
Then $xy < 0$ and $w < 0$. Since $w = -xy-z$, and $x y < - \epsilon(k)^2$,
\[
\vert w \vert =  -w = xy + z  < -\epsilon(k)^2 + z = -\epsilon(k)^2 + \vert z\vert 
\]
as claimed.
\end{proof}

\begin{lem}
The infinite descending path contains 
a vertex $v_n$ such that:
\begin{itemize}
\item $\vert z_n\vert <  2$,   or:
\item One of $\vert x_n\vert, \vert y_n\vert$ is $<  \epsilon(k)$.
\end{itemize}
\end{lem}
\begin{proof}
Suppose, for all vertices $v_n$, that 
$\vert x_n\vert, \vert y_n\vert \ge  \epsilon(k)$.
Lemma~\ref{lem:InfManyColoringEdges} 
guarantees an infinite set of $\R$-edges 
in this descending path. 
Apply   Lemma~\ref{lem:DefiniteDecrease} 
to this, 
obtaining a subsequence $(x_n,y_n,z_n)$.
Then 
$\vert z_n \vert$ decreases uniformly by $\epsilon(k)^2$.
Eventually $\vert z_n \vert < 2$ as desired.
\end{proof}

If $\vert z_n\vert < 2$, then we're done. 
Therefore we may assume that $\vert x_n\vert <  \epsilon(k)$
(the case that $\vert y_n\vert<\epsilon(k)$ being completely analogous).

Proposition~\ref{prop:eliipticnearepsilon} implies that 
$X_n$ has a neighboring region $W$ such that $|w| <2$. 
Let $v'(X_n,Y',Z')$ be the vertex on $\partial X_n$ 
nearest to
$v(X_n,Y_n,Z_n)$, and such that $\vert z' \vert <2$. 
We claim that the path descends in a unique direction from $v(X_n,Y_n,Z_n)$ to 
$v'(X_n,Y',Z')$. 

First note that by construction, if $v(X,Y,Z)$ is a vertex between these two vertices, 
then $|z| \ge 2$ and hence $v(X,Y,Z)$ is a merge or a sink. 
Furthermore, the edge $e^{X_n,Y'}$ which lies on the path between these two  vertices  points towards $v'(X_n,Y',Z')$, again by construction. 
Therefore 
all the edges from $v(X_n,Y_n,Z_n)$ to $v'(X_n,Y',Z')$ point towards 
$v'(X_n,Y',Z')$ 
and this descending path is unique from $v(X_n,Y_n,Z_n)$. 
Finally, this leads to a contradiction as the descending path now meets a region $Z'$ with 
$\vert z' \vert <2$.

\begin{prop}
Suppose that $(ix,iy,z)$ is a purely imaginary character with $k = \kappa(ix,iy,z) < 2$. 
Consider the corresponding directed tree $\vT$. 
If both $z$ and $z' :=  -xy - z$ are $ \ge 2$, then the edge $e$
between $Z$ and $Z'$ contains a sink. 
All the edges adjacent to $e$ are merges and point towards $e$.
\end{prop}
\begin{proof}
By Edge Relation~ \eqref{eq:edgerelation}, 
$x y = - (z + z') \le -4$.  Furthermore $k < -2 - z z'  \le -6$.
In that case $v(ix,iy,z)$ is a {\em sink.\/} 
To this end, define $y' := xz - y$. 
We show that $\vert y' \vert > \vert y\vert$ as follows.
By a sign-change, we can assume that $y> 0$. 
Since  
\[ 4 \le z + z' = -xy, \]
$x < 0$ and
\[ y' = xz - y < 0.\]
Thus 
\[ \vert y' \vert = - y'  = y - x z > y = \vert y \vert \]
as claimed. 

This argument implies that the directed edge  $e = e^{X,Y}(Z, Z')$
points towards $v(X,Y,Z)$. 
Similarly the other three edges adjacent to the edge connecting Z and Z' point
towards $v(X,Y,Z)$. 
Thus all four edges adjacent to $e$ point inward.

If this edge is decisive, one of these two vertices must be a sink. 
Otherwise the edge is indecisive, and either of two vertices will
be a sink. Changing the direction establishes that other vertex as a sink. 
\end{proof}

\begin{proof}[Conclusion of proof of Proposition~\ref{prop:ProperAction2} when $k<2$]
\qquad\qquad\qquad\qquad
\smallskip

 \noindent
If $k < 2$ and $\vert\mu(Z)\vert \ge 2$ for a sink $v(X,Y,Z)$, 
Proposition~\ref{prop:mergesaboutZtwob} 
implies that 
$\vT$ is well-directed towards $v(X,Y,Z)$ and $\mu\in\B'$.

Otherwise, Proposition \ref{prop:descendingpath} 
implies $\vert\mu(Z)\vert<2$ for some $Z \in \O_{\R}$, 
so 
$\mu \notin \B$. 
This completes the proof of Proposition~\ref{prop:ProperAction2} when $k<2$.
\end{proof} 
\subsection{Ergodicity}

We now complete the proof that the action of $\Gamma$ is ergodic
on the complement of the Fricke orbit $\FO{\Czt}$ when $k < 2$.


\begin{prop}\label{prop:Ergodicity}
Suppose that  $\mu$ is a trace labeling with $k<2$.
Suppose that $Z \in \O_{\R}$ with  $\vert z \vert <2$.
Then $\Gamma$ acts ergodically on the complement
\[
\kCk \setminus\FO{\Czt}. \]
\end{prop}

\noindent
The complementary domain $Z$ corresponds to a primitive element in $\Ft$
and the corresponding Nielsen move (or Dehn twist) $\nu_Z$ defines an automorphism
of $\Ft$:
\begin{align*}
\Ft &\xrightarrow{~\nu_Z~} \Ft \\
X & \longmapsto  Z X = Y^{-1} \\
Y & \longmapsto  Y Z^{-1}  = Y^2 X \\
Z  & \longmapsto  Z  = Y^{-1} X^{-1} 
\end{align*}
which induces the polynomial automorphism of $\RepF$:
\[ 
\bmatrix x \\ y \\ z \endbmatrix \xmapsto{~(\nu_Z)_*}
\bmatrix x z - y \\ x \\ z \endbmatrix.  \] 
The induced map $(\nu_Z)_*$ preserves the Hamiltonian flow $\Ham(\mu_Z)$ of the trace function, whose orbits when $\vert z \vert < 2$ are ellipses. 
In terms of the parametrization of these orbits, 
$(\nu_Z)_*$ acts by translation by $2 \cos^{-1}(\mu_Z/2)$, 
and when $\cos^{-1}(\mu_Z/2)\notin \Q \pi$, acts by an irrational rotation of $S^1$.
In particular this action is ergodic with respect to the natural Lebesgue measure
on this orbit. 
By ergodic decomposition of the invariant area form with respect to this foliation,
every $(\nu_Z)_*$-invariant measurable function will be 
almost everywhere  $\Ham(\mu_Z)$-invariant. (This is the same technique used in 
\cite{MR1491446, 
MR2026539,
MR2807844, 
MR2346275}  
for related results.)

The orbits of $\Ham(\mu_Z)$ are ellipses on the level plane $z = z_0$ centered
at the origin. These ellipses contain the orbit of $\nu_Z$ as a countable dense
subset. Thus for each $\epsilon > 0$, for some $N$, 
the coordinates $(x_n,y_n,z_0)$ of $\nu_Z^n(x,y,z_0)$ satisfy
\begin{align*}
\vert x_n \vert &< \epsilon, \\
\vert y_n \vert &< M(k,z_0) 
\end{align*}
where $M = M(k,z_0)$ depends just on $k$ and $z_0$.

The trace function $\mu_{Z'}$ satisfies
\[
\vert \mu_{Z'} \vert = \vert \mu_{X} \mu_{Y} -\mu_{Z}  \vert = 
\vert - x_n y_n - z \vert < \vert z\vert + M \epsilon \]
so for some $n$ we may assume $Z'$ is elliptic.

Away from a nullset, the trace function $\mu_{Z'}$ also defines a 
Hamiltonian flow whose orbits are ellipses. 
We show that  $\Ham(\mu_Z)$ and $\Ham(\mu_{Z'})$ span the tangent space
to the level surface at such a point.

By \eqref{eq:RealPoissonStructure},  the Hamiltonian vector field 
of $z$ equals
\begin{align*}
\Ham(\mu_Z) &=  \BB_\Coloring \cdot dz \\ & = 
(2 x - z y) \dd{y} + (-2 y + z x) \dd{x} \end{align*}
and the Hamiltonian vector field of $z'$ equals
\begin{align*}
\Ham(\mu_{Z'}) & =  \BB_\Coloring \cdot dz' \\ & = 
\Big( z x + (x^2 + 2) y \Big) \dd{x} +
\Big(- z y - (y^2 + 2) x \Big) \dd{y}  \\
& \qquad \qquad+
2 (x^2 - y^2) \dd{z} \end{align*}
since $dz' = - y\ dx - x\ dy - dz$.

At a point $p=(x,y,z)$, the vector field  $\Ham(\mu_Z)$ is zero if and only
if \[-2y + zx  = - 2 x + y z = 0.\] 
As in the proof of Lemma~\ref{lem:CriticalPointsOfKappaColoring},
if $z = 0$, then $p = (0,0,0)$ and $k=-2$.

Thus assume $z\neq 0$. Then if either $x = 0$ or $y= 0$, then both
$x = y = 0$ and $p$ is the dihedral character $\big(0,0,\pm\sqrt{k+2}\big)$.

Thus we can assume $x\neq 0,y\neq 0,z\neq 0$. 
As in the proof of Lemma~\ref{lem:CriticalPointsOfKappaColoring} again,
$x/y = y/x = 2/z$ and 
\[  p\ \in\  \Big\{ (x, x, 2) \ \Big|\  x\in\R \Big\}\  \bigcup\  
\Big\{ (x, -x, -2) \ \Big|\  x\in\R \Big\} \] 
which implies $k = 2$. 

When $x \neq \pm y$, then $x^2 - y^2 \neq 0$, then 
the coefficient of $\dd{z}$ in  $\Ham(\mu_{Z'})$ is nonzero, whereas 
the coefficient of $\dd{z}$ in  $\Ham(\mu_{Z})$ is zero.
Thus $\Ham(\mu_{Z})$  and $\Ham(\mu_{Z'})$ are linearly independent
unless $x = \pm y$, which describes a nullset in $\kCk$.
\section{Imaginary characters with $k=2$.}\label{sec:ImagCharkEqualTwo}

When $k=2$, the level set $\kCtwo$ admits a 
{\em rational parametrization\/} as in 
Goldman~\cite{MR2497777}.
These characters correspond to {\em reducible\/} representations, 
which may be taken to be diagonal matrices.  
The representation 
\begin{align}\label{eq:ReducibleRepresentation}
\rho(X) &:= 
i \bmatrix  e^{a/2} & 0 \\ 0 & -e^{-a/2} \endbmatrix \notag\\
\rho(Y) &:= 
i \bmatrix  e^{b/2} & 0 \\ 0 & -e^{-b/2} \endbmatrix \notag \\
\rho(Z) &:=  
\bmatrix  -e^{-(a+b)/2} & 0 \\ 0 & -e^{(a+b)/2} \endbmatrix 
\end{align}
corresponds to the character $(ix, iy, z)$ where 
\begin{align}\label{eq:ReducibleCharacter}
x & = 2 \sinh(a/2) 
\notag\\
y & = 2 \sinh(b/2) 
\notag\\
z & = -2 \cosh\big((a+b)/2\big) 
\end{align}
are real and satisfy the defining equation:
\[ x^2 + y^2 - z^2 + x y z = 4. \]
The resulting mapping 
\begin{align}\label{eq:Parametrization}
\R^2 &\xrightarrow{~\Psi~} \kCtwo \notag\\
(a,b) &\longmapsto (ix,iy,z) \end{align}
defines a diffeomorphism of $\R^2$ onto the component $\graph(z_-)$ of
$\kCtwo$.
\big(The composition $\sigma_1\circ\Psi$ 
(or $\sigma_2\circ\Psi$) defines a diffeomorphism of $\R^2$ 
with the other component $\graph(z_+)$ of $\kCtwo$.\big)

Geometrically, these characters correspond to actions which stabilize a line $\ell\subset\Ht$.
When $x\neq 0$, then $\rho(X)$ is a glide-reflection about $\ell$. 
Otherwise $\rho(X)$ is reflection about $\ell$.
Similarly $y\neq 0$ (respectively $y=0$) corresponds to the cas that  $\rho(Y)$ is a glide-reflection 
(reflection) about $\ell$. In these cases $\rho(Z)$ is either $\Id$ or a transvection in $\ell$.

Ergodicity of the $\Gamma$-action on $\kCtwo$ 
follows from Moore's ergodicity theorem~\cite{MR0193188} (see also Zimmer~\cite{MR776417}) 
as in Goldman~\cite{MR2264541} in the orientation-preserving case.

The mapping $\Psi$ is equivariant with respect to an action of $\GLtZc$ on $\R^2$.
Furthermore $\GLtZc$ is a non-uniform lattice in the Lie group $\SLtRpm$,
and $\SLtRpm$ acts transitively on $\R^2$ with isotropy subgroup
\[
N := \Big\{ \bmatrix 1 & x  \\ 0 & \pm 1 \endbmatrix  \; \Big|\;   x\in\R \Big\}. 
\]
Since $N\subset \SLtRpm$ is noncompact and $\SLtRpm/\GLtZc$ has finite Haar measure,
Moore's theorem implies the action of $\GLtZ$ on $\R^2$ is ergodic.
It follows that $\Gamma$ acts ergodically on $\kCtwo$.

\clearpage

\begin{figure}
\centerline{\includegraphics[scale=.4]{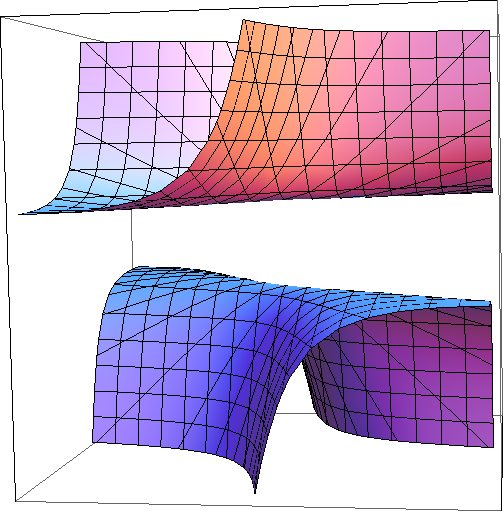}}
\bigskip
\centerline{\includegraphics[scale=.6]{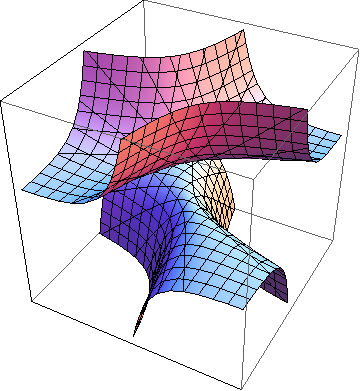}}
\caption[Two views of level set $k>2$]{
Two views of the two-component  level set for $k>2$.1
Each component projects diffeomorphically onto $\R^2$.}
\end{figure}


\begin{figure}
\centerline{\includegraphics[scale=.3]{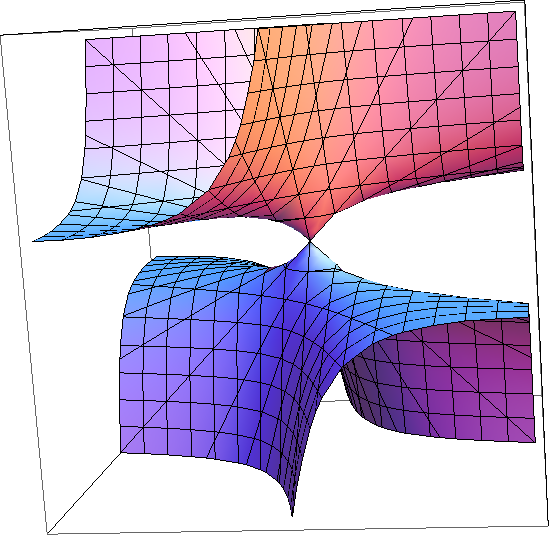}}
\centerline{\includegraphics[scale=.5]{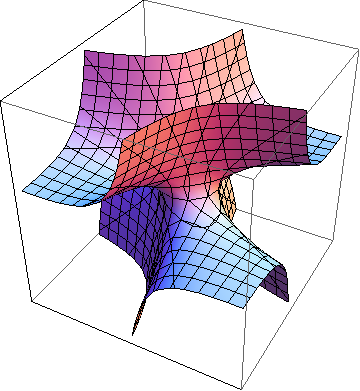}}
\caption
[Level surfaces for $k<2$]
{The first figure depicts the  connected, but singular,  level set, for $k = -2 <2$. This is the purely imaginary real form of the Markoff surface.
The second figure depicts the level set for $k=-10 < 2$, which is connected
and intersects the elliptic region $\vert z\vert < 2$ in an annulus.}
\end{figure}

\begin{figure}[ht]
\includegraphics[scale=.5]{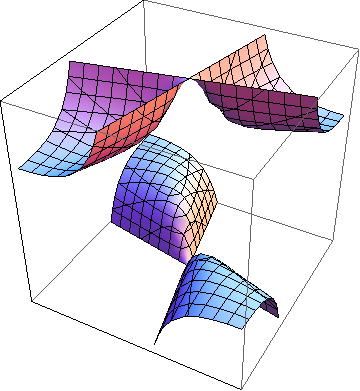}

\bigskip
\includegraphics[scale=.45]{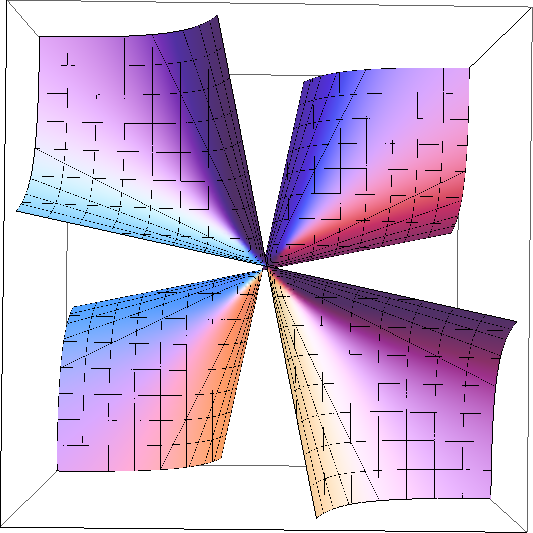}
\caption
[Four components of $\Fricke'(C_{1,1})\cap \kCk$ for $k > 2$]
{The four components of the generalized Fricke space 
$\Fricke'(C_{1,1})\cap \kCk$ for $k=23$.
These components are permuted by the group of sign-changes.
They are bounded by the singular hyperbolae on the planes
$z=\pm 5$.
In the second picture, the four components of the generalized Fricke space 
$\Fricke'(\Coo)\cap \kCk$ for $k=23$.
projected to the $xy$-plane, forming four triangles each
having a vertex at the origin.
}
\label{fig:GenFrickeSpaces}
\end{figure}
\clearpage
\input gmst.bbl

%
\end{document}

%% file: gmst.bbl
\providecommand{\bysame}{\leavevmode\hbox to3em{\hrulefill}\thinspace}
\providecommand{\MR}{\relax\ifhmode\unskip\space\fi MR }
\providecommand{\MRhref}[2]{%
  \href{http://www.ams.org/mathscinet-getitem?mr=#1}{#2}
}
\providecommand{\href}[2]{#2}

%% file: gmst_ArXiv.bbl
\begin{thebibliography}{10}

\bibitem{MR1643429}
B.~H. Bowditch, \emph{Markoff triples and quasi-{F}uchsian groups}, Proc.
  London Math. Soc. (3) \textbf{77} (1998), no.~3, 697--736. \MR{1643429
  (99f:57014)}

\bibitem{MR2553877}
Serge Cantat, \emph{Bers and {H}\'enon, {P}ainlev\'e and {S}chr\"odinger}, Duke
  Math. J. \textbf{149} (2009), no.~3, 411--460. \MR{2553877 (2011f:37077)}

\bibitem{MR2649343}
Serge Cantat and Frank Loray, \emph{Dynamics on character varieties and
  {M}algrange irreducibility of {P}ainlev\'e {VI} equation}, Ann. Inst. Fourier
  (Grenoble) \textbf{59} (2009), no.~7, 2927--2978. \MR{2649343 (2011c:37105)}

\bibitem{CDG}
Virginie Charette, Todd Drumm, and William~M. Goldman, \emph{Proper affine
  deformations of the One-Holed Torus}, Transformation Groups (to appear),
  arXiv:1501.04535.

\bibitem{MR3180618}
Virginie Charette, Todd~A. Drumm, and William~M. Goldman, \emph{Finite-sided
  deformation spaces of complete affine 3-manifolds}, J. Topol. \textbf{7}
  (2014), no.~1, 225--246. \MR{3180618}

\bibitem{MR1478672}
John~H. Conway, \emph{The sensual (quadratic) form}, Carus Mathematical
  Monographs, vol.~26, Mathematical Association of America, Washington, DC,
  1997, With the assistance of Francis Y. C. Fung. \MR{1478672 (98k:11035)}

\bibitem{MR0183872}
Robert Fricke and Felix Klein, \emph{Vorlesungen \"uber die {T}heorie der
  automorphen {F}unktionen. {B}and 1: {D}ie gruppentheoretischen {G}rundlagen.
  {B}and {II}: {D}ie funktionentheoretischen {A}usf\"uhrungen und die
  {A}ndwendungen}, Bibliotheca Mathematica Teubneriana, B\"ande 3, vol.~4,
  Johnson Reprint Corp., New York; B. G. Teubner Verlagsgesellschaft, Stuttg
  art, 1965. \MR{0183872 (32 \#1348)}

\bibitem{MR1491446}
William~M. Goldman, \emph{Ergodic theory on moduli spaces}, Ann. of Math. (2)
  \textbf{146} (1997), no.~3, 475--507. \MR{1491446 (99a:58024)}

\bibitem{MR2026539}
\bysame, \emph{The modular group action on real {${\rm SL}(2)$}-characters of a
  one-holed torus}, Geom. Topol. \textbf{7} (2003), 443--486. \MR{2026539
  (2004k:57001)}

\bibitem{MR2264541}
\bysame, \emph{Mapping class group dynamics on surface group representations},
  Problems on mapping class groups and related topics, Proc. Sympos. Pure
  Math., vol.~74, Amer. Math. Soc., Providence, RI, 2006, pp.~189--214.
  \MR{2264541 (2007h:57020)}

\bibitem{MR2346275}
\bysame, \emph{An ergodic action of the outer automorphism group of a free
  group}, Geom. Funct. Anal. \textbf{17} (2007), no.~3, 793--805. \MR{2346275
  (2008g:57001)}

\bibitem{MR2497777}
\bysame, \emph{Trace coordinates on {F}ricke spaces of some simple hyperbolic
  surfaces}, Handbook of {T}eichm\"uller theory. {V}ol. {II}, IRMA Lect. Math.
  Theor. Phys., vol.~13, Eur. Math. Soc., Z\"urich, 2009, pp.~611--684.
  \MR{2497777 (2010j:30093)}

\bibitem{GoldmanStantchev}
William~M. Goldman and Gueorgui~M. Stantchev, \emph{Dynamics of the
  automorphism group of the {${\rm GL}(2,{\mathbb R})$}-characters of a
  once-punctured torus}, ArXiv:math/0309072v1[math.DG], 2003.

\bibitem{MR2807844}
William~M. Goldman and Eugene~Z. Xia, \emph{Ergodicity of mapping class group
  actions on {${\rm SU}(2)$}-character varieties}, Geometry, rigidity, and
  group actions, Chicago Lectures in Math., Univ. Chicago Press, Chicago, IL,
  2011, pp.~591--608. \MR{2807844 (2012k:22028)}

\bibitem{HuTanZhang}
Hengnan Hu, Ser~Peow Tan, and Ying Zhang, \emph{Polynomial automorphisms of
  {$\mathbb{C}^n$} preserving the {M}arkoff-{H}urwitz polynomial},
  ArXiv:math/{1501.06955}v2[math.GT], 2015.

\bibitem{MR283209}
Donghi Lee and Makoto Sakuma, \emph{Simple loops on 2-bridge spheres in
  2-bridge link complements}, Electron. Res. Announc. Math. Sci. \textbf{18}
  (2011), 97--111.

\bibitem{MR3109863}
\bysame, \emph{A variation of {M}c{S}hane's identity for 2-bridge links},
  Geometry and Topology \textbf{17} (2013), no.~4, 2061--2101.

\bibitem{MR558891}
Wilhelm Magnus, \emph{Rings of {F}ricke characters and automorphism groups of
  free groups}, Math. Z. \textbf{170} (1980), no.~1, 91--103. \MR{558891
  (81a:20043)}

\bibitem{MR3420542}
Sara Maloni, Fr{\'e}d{\'e}ric Palesi, and Ser~Peow Tan, \emph{On the character
  variety of the four-holed sphere}, Groups Geom. Dyn. \textbf{9} (2015),
  no.~3, 737--782. \MR{3420542}

\bibitem{MR3038545}
Yair~N. Minsky, \emph{On dynamics of {$Out(F_n)$} on {$PSL_2({\mathbb C})$}
  characters}, Israel J. Math. \textbf{193} (2013), no.~1, 47--70. \MR{3038545}

\bibitem{MR0193188}
Calvin~C. Moore, \emph{Ergodicity of flows on homogeneous spaces}, Amer. J.
  Math. \textbf{88} (1966), 154--178. \MR{0193188 (33 \#1409)}

\bibitem{nielsen1917}
J.~Nielsen, \emph{Die isomorphismengruppe der allgemeinen unendlichen gruppe
  mit zwei erzeugenden}, Math. Ann. \textbf{78} (1917), 385--397.

\bibitem{MR2399656}
Paul Norbury, \emph{Lengths of geodesics on non-orientable hyperbolic
  surfaces}, Geom. Dedicata \textbf{134} (2008), 153--176. \MR{2399656
  (2009b:32021)}

\bibitem{STY}
Caroline Series, Ser~Peow Tan, and Yasushi Yamashita, \emph{The diagonal slice
  of schottky space}, math.GT.1409.6863v1 (2014).

\bibitem{MR2705402}
Gueorgui~M. Stantchev, \emph{Dynamics of the modular group acting on
  {GL}(2,{R})-characters of a once-punctured torus}, ProQuest LLC, Ann Arbor,
  MI, 2003, Thesis (Ph.D.)--University of Maryland, College Park. \MR{2705402}

\bibitem{MR2191691}
Ser~Peow Tan, Yan~Loi Wong, and Ying Zhang, \emph{The {${\rm SL}(2,\mathbb C)$}
  character variety of a one-holed torus}, Electron. Res. Announc. Amer. Math.
  Soc. \textbf{11} (2005), 103--110. \MR{2191691 (2006k:57054)}

\bibitem{MR2247658}
\bysame, \emph{Necessary and sufficient conditions for {M}c{S}hane's identity
  and variations}, Geom. Dedicata \textbf{119} (2006), 199--217. \MR{2247658
  (2007e:57017)}

\bibitem{MR2405161}
\bysame, \emph{End invariants for {${\rm SL}(2,\mathbb C)$} characters of the
  one-holed torus}, Amer. J. Math. \textbf{130} (2008), no.~2, 385--412.
  \MR{2405161 (2009h:57034)}

\bibitem{MR2370281}
\bysame, \emph{Generalized {M}arkoff maps and {M}c{S}hane's identity}, Adv.
  Math. \textbf{217} (2008), no.~2, 761--813. \MR{2370281 (2008k:57035)}

\bibitem{MR1508833}
H.~Vogt, \emph{Sur les invariants fondamentaux des \'equations
  diff\'erentielles lin\'eaires du second ordre}, Ann. Sci. \'Ecole Norm. Sup.
  (3) \textbf{6} (1889), 3--71. \MR{1508833}

\bibitem{MR776417}
Robert~J. Zimmer, \emph{Ergodic theory and semisimple groups}, Monographs in
  Mathematics, vol.~81, Birkh\"auser Verlag, Basel, 1984. \MR{776417
  (86j:22014)}

\end{thebibliography}
